\documentclass[11pt]{article}

\usepackage{amsmath,amssymb,amsfonts,amsthm,enumitem,pstricks,float,xy,mathrsfs,accents}
\xyoption{all}

\numberwithin{equation}{section} 
\newtheorem{thm}{Theorem}[section]
\newtheorem{prp}[thm]{Proposition} 
\newtheorem{lmm}[thm]{Lemma}  
\newtheorem{crl}[thm]{Corollary}
\newtheorem*{prp*}{Proposition}
\theoremstyle{definition}

\newtheorem{eg}[thm]{Example}
\newtheorem{rmk}[thm]{Remark}

\DeclareFontFamily{U}{mathx}{\hyphenchar\font45}
\DeclareFontShape{U}{mathx}{m}{n}{
      <5> <6> <7> <8> <9> <10>
      <10.95> <12> 
      mathx10
      }{}
\DeclareSymbolFont{mathx}{U}{mathx}{m}{n}
\DeclareFontSubstitution{U}{mathx}{m}{n}
\DeclareMathAccent{\widecheck}{0}{mathx}{"71}

\def\BE#1{\begin{equation}\label{#1}}
\def\EE{\end{equation}}
\def\eref#1{(\ref{#1})}

\def\lra{\longrightarrow}

\def\Lra{\Longrightarrow}
\def\Llra{\Longleftrightarrow}
\def\xra#1{\xrightarrow{\hspace*{#1cm}}}

\def\ov#1{\overline#1}
\def\ti#1{\tilde{#1}}
\def\wc#1{\widecheck{#1}}
\def\wt#1{\widetilde{#1}}
\def\wh#1{\widehat{#1}}
\def\tn#1{\textnormal{#1}}

\def\ori#1{\accentset{\circ}{#1}}
\def\obu#1{\accentset{\bullet}{#1}}
\def\sf#1{\textsf{#1}}

\def\al{\alpha}
\def\be{\beta}
\def\de{\delta}
\def\ep{\epsilon}
\def\ga{\gamma}
\def\io{\iota}

\def\la{\lambda}

\def\om{\omega}
\def\si{\sigma}

\def\vph{\varphi}

\def\vp{\varpi}

\def\De{\Delta}

\def\Th{\Theta}

\def\C{\mathbb C}

\def\cK{\mathcal K}

\def\fM{\mathfrak M}

\def\cN{\mathcal N}

\def\P{\mathbb P}

\def\cR{\mathcal R}
\def\Q{\mathbb Q}
\def\R{\mathbb R}
\def\fR{\mathfrak R}
\def\cS{\mathcal S}

\def\T{\mathbb T}

\def\fX{\mathfrak X}

\def\Z{\mathbb Z}

\def\bs{\mathbf s}

\def\bff{\mathbf f}

\def\fI{\mathfrak i}

\def\fc{\mathfrak c}

\def\0{\mathbf 0}

\def\Si{\Sigma}

\def\Bl{\textnormal{Bl}}

\def\Deck{\tn{Deck}}

\def\ev{\textnormal{ev}}

\def\Flux{\tn{Flux}}

\def\Hur{\tn{Hur}}
\def\id{\textnormal{id}}
\def\Im{\textnormal{Im}}

\def\ord{\textnormal{ord}}

\def\PD{\textnormal{PD}}

\def\rk{\tn{rk}}

\def\vir{\textnormal{vir}}

\def\i{\infty}

\def\prt{\partial}

\def\eset{\emptyset}

\begin{document}

\title{On the Rim Tori Refinement of\\
Relative Gromov-Witten Invariants}
\author{Mohammad F.~Tehrani and Aleksey Zinger\thanks{Partially 
supported by NSF grant 0846978}}
\date{\small\today}
\maketitle

\begin{abstract}
\noindent
We construct Ionel-Parker's proposed refinement of 
the standard relative Gromov-Witten invariants in terms of abelian covers
of the symplectic divisor and discuss in what sense it gives rise to invariants.
We use it to obtain some vanishing results for the standard relative Gromov-Witten invariants.
In a separate paper, we describe to what extent this refinement sharpens
the usual symplectic sum formula and give further  qualitative applications.
\end{abstract}

\tableofcontents

\section{Introduction}
\label{intro_sec}

\noindent
Gromov-Witten invariants of symplectic manifolds, which include nonsingular projective varieties, 
are certain counts of pseudo-holomorphic curves that play 
prominent roles in symplectic topology, algebraic geometry, and string theory.
The decomposition formulas, known as symplectic sum formulas 
in symplectic topology and degeneration formulas in algebraic geometry, 
are one of the main tools used to compute Gromov-Witten invariants;
they relate Gromov-Witten invariants of one symplectic manifold 
to Gromov-Witten invariants of two simpler symplectic manifolds.
Unfortunately, the formulas of \cite{Jun2,LR} do not completely determine 
the former in terms of the latter
in many cases because of the so-called \sf{vanishing cycles}:
second homology classes in the first manifold which vanish when projected 
to the union of the other two manifolds; see~\eref{cRXYVdfn_e}.
A~refinement to the usual relative Gromov-Witten invariants of \cite{Jun1,LR} 
is sketched in~\cite{IPrel};
the aim of this refinement is to resolve the unfortunate deficiency of 
the  formulas of \cite{Jun2,LR} in~\cite{IPsum}.
In this paper, we formally construct the refinement of~\cite{IPrel},
discuss the invariance and computability aspects of the resulting curve counts,
and 
obtain some vanishing results for the usual relative GW-invariants.
In the sequel~\cite{GWsumIP}, we describe the applicability of this refinement
to computing the Gromov-Witten invariants of symplectic sums and 
obtain  further  qualitative applications.

\subsection{Relative GW-invariants}
\label{RelGW_subs0}

\noindent
Let $(X,\om)$ be a compact symplectic manifold and $J$ be an $\om$-tame almost complex structure on~$X$.
For $g,k\!\in\!\Z^{\ge0}$ and $A\!\in\!H_2(X;\Z)$, we denote by $\ov\fM_{g,k}(X,A)$ 
the moduli space of stable $J$-holomorphic $k$-marked degree~$A$ maps from connected nodal 
curves of genus~$g$.
By \cite{LT,FO,BF}, this moduli space carries a virtual class, which is independent of~$J$
and of representative~$\om$ in a deformation equivalence class 
of symplectic forms on~$X$.
If $V\!\subset\!X$ is a compact symplectic divisor (symplectic submanifold of real codimension~2), 
$\ell\!\in\!\Z^{\ge0}$, $\bs\!\equiv\!(s_1,\ldots,s_{\ell})$
is an $\ell$-tuple of positive integers such~that 
\BE{bsumcond_e} s_1+\ldots+s_{\ell}=A\cdot V,\EE
and $J$ restricts to an almost complex structure on $V$, let 
$\ov\fM_{g,k;\bs}^V(X,A)$  denote the moduli space of 
stable $J$-holomorphic $(k\!+\!\ell)$-marked maps from connected nodal curves of genus~$g$ 
that have contact with~$V$ at the last $\ell$ marked points of orders $s_1,\ldots,s_{\ell}$.
According to \cite{LR,Jun1}, this moduli space carries a virtual class, 
which is independent of~$J$ and of representative~$\om$ in a deformation equivalence class 
of symplectic forms on~$(X,V)$.\\

\noindent
There are natural \sf{evaluation morphisms}
\begin{alignat}{1}\label{evdfn_e1}
\ev_X\!\equiv\!\ev_1\!\times\!\ldots\!\times\!\ev_k\!: 
\ov\fM_{g,k}(X,A),\ov\fM_{g,k;\bs}^V(X,A)&\lra X^k,\\
\label{evdfn_e2}
\ev_X^V\!\equiv\!\ev_{k+1}\!\times\!\ldots\!\times\!\ev_{k+\ell}\!:
\ov\fM_{g,k;\bs}^V(X,A)&\lra V_{\bs}\equiv V^{\ell},
\end{alignat}
sending each stable map to its values at the marked points.
The (\textsf{absolute}) \textsf{GW-invariants of~$(X,\om)$}
are obtained by pulling back elements of $H^*(X^k;\Q)$ by the morphism~\eref{evdfn_e1}
and integrating them and other natural classes on  $\ov\fM_{g,k}(X,A)$
against the virtual class of $\ov\fM_{g,k}(X,A)$.
The (\textsf{relative}) \textsf{GW-invariants of~$(X,V,\om)$}
are obtained by pulling back elements of $H^*(X^k;\Q)$ and $H^*(V_{\bs};\Q)$
by the morphisms~\eref{evdfn_e1} and~\eref{evdfn_e2}, and integrating them and 
other natural classes on  $\ov\fM_{g,k;\bs}^V(X,A)$
against the virtual class of $\ov\fM_{g,k;\bs}^V(X,A)$.\\

\noindent
As emphasized in \cite[Section~5]{IPrel}, two preimages of the same point in~$V_{\bs}$
under~\eref{evdfn_e2} determine an element~of 
\BE{cRXVdfn_e0}\cR_X^V\equiv \ker\big\{\io_{X-V*}^X\!:\,H_2(X\!-\!V;\Z)\lra H_2(X;\Z)\big\},\EE
where $\io_{X-V}^X\!:X\!-\!V\!\lra\!X$ is the inclusion; see Section~\ref{prelimcut_subs}.
The elements of~$\cR_X^V$, called \sf{rim tori} in~\cite{IPrel},
can be represented by circle bundles over loops~$\ga$ in~$V$; see Section~\ref{cuthom_subs}.
By standard topological considerations,
\BE{cRXVvsH1V_e}\cR_X^V\approx H_1(V;\Z)_X\equiv 
\frac{H_1(V;\Z)}{H_X^V},
\qquad\hbox{where}\quad H_X^V\equiv \big\{A\!\cap\!V\!:\,A\!\in\!H_3(X;\Z)\big\} \,;\EE
see Corollary~\ref{rimtori_crl}.\\

\noindent
The main claim of \cite[Section~5]{IPrel} is that the above observations can be used 
to lift~\eref{evdfn_e2} over some regular (Galois), possibly disconnected (unramified) covering
\BE{IPcov_e}\pi_{X;\bs}^V\!: \wh{V}_{X;\bs}\lra V_{\bs},\EE
though the topology of this cover is not specified and the group of its deck transformation
is described incorrectly as~$\cR_X^V$ in~\cite{IPrel}.
Since
\BE{evfactor_e}\ev_X^V\!=\!\pi_{X;\bs}^V\!\circ\!\wt\ev_X^V\!:\,\ov\fM_{g,k;\bs}^V(X,A)\lra V_{\bs}\EE
for some morphism
\BE{evXVlift_e} \wt\ev_X^V\!:\ov\fM_{g,k;\bs}^V(X,A)\lra \wh{V}_{X;\bs}\,,\EE
the numbers obtained by pulling back elements of $H^*(\wh{V}_{X;\bs};\Q)$ by~\eref{evXVlift_e},
instead of elements of $H^*(V_{\bs};\Q)$ by~\eref{evdfn_e2}, and integrating them and 
other natural classes on  $\ov\fM_{g,k;\bs}^V(X,A)$ against the virtual class of 
$\ov\fM_{g,k;\bs}^V(X,A)$ refine the usual GW-invariants of~$(X,V,\om)$.
We will call these numbers the {\sf{IP-counts for~$(X,V,\om)$}.
\\

\noindent
The covering~\eref{IPcov_e}, which is completely determined by $(X,V)$ and~$\bs$,
is defined in Section~\ref{RTCov_subs} based on the sketch in \cite[Section~5]{IPrel}
and after some preparation in Section~\ref{AbCovNotat_subs}.
For example,
$$\wh{V}_{X;()}=\cR_X^V\!\times\!V_{()}=\cR_X^V\,,$$
where $()\!\in\!\Z_+^0$ is the empty vector.
If $V$ is connected, then
$$ \wh{V}_{X;(1)}=\wh{V}_X$$
is the abelian covering corresponding to the quotient $H_1(V;\Z)_X$ of $H_1(V;\Z)$,
i.e.~to the preimage of~$H_X^V$ under the Hurewicz homomorphism
$\pi_1(X)\!\lra\!H_1(X;\Z)$.
The group of deck transformations of the covering~\eref{IPcov_e} is given~by
\BE{DeckGr_e}\Deck\big(\pi_{X;\bs}^V\big) 
=\frac{\cR_X^V}{\cR_{X;\bs}'^{\,V}}\!\times\! \cR_{X;\bs}'^{\,V}\EE
for a certain submodule $\cR_{X;\bs}'^{\,V}$ of $\cR_{X}^V$.
For example,
$$\cR_{X;\bs}'^{\,V}=\begin{cases}\{0\},&\hbox{if}~\ell\!=\!0;\\
\gcd(\bs)\cR_X^V,&\hbox{if}~|\pi_0(V)|\!=\!1.
\end{cases}$$
In general, the deck group~\eref{DeckGr_e} is different from $\cR_X^V$ 
(contrary to an explicit statement in \cite[Section~5]{IPrel}).
By \cite[Assertion~6]{Mi68} in the case $H_1(V;\Z)_X$ is of rank~1 and its extension~\cite{Mi14}, 
$H_*(\wh{V}_{X;\bs};\Q)$ is not finitely generated if $V$ is connected, 
$\chi(V)\!\neq\!0$, and $H_1(V;\Z)_X$  is not a torsion group 
(so that the covering~\eref{IPcov_e} is infinite).\\

\noindent
By standard covering spaces considerations, the total relative evaluation map~\eref{evdfn_e2}
lifts over the covering~\eref{IPcov_e}; see Lemma~\ref{RimToriAct_lmm}.
The lift~\eref{evXVlift_e} of~\eref{evdfn_e2} is not unique and 
involves choices of base points in various spaces.
However, these choices can be made in a systematic manner and the lift~\eref{evXVlift_e}
extends over the space of stable smooth maps (and $L^p_1$-maps with $p\!>\!2$); 
see Theorem~\ref{RimToriAct_thm} and Remark~\ref{cHlift_rmk}. 
This ensures that the IP-counts for~$(X,V,\om)$ are independent of~$J$
and of representative~$\om$ in a deformation equivalence class of symplectic forms on~$(X,V)$.\\

\noindent
If $V\!=\!V_1\!\sqcup\!V_2$ for some symplectic divisors $V_1,V_2\!\subset\!X$,
it is natural to consider the moduli spaces
$$\ov\fM_{g,k;\bs_1\bs_2}^{V_1,V_2}(X,A) \subset\bigcup_{\bs}\ov\fM_{g,k;\bs}^V(X,A) $$
that keep track of contacts with $V_1$ and~$V_2$ separately.
There is then a forgetful morphism
$$f\!: \ov\fM_{g,k;\bs_1\bs_2}^{V_1,V_2}(X,A)\lra \ov\fM_{g,k;\bs_1}^{V_1}(X,A)$$
which drops the relative contacts with~$V_2$.
The total relative evaluation morphisms~\eref{evdfn_e2} corresponding to the two moduli
spaces above are compatible with~$f$
and the projection 
$$\pi_1\!:(V_1)_{\bs_1}\!\times\!(V_2)_{\bs_2}\lra(V_1)_{\bs_1}\,,$$
i.e.~the part of the diagram in Figure~\ref{IPev_fig} involving only the corners commutes.
The projection~$\pi_1$ lifts to a smooth map~$\wt\pi_1$ between the total spaces of the coverings
$$ \pi_{X;\bs_1\bs_2}^{V_1,V_2}\!:\wh{V}_{X;\bs_1\bs_2}\lra (V_1)_{\bs_1}\!\times\!(V_2)_{\bs_2}
\quad\hbox{and}\quad
\pi_{X;\bs_1}^{V_1}\!: (\wh{V}_1)_{X;\bs_1}\lra (V_1)_{\bs_1},$$
so that the right square in  Figure~\ref{IPev_fig} commutes.
The relevant base points determining the lifts~\eref{evXVlift_e} can be chosen so
that the resulting lifted evaluation morphisms are compatible with~$f$ and~$\wt\pi_1$,
i.e.~the  left square in  Figure~\ref{IPev_fig} commutes.

\begin{figure}
\begin{gather*}
\xymatrix{\ov\fM_{g,k;\bs_1\bs_2}^{V_1,V_2}(X,A) 
\ar@/^2pc/[rrrr]^{\ev_{X;\bs_1\bs_2}^{V_1,V_2}}
\ar[rr]^>>>>>>>>>>>{\wt\ev_{X;\bs_1\bs_2}^{V_1,V_2}} \ar[d]_f  &&  
\wh{V}_{X;\bs_1\bs_2}\ar[rr]^>>>>>>>>>>>{\pi_{X;\bs_1\bs_2}^{V_1,V_2}} \ar[d]^{\wt\pi_1}&&  
(V_1)_{\bs_1}\!\times\!(V_2)_{\bs_2} \ar[d]^{\pi_1}\\
\ov\fM_{g,k;\bs_1}^{V_1}(X,A) \ar@/_2pc/[rrrr]^{\ev_{X;\bs_1}^{V_1}}
 \ar[rr]^>>>>>>>>>>>{\wt\ev_{X;\bs_1}^{V_1}} && 
(\wh{V}_1)_{X;\bs_1} \ar[rr]^{\pi_{X;\bs_1}^{V_1}} && (V_1)_{\bs_1}}
\end{gather*}
\caption{The potential compatibility of the lifted relative evaluation morphisms~\eref{evXVlift_e}.}
\label{IPev_fig}
\end{figure}

\subsection{Qualitative applications}
\label{appl_subs}

\noindent
The set of all IP-counts for~$(X,V)$ for elements in an orbit for 
the $\Deck(\pi_{X;\bs}^V)$-action on $H^*(\wh{V}_{X;\bs};\Q)$ 
depends only on $(X,V,\om)$, the cohomology class on~$X^k$,
and intrinsic classes on $\ov\fM_{g,k;\bs}^V(X,A)$, such as descendants.
However, the individual IP-counts also depend on the precise choice of the lift~\eref{evXVlift_e}. 
If $\ell\!=\!0$, the cover~\eref{IPcov_e} is trivial and these numbers can be indexed
by the elements of~$\cR_X^V$.
This is generally not the case if $\ell\!\neq\!0$, 
including in the last claim of \cite[Lemma~14.5]{IPsum} and in \cite[Lemma~14.8]{IPsum};
see \cite[Remarks~6.5,6.8]{GWsumIP}.
Because the IP-counts generally depend on the choice of the lift~\eref{evXVlift_e}
and the homology of $\wh{V}_{X;\bs}$ is usually very complicated,
they appear to be of little quantitative use outside of very rare cases.
On the other hand, they can sometimes provide qualitative information, 
as indicated by Theorem~\ref{RelGW_thm} below.

\begin{thm}\label{RelGW_thm}
Let $(X,\om)$ be a compact symplectic manifold and $V\!\subset\!X$ be a compact 
symplectic divisor which admits a fibration $q\!:V\!\lra\!(S^1)^m$ with 
a connected fiber~$F$ such that $H_1(F;\Q)\!=\!\{0\}$. 
If $\ell\!\in\!\Z^+$ and $\bs\!\in\!\Z_+^{\ell}$, then
$$\ev_X^{V*}\al\cap \big[\ov\fM_{g,k;\bs}^V(X,A)\big]^{\vir}=0
\qquad\forall~\al\in H^r(V_{\bs};\Q),~r\!>\!(\dim_{\R}V)\ell\!-\!\rk_{\Z}H_1(V;\Z)_X\,,$$
i.e.~all relative GW-invariants of $(X,V,\om)$ with non-trivial contacts with~$V$ and
relative insertions $\al$ as above vanish.
\end{thm}

\noindent
The proof of this theorem readily extends to disconnected divisors~$V$, 
after replacing the rank of $H_1(V;\Z)_X$ with the rank of
the appropriate submodule of~$H_1(V;\Z)_X$, depending on~$\bs$;
see Remark~\ref{RelGW_rmk}.\\

\noindent
If $V\!=\!(S^1)^{2n-2}$ and $H^3(X;\Z)\!=\!0$, then
$$H_1(V;\Z)_X= H_1(V;\Z)\approx\Z^{2n-2}\,.$$
In this case, the relative GW-invariants of $(X,V,\om)$ with non-trivial contacts with~$V$ vanish
whenever the degree of the relative insertion exceeds $(2n\!-\!2)(\ell\!-\!1)$.
In particular, the only relative GW-invariants of $(X,V,\om)$ with a single 
(but arbitrary order) contact that may be nonzero are those that involve no relative
constraint (insertion $1\!\in\!H^*(V;\Q)$).
This particular observation is immediate from~\eref{evfactor_e}, because 
$\wh{V}_{X;(s)}\!\approx\!\C^{n-1}$ and thus $\wh{V}_{X;(s)}$  has no positive-degree cohomology
for any $s\!\in\!\Z^+$.
We use this fact  in \cite[Section~6.3]{GWsumIP} to streamline the proof of \cite[(15.4)]{IPsum},
after correcting its statement;
this formula computes some GW-invariants of the blowup~$\wh\P^2_9$ of~$\P^2$ at 9~points.\\

\noindent
If $V$ is any topological space, a loop of homeomorphisms 
$$\Psi_t\!: V\lra V, \qquad t\in[0,1],~~\Psi_0=\Psi_1,$$
and a point $x\!\in\!V$ determines a loop $t\!\lra\!\Psi_t(x)$ in~$V$
and thus an element of $H_1(V;\Z)$.
The latter is independent of the choice of $x\!\in\!V$.
We denote the set of all elements of $H_1(V;\Z)$ obtained in this way by~$\Flux(V)$. 
It is a subgroup of $H_1(V;\Z)$, usually called \sf{the flux subgroup} 
(or \sf{group}).
If in addition $V\!\subset\!X$ is a compact oriented submanifold of a compact oriented manifold
and $H_1(V;\Z)_X$ is as in~\eref{cRXVvsH1V_e}, let
$$\Flux(V)_X\subset H_1(V;\Z)_X$$
denote the image of $\Flux(V)$ under the quotient projection.

\begin{thm}\label{EqRelGWs_thm}
Let $(X,\om)$ be a compact symplectic manifold and
$V\!\subset\!X$ be a compact connected symplectic divisor such~that 
\BE{EqRelGWs_e1} \Flux(V)_X=H_1(V;\Z)_X  \,.\EE
Suppose $A\!\in\!H_2(X;\Z)$ and $\bs\!\in\!\Z_+^{\ell}$ with $\ell\!>\!0$.
If $\gcd(\bs)$ and $|\cR_X^V|$ are relatively prime, then 
the IP-counts for $(X,V,\om)$ in degree~$A$ with relative contacts~$\bs$
are independent of the choice of the lift~\eref{evXVlift_e} and are thus determined by $(X,V,\om)$.
If in addition
\BE{EqRelGWs_e2} \rk_{\Z}H_1(V;\Z)_X\in\{0,1\},\EE
then these IP-counts are 
the same as the corresponding  GW-invariants of~$(X,V,\om)$.
\end{thm}

\noindent
If the rim tori module $\cR_X^V\!\approx\!H_1(V;\Z)_X$ is infinite, 
we call $\gcd(\bs)$ and $|\cR_X^V|$ \sf{relatively prime} if $\gcd(\bs)\!=\!1$.
Let 
$$H^*\big(\wh{V}_{X;\bs};\Q\big)^{\pi_{X;\bs}^V}
=\big\{\eta\!\in\!H^*\big(\wh{V}_{X;\bs};\Q\big)\!:\,g^*\eta\!=\!\eta~
\forall~g\!\in\!\Deck(\pi_{X;\bs}^V)\big\}.$$
In general, the set of GW-invariants of $(X,V,\om)$ in degree~$A$ with relative contacts~$\bs$
with all possible cohomology insertions can be identified with 
the subset of IP-counts with the cohomology insertions~in 
\BE{cohcover_e}
\pi_{X;\bs}^{V\,*}H^*\big(V_{\bs};\Q\big) \subset 
H^*\big(\wh{V}_{X;\bs};\Q\big)^{\pi_{X;\bs}^V} \subset
H^*\big(\wh{V}_{X;\bs};\Q\big);\EE
the GW-invariants and IP-counts with such insertions are the same by~\eref{evfactor_e}.
The substance of the first conclusion of Theorem~\ref{EqRelGWs_thm} is that
the second inclusion in~\eref{cohcover_e} is an equality, as any two lifts~\eref{evXVlift_e}
are related by an element of $\Deck(\pi_{X;\bs}^V)$.
The substance of the second conclusion of Theorem~\ref{EqRelGWs_thm} is that both inclusions
in~\eref{cohcover_e} are equalities.
The cohomology homomorphism~$\pi_{X;\bs}^{V\,*}$ may still not be injective;
the GW-invariants of $(X,V,\om)$ with insertions in its kernel vanish and thus can be disregarded.\\

\noindent
Theorem~\ref{EqRelGWs_thm} is established at the end of Section~\ref{RTCov_subs}.
We show that~\eref{EqRelGWs_e1} implies
that the second inclusion in~\eref{cohcover_e} is in fact an equality. 
The first inclusion in~\eref{cohcover_e} is an equality if~\eref{EqRelGWs_e2} holds; 
see Corollary~\ref{CohSurj_crl}.
Both inclusions in~\eref{cohcover_e} are equalities if $V\!=\!\T^{2n-2}$  
and the cover $\wh{V}_{X;\bs}$ is connected, 
as can be seen by considering all connected covers of~tori.
From this observation, we obtain the following conclusion concerning
IP-counts relative to tori. 

\begin{prp}\label{EqRelGWs_prp}
Suppose $(X,\om)$ is a compact symplectic manifold and
$V\!\subset\!X$ is a symplectic divisor
such that $V\!\approx\!\T^{2n-2}$.
Let $A\!\in\!H_2(X;\Z)$ and $\bs\!\in\!\Z_+^{\ell}$ with $\ell\!>\!0$.
If $\gcd(\bs)$ and $|\cR_X^V|$ are relatively prime, then the IP-counts
for $(X,V,\om)$ in degree~$A$ with relative contacts~$\bs$ are 
the same as the corresponding  GW-invariants of~$(X,V,\om)$.
\end{prp}

\begin{rmk}\label{EqRelGWs_rmk0}
As pointed out by B.~Wieland on {\it MathOverflow}, 
the first inclusion in~\eref{cohcover_e} can fail to be an equality as soon as $H_1(V;\Z)_X$ 
is at least~$\Z^2$.
It can fail to be an equality even if $\pi_1(V)$ is free abelian.
\end{rmk}

\noindent
Theorem~\ref{EqRelGWs_thm} and Proposition~\ref{EqRelGWs_prp} do not provide
any new information about the GW-invariants of~$(X,V,\om)$.
However, they can be useful in refining the usual symplectic sum formula in
a narrow set of cases. 
This formula expresses certain sums of GW-invariants of one symplectic manifold
in terms of relative GW-invariants of simpler manifolds;
Theorem~\ref{EqRelGWs_thm} and Proposition~\ref{EqRelGWs_prp} imply that 
all the summands in each given sum are the same in these cases.\\

\noindent
Generalizations of Theorem~\ref{EqRelGWs_thm} and of Proposition~\ref{EqRelGWs_prp} 
to a disconnected divisor~$V$ are described in Remark~\ref{EqRelGWs_rmk}.

\subsection{Outline of the paper}
\label{outline_subs}

\noindent
The relevant setting for relative GW-invariants and the symplectic sum formula
is the codimension $\fc\!=\!2$ case of the topological setup of Section~\ref{GenTopolSetup_sec}.
As restricting to the $\fc\!=\!2$ case carries no benefit, we consider the general case
to the extent possible.
Sections~\ref{cuthom_subs}, \ref{gluehom_subs}, and~\ref{gluecoh_subs} 
are concerned with changes in the topology of manifolds under surgery
that are directly relevant in the symplectic sum context.
The rim tori module~$\cR_X^V$ and the vanishing cycles module~$\cR_{X,Y}^V$
are described explicitly in Sections~\ref{rimtori_subs} and~\ref{vancyc_subs}, respectively, 
with the aim of easily computing them in many situations.
Section~\ref{gluerimtori_subs} compares the rim tori modules before and after surgery.
The notation for the abelian covers relevant for our purposes is introduced 
in Section~\ref{AbCovNotat_subs};
some of their properties, focusing on finite generation of the (co)homology, 
are discussed in Section~\ref{AbCovProper_subs}.
In Section~\ref{RTCov_subs}, 
we define the intended rim tori coverings~\eref{IPcov_e} of \cite[Section~5]{IPrel} 
as special cases of the abelian covers of  Section~\ref{AbCovNotat_subs},
show that the evaluation morphisms~\eref{evdfn_e2} 
lift to these covers, and establish Theorem~\ref{EqRelGWs_thm}.
In Section~\ref{RTCov_subs2}, we show that these lifts can be chosen systematically,
in respect to the intended applications in~\cite{IPsum} and the diagram in Figure~\ref{IPev_fig};
see Theorem~\ref{RimToriAct_thm}.
Theorem~\ref{RelGW_thm} is established in Section~\ref{RelGWThm_subs}.\\

\noindent
The main purpose of this paper is to investigate the topological aspects of the rim tori
refinement to the standard relative GW-invariants 
in preparation for considering its applicability in the symplectic sum context in~\cite{GWsumIP}. 
We pre-suppose a suitable analytic framework for the construction of relative GW-invariants and
describe the necessary steps to implement 
the idea of \cite[Section~5]{IPrel} as an enhancement on existing constructions.
We deduce some qualitative applications arising from this refinement and 
discuss its usability for quantitative purposes.
A significant number of examples are included for illustrative purposes.\\

\noindent
The authors would like to thank   E.~Ionel, D.~McDuff, M.~McLean, J.~Milnor, D.~Ruberman,
J.~Starr, M.~Wendt, and B.~Wieland for enlightening discussions.

\section{General topological context}
\label{GenTopolSetup_sec}

\noindent
The symplectic sum construction is a surgery operation that cuts out tubular neighborhoods of 
a common submanifold from two manifolds and glues the remainders along the boundaries of 
the two tubular neighborhoods.
Below we discuss central topological aspects of this construction from a more general perspective. 
Throughout this paper, by a \sf{manifold} we will mean a smooth manifold.

\subsection{Complement of submanifold}
\label{prelimcut_subs}

\noindent
Let $X$ be an oriented manifold, $V\!\subset\!X$ be a compact oriented submanifold of codimension~$\fc$,
$S_XV\!\subset\!\cN_XV$ be the sphere subbundle of the normal bundle of~$V$ in~$X$, 
and
\BE{cRXVdfn_e}\cR_X^V\equiv \ker\big\{\io^X_{X-V*}\!:
H_{\fc}(X\!-\!V;\Z)\!\lra\!H_{\fc}(X;\Z)\big\}.\EE
By Lemma~\ref{rimtori_lmm}, each element of $\cR_X^V$ can be represented by a cycle of 
the form $\io_{S_XV}^{X-V}(S_XV|_{\ga})$ for some loop $\ga\!\subset\!V$;
see the end of Section~\ref{cuthom_subs}.
In the $\fc\!=\!2$ case, i.e.~as in~\eref{cRXVdfn_e0}, these cycles are called \sf{rim tori}
in~\cite{IPrel,IPsum}.\\

\noindent
Suppose in addition that $f\!:Z\!\lra\!X$  is an $\fc$-pseudocycle,
as in \cite[Section~1.1]{Z}, and  $x\!\in\!f^{-1}(V)$ is an isolated point.
We define \sf{the order of contact of~$f$ with~$V$ at~$x$}, $\ord_x^Vf\!\in\!\Z$,
as follows.
On a small neighborhood of~$x$, $f$ can be homotoped without changing its intersection with $V$
so that it takes 
a small sphere $S_Zx$ in $T_xZ$ to a small sphere $S_XV|_{f(x)}\subset\cN_VX$;
the number $\ord_x^Vf$ is the degree of this map.
This definition agrees with the definition used in the construction
of $\ov\fM_{g,k;\bs}^V(X,A)$.\\

\noindent
We now combine two pseudocycles with the same contacts with~$V$ 
into a pseudocycle to~$X\!-\!V$.
Suppose
$$f\!:(Z,x_1,\ldots,x_{\ell})\lra(X,V)\qquad\hbox{and}\qquad 
f'\!:(Z',x_1',\ldots,x_{\ell}')\lra (X,V)$$ 
are two $\fc$-pseudocycles such that 
\begin{equation*}\begin{split}
f^{-1}(V)=\{x_1,\ldots,x_{\ell}\},&\qquad f'^{-1}(V)=\{x_1',\ldots,x_{\ell}'\},\\
f(x_i)=f'(x_i'),\quad
\ord_{x_i}^Vf&=\ord_{x_i'}^Vf'\qquad\forall\,i\!=\!1,2,\ldots,\ell.
\end{split}\end{equation*}
We can then obtain a smooth map $f\!\#\!(-f')\!:Z\!\#Z'\!\lra X\!-\!V$ by
\begin{enumerate}[label=$\bullet$,leftmargin=*]
\item removing small balls $B_{x_i}$ and $B_{x_i'}$ around each of the points $x_i$ and $x_i'$
to form manifolds with boundary $\hat{Z}$ and~$\hat{Z}'$,
\item forming a smooth oriented manifold $Z\!\#(-Z')$ by
identifying the $i$-th boundary components of $\hat{Z}$ and~$\hat{Z}'$ by 
an orientation-preserving diffeomorphism $\vph_i\!:(\prt\hat{Z})_i\!\lra\!(\prt\hat{Z}')_i$
for each~$i$,
\item homotoping $f$ and $f'$ on small neighborhoods of $\prt B_{x_i}$ and $\prt B_{x_i'}$ 
within a small ball around $f(x_i)\!=\!f'(x_i')$ in~$X$ so that
$f\!=\!f'\!\circ\!\vph_i$ for all~$i$.
\end{enumerate}
The last condition is achievable because the degrees of
$f,f'\!\circ\!\vph_i\!: S^{\fc-1}\lra S^{\fc-1}$
are the same and the degree homomorphism $\pi_{\fc-1}(S^{\fc-1})\!\lra\!\Z$ is an isomorphism
if $\fc\!\ge\!2$.\\

\noindent
The above construction~of $f\!\#\!(-f')$
depends only on $f$, $f'$, and choices of degree~1 maps from 
$\ell$ disjoint copies of $[0,1]\!\times\!S^{\fc-1}$ to $[0,1]\!\times\!S^{\fc-1}$.
Thus,  $f$ and $f'$ completely determine the homology class of $f\!\#\!(-f')$.
If in addition $[f]\!=\![f']$ in $H_{\fc}(X;\Z)$, then $[f\!\#\!(-f')]\!\in\!\cR_X^V$.

\subsection{Splice of two manifolds}
\label{prelimglue_subs}
 
\noindent
If $X$ and $Y$ are manifolds, $V\!\subset\!X,Y$ is a closed submanifold,
and $\vph\!:S_XV\!\lra\!S_YV$ is a diffeomorphism commuting with the projections to~$V$, 
let $X\!\#_{\vph}\!Y$ be the  manifold
obtained by gluing the complements of tubular neighborhoods of $V$ in $X$ and~$Y$
by~$\vph$ along their common boundary.
If $X$ and $Y$ are oriented and $\vph$ is orientation-reversing, then
$X\!\#_{\vph}\!Y$ is oriented as~well.\\

\noindent
We denote~by 
$$q_{\vph}\!:X\!\#_{\vph}\!Y\lra X\!\cup_V\!Y$$
a continuous map which restricts to the identity outside of a tubular neighborhood
of $S_XV\!=_{\vph}\!S_YV$, is a diffeomorphism on the complement of $q_{\vph}^{-1}(V)$,
and restricts to the bundle projection $S_XV\!\lra\!V$.
We will call such a map~$q_{\vph}$ a \textsf{collapsing map}.
If $\fc$ is the codimension of $V$ in $X$ and $Y$, let
\BE{cRXYVdfn_e}
\cR_{X,Y}^V\equiv \ker\big\{q_{\vph_*}\!:\,H_{\fc}(X\!\#_{\vph}\!Y;\Z)\lra 
H_{\fc}(X\!\cup_V\!Y;\Z)\big\}.\EE
By the $m\!=\!\fc$ case of Lemma~\ref{cRXYV_lmm}, this collection of \sf{vanishing cycles}
is the span of the images of $\cR_X^V$  and $\cR_Y^V$
under the homology homomorphisms induced by the inclusions
$$\io^{X\#_{\vph}Y}_{X-V}\!:\,X\!-\!V\lra X\!\#_{\vph}\!Y \qquad\hbox{and}\qquad
\io^{X\#_{\vph}Y}_{Y-V}\!:\,Y\!-\!V\lra X\!\#_{\vph}\!Y, $$
respectively.\\

\noindent
Suppose in addition $X$, $Y$, and $V$ are compact and oriented
and $V_1,\ldots,V_N$ are the topological components of~$V$. 
Let
\BE{HfcXTdfn_e}\begin{split}
H_{\fc}(X;\Z)\!\times_V\!H_{\fc}(Y;\Z)&=
\big\{(A_X,A_Y)\!\in\!H_{\fc}(X;\Z)\!\times\!H_{\fc}(Y;\Z)\!:\\
&\hspace{1in}
A_X\!\cdot_X\!V_r=A_Y\!\cdot_Y\!V_r~~\forall\,r\!=\!1,\ldots,N\big\},
\end{split}\EE
where $\cdot_X$ and $\cdot_Y$ denote the homology intersection pairings in~$X$ and~$Y$,
respectively.
Given an orientation-reversing diffeomorphism~$\vph$,
we describe below an operation gluing $\fc$-cycles in $X$ and $Y$ into $\fc$-cycles in 
$X\!\#_{\vph}\!Y$. 
It induces a homomorphism 
\BE{homsumdfn_e} H_{\fc}(X;\Z)\!\times_V\!H_{\fc}(Y;\Z)
\lra H_{\fc}(X\!\#_{\vph}\!Y;\Z)/\cR_{X,Y}^V, \qquad
(A_X,A_Y)\lra A_X\!\#_{\vph}\!A_Y,\EE
which is central to the symplectic sum formula for GW-invariants.\\

\noindent
Suppose
$$f_X\!:(Z_X,x_1,\ldots,x_{\ell})\lra(X,V) \qquad\hbox{and}\qquad  
f_Y\!:(Z_Y,y_1,\ldots,y_{\ell})\lra(Y,V)$$ 
are $\fc$-pseudocycles with boundary disjoint from~$V$ such that 
\begin{gather*}
f_X^{-1}(V)=\{x_1,\ldots,x_{\ell}\},\qquad f_Y^{-1}(V)=\{y_1,\ldots,y_{\ell}\},\\
f_X(x_i)=f_Y(y_i),\quad
\ord_{x_i}^Vf_X=\ord_{y_i}^Vf_Y\qquad\forall\,i\!=\!1,2,\ldots,\ell.
\end{gather*}
We can then obtain a smooth map $f_X\!\#_{\vph}\!f_Y\!:Z_X\!\#Z_Y\!\lra X\!\#_{\vph}\!Y$ by
\begin{enumerate}[label=$\bullet$,leftmargin=*]
\item removing small balls $B_{x_i}$ and $B_{y_i}$ around each of the points $x_i$ and $y_i$
to form manifolds with boundary $\hat{Z}_X$ and~$\hat{Z}_Y$,
\item forming a smooth oriented manifold $Z_X\!\#Z_Y$ by
identifying the $i$-th boundary components of $\hat{Z}_X$ and~$\hat{Z}_Y$ by 
an orientation-reversing diffeomorphism $\vph_i\!:(\prt\hat{Z}_X)_i\!\lra\!(\prt\hat{Z}_Y)_i$ 
for each~$i$,
\item homotoping $f_X$ and $f_Y$ on small neighborhoods of $\prt B_{x_i}$ and $\prt B_{y_i}$ 
within small balls around $f_X(x_i)$ in $X$ and $f_Y(y_i)$ in $Y$ so that
$\vph\!\circ\!f_X\!=\!f_Y\!\circ\!\vph_i$ for all~$i$.
\end{enumerate}
The last condition is achievable because the degrees of
$$\vph\!\circ\!f_X\!\circ\!\vph_i^{-1},f_Y\!: S^{\fc-1}\lra S^{\fc-1}$$
are the~same.\\

\noindent 
The above construction of $f_X\!\#_{\vph}\!f_Y$ depends only on $f_X$, $f_Y$, and 
choices of degree~$-1$ maps from $\ell$ disjoint copies of $[0,1]\!\times\!S^{\fc-1}$ to 
$[0,1]\!\times\!S^{\fc-1}$.
Thus, $f_X$ and $f_Y$ completely determine the homology class of $f_X\!\#_{\vph}\!f_Y$.
Furthermore,
$$q_{\vph*}\big([f_X\!\#_{\vph}\!f_Y]\big)=
\io^{X\cup_VY}_{X*}\big([f_X)]\big)\!+\!\io^{X\cup_VY}_{Y*}\big([f_Y]\big)
\in H_{\fc}(X\!\cup_{\vph}\!Y;\Z).$$
If $[f_X]\!=\![f_X']$ in $H_{\fc}(X;\Z)$, $[f_X\!\#\!(-f_X')]\!\in\!\cR_X^V$
by Section~\ref{prelimcut_subs}.
Thus, the homology class of $f_X\!\#_{\vph}\!f_Y$ in $X\!\#_{\vph}\!Y$
as above is determined by the homology classes of $f_X$ in $X$ and $f_Y$ in~$Y$
only up to an element of~$\cR_{X,Y}^V$.\\

\noindent
Suppose  $(X,\om_X)$ and $(Y,\om_Y)$ are symplectic manifolds
with a common compact symplectic divisor $V\!\subset\!X,Y$ such that
$$e(\cN_XV)=-e(\cN_YV)\in H^2(V;\Z).$$
The symplectic sum construction of \cite{Gf,MW} then produces 
a symplectic manifold $(X\!\#_V\!Y,\om_{\#})$ of the form $X\!\#_{\vph}\!Y$.
Let $\eta$ be a coset of $H_2(X\!\#_V\!Y;\Z)$ modulo~$\cR_{X,Y}^V$.
According to the symplectic sum formulas of \cite{LR,Jun2},
the sum of the genus~$g$ GW-invariants of $X\!\#_V\!Y$ in degrees $A\!\in\!\eta$
is the same as the sum of the genus~$g$ GW-invariants of
$(X,V,\om_X)$ and $(Y,V,\om_Y)$ of degrees $A_X$ and~$A_Y$ such that
 $A_X\!\#_{\vph}\!A_Y\!=\!\eta$.
It would of course be preferable to express individual GW-invariants of $X\!\#_V\!Y$
in terms of relative GW-invariants of $(X,V,\om_X)$ and~$(Y,V,\om_Y)$.
The rim tori refinement of standard relative invariants is suggested in~\cite{IPrel}
with the aim of resolving this deficiency in~\cite{IPsum}.
In~\cite{GWsumIP}, we discuss to what extent this is achieved.

\section{Cutting out a submanifold}
\label{cut_sec}

\noindent
We discuss changes in the homology after cutting out a submanifold in Section~\ref{cuthom_subs}.
Lemma~\ref{rimtori_lmm} contains \cite[Lemma 5.2]{IPrel} and 
the corresponding part of the proof of the former is essentially the same
as the proof of the latter.
We use it in Section~\ref{rimtori_subs} to give an explicit description of the rim tori module~$\cR_X^V$
and to compare it with the rim tori module~$\cR_X^{U\cup V}$ for a submanifold with additional connected
components.

\subsection{Changes in homology}
\label{cuthom_subs}

\noindent
Given a manifold $X$ and a closed submanifold $V\!\subset\!X$, 
we will view  the sphere subbundle~$S_XV$ of the normal bundle~$\cN_XV$ of $V$ in $X$
as a hypersurface in~$X$.
If in addition $V$ is compact oriented and the codimension of $V$ in $X$ is~$\fc$, 
we define
\BE{DeSV_e}
\De_X^V\!: H_m(V;\Z)\lra H_{m+\fc-1}(S_XV;\Z), \qquad
\De_X^V(\ga)=\PD_{S_XV}\big(q_X^{V\,^*}(\PD_V\ga)\big),\EE
where $q_X^V\!:S_XV\!\lra\!V$ is the projection map.
If $X$ is also compact, let
$$\cap V\!: H_*(X;\Z)\lra H_{*-\fc}(V;\Z),\qquad
A\cap V=\PD_V\big((\PD_XA)|_V\big).$$
If $U\!\subset\!X\!-\!V$ is another subset (possibly empty), we~take 
$$\cap V\!:\,  
H_*(X\!-\!U;\Z) \stackrel{\io_{X-U*}^X}{\xra{1.2}} H_*(X;\Z) \stackrel{\cap V}{\lra} H_{*-\fc}(V;\Z)$$
to be the composition.
Let
\BE{H1VXdfn_e} H^V_{X-U}=
\big\{A\!\cap\!V\!: A\!\in\!H_{\fc+1}(X\!-\!U;\Z)\big\}\subset H_1(V;\Z), \quad
H_1(V;\Z)_{X-U}=\frac{H_1(V;\Z)}{H^V_{X-U}}.\EE

\begin{lmm}\label{rimtori_lmm}
Suppose $X$ is a compact oriented manifold, $V\!\subset\!X$ is a compact oriented submanifold
of codimension~$\fc$, and $U\!\subset\!X\!-\!V$ is a compact subset.
Then,  the sequence
\BE{rimtori_e}\begin{split}
\ldots&\xra{1.7} H_m(X\!-\!U\!\cup\!V) \stackrel{\io_{X-U\cup V*}^{X-U}}{\xra{1.2}}
H_m(X\!-\!U) \stackrel{\cap V}{\xra{.4}} H_{m-\fc}(V)\\
&\stackrel{\io_{S_XV*}^{X-U\cup V}\!\circ\De_X^V}{\xra{1.7}}\! 
H_{m-1}(X\!-\!U\!\cup\!V) \xra{1}\ldots
\end{split}\EE
is exact for any coefficient ring.
\end{lmm}

\begin{proof}
Taking the Poincare dual of the Gysin sequence for $S_XV\!\lra\!V$, we obtain an exact sequence
\BE{GysinSeq_e}
\ldots\!\stackrel{\De_X^V}{\xra{.7}} 
H_m(S_XV) \stackrel{q_{X*}^V}{\xra{.8}} H_m(V) \stackrel{e(\cN_XV)\cap}{\xra{1.3}} 
H_{m-\fc}(V)  \stackrel{\De_X^V}{\xra{.7}} H_{m-1}(S_XV) \stackrel{q_{X*}^V}{\xra{.8}}
\!\ldots\,.\EE
By the proof of Mayer-Vietoris for $X\!-\!U\!=\!(X\!-\!U\!\cup\!V)\!\cup_{S_XV}\!V$, 
\BE{MSseqX_e}\begin{split}
\ldots\!&\stackrel{\de_X}{\xra{.3}}H_m(S_XV) \stackrel{(\io_{S_XV*}^{X-U\cup V},-q_{X*}^V)}{\xra{2.4}}
H_m(X\!-\!U\!\cup\!V)\!\oplus\!H_m(V)
\stackrel{\io_{X-U\cup V*}^{X-U}+\io_{V*}^{X-U}}{\xra{2}} H_m(X\!-\!U) \\
&\stackrel{\de_X}{\xra{.3}} H_{m-1}(S_XV) \stackrel{(\io_{S_XV*}^{X-U\cup V},-q_{X*}^V)}{\xra{2.4}}
\!\ldots\,,  
\end{split}\EE
the connecting homomorphism $\de_X$ is the composition $\De_X^V\circ(\cdot\!\cap\!V)$,
at least up to sign.
Since 
\BE{cap0_e}\io_{X-U\cup V*}^{X-U}(A)\cap V=0 \qquad\forall\,A\in H_m(X\!-\!U\!\cup\!V),\EE 
the claim now follows from the observation that 
\BE{capeuler_e}\io_{V*}^{X-U}(A)\cap V=e(\cN_XV)\cap A\qquad\forall\,A\in H_m(V),\EE
at least up to sign
(dependent on one's definitions of cup and cap products and Poincare dual);
see below for details.\\

\noindent
The composition of the first two labeled homomorphisms in~\eref{rimtori_e} being~0 is equivalent
to~\eref{cap0_e}.
If $A_{X-U}\!\cap\!V\!=\!0$ for some $A_{X-U}\!\in\!H_m(X\!-\!U)$, then 
$\de_X(A_{X-U})\!=\!0$ and~so
$$ A_{X-U}= \io_{X-U\cup V*}^{X-U}\big(A_{X-U\cup V}\big)+\io_{V*}^{X-U}(A_V)
\quad\hbox{for some}~~ A_{X-U\cup V}\in H_m(X\!-\!U\!\cup\!V),~A_V\!\in\!H_m(V),$$
by the exactness of~\eref{MSseqX_e}.
By $A_{X-U}\!\cap\!V\!=\!0$, \eref{cap0_e}, and~\eref{capeuler_e}, 
$e(\cN_XV)\!\cap\!A_V\!=\!0$ and~so
$$A_V=q_{X*}^V\big(A_{S_XV}\big)  \quad\hbox{for some}~~ A_{S_XV}\in H_m(S_XV)$$
by the exactness of~\eref{GysinSeq_e}.
By the last two displayed expressions,
$$A_{X-U}=  \io_{X-U\cup V*}^{X-U}\big(A_{X-U\cup V}+ \io_{S_XV*}^{X-U\cup V}(A_{S_XV})\big),$$
which establishes the exactness of~\eref{rimtori_e} at $H_m(X\!-\!U)$.\\ 

\noindent
The composition of the last two labeled homomorphisms in~\eref{rimtori_e} being~0
follows from the exactness of~\eref{MSseqX_e} and $\de_X\!=\!\De_X^V\circ(\cdot\!\cap\!V)$.
Suppose
$$\big\{\io_{S_XV*}^{X-U\cup V}\!\circ\De_X^V\big\}(A_V)=0
\quad\hbox{for some}~~A_V\in H_{m-\fc}(V).$$
Since $q_{X*}^V\!\circ\!\De_X^V\!=\!0$ by the exactness of~\eref{GysinSeq_e}, 
$$\De_X^V(A_V)=\de_X\big(A_{X-U}\big)=\De_X^V\big(A_{X-U}\!\cap\!V\big)
\quad\hbox{for some}~~A_{X-U}\in H_m(X\!-\!U)$$
by the exactness of~\eref{MSseqX_e}.
Thus,
$$A_V=A_{X-U}\!\cap\!V+e(\cN_XV)\!\cap\!A_V' 
=\big(A_{X-U}\!+\!\io_{V*}^{X-U}(A_V')\big)\cap V 
\quad\hbox{for some}~~A_V'\in H_m(V)$$
by the exactness of~\eref{GysinSeq_e} and~\eref{capeuler_e}.
This establishes the exactness of~\eref{rimtori_e} at $H_{m-\fc}(V)$.\\ 

\noindent
The vanishing of the composition 
$$\io_{X-U\cup V*}^{X-U}\!\circ\!\big\{\io_{S_XV*}^{X-U\cup V}\!\circ\De_X^V\big\}\!:\,
H_{m-\fc}(V)\lra H_{m-1}(X\!-\!U)$$
follows from the exactness of~\eref{MSseqX_e}.
If $\io_{X-U\cup V*}^{X-U}(A_{X-U\cup V})\!=\!0$ for some 
$A_{X-U\cup V}\!\in\!H_{m-1}(X\!-\!U\!\cup\!V)$, then
$$A_{X-U\cup V}=\io_{S_XV*}^{X-U\cup V}\big(A_{S_XV}\big),~~ q_{X*}^V\big(A_{S_XV}\big)=0 
\quad\hbox{for some}~~A_{S_XV}\in H_{m-1}(S_XV)$$
by the exactness of~\eref{MSseqX_e}.
By the exactness of~\eref{GysinSeq_e} and $q_{X*}^V(A_{S_XV})\!=\!0$,
$$A_{S_XV}=\De_X^V(A_V) \quad\hbox{for some}~~A_V\in H_{m-\fc}(V).$$
This establishes the exactness of~\eref{rimtori_e} at $H_{m-1}(X\!-\!U\!\cup\!V)$.
\end{proof}

\noindent
By \cite[Theorem~1.1]{Z}, every integral homology class in a manifold can be represented by a pseudocycle.
If $V$ is as in~\eref{DeSV_e} and $f\!:Z\!\lra\!V$ is a pseudocycle, then the pseudocycle
$$\pi_2\!: f^*S_XV\equiv\big\{(z,v)\!\in\!Z\!\times\!S_XV\!:\,f(z)\!=\!q_X^V(v)\big\}
\lra S_XV,$$
where $\pi_2\!:Z\!\times\!S_XV\!\lra\!S_XV$ is the projection on the second coordinate,
represents $\De_X^V([f])$.
If $Z\!=\!S^1$, $f^*S_XV\!\lra\!S^1$ is a trivial $S^{\fc-1}$-bundle and 
thus $f$ lifts to a map 
$$\ti{f}\!:S^1\!\times\!S^{\fc-1}\lra \{q_X^V\}^{-1}(f(S^1))\subset S_XV
\qquad\hbox{s.t.}\quad q_X^V\circ\ti{f}=f\circ\pi_1\,,$$
where $\pi_1\!:S^1\!\times\!S^{\fc-1}\!\lra\!S^1$ is the projection on the first component.
Thus, the elements~of the module~$\cR_{X-U}^V$ as in~\eref{cRXVdfn_e}
can be represented by cycles of the form $\io_{S_XV}^{X-U\cup V}(S_XV|_{\ga})$ 
for loops $\ga\!\subset\!V$, according to Lemma~\ref{rimtori_lmm}.

\subsection{The rim tori}
\label{rimtori_subs}

\noindent
We now use Lemma~\ref{rimtori_lmm} to describe the rim tori module~$\cR_{X-U}^V$ explicitly
and to compare it with other such modules.
The first corollary below is an immediate consequence of the $m\!=\!\fc$ case of this lemma.

\begin{crl}\label{rimtori_crl}
Suppose $X$ is a compact oriented $n$-manifold, $V\!\subset\!X$ is a compact oriented submanifold
of codimension~$\fc$, and $U\!\subset\!X\!-\!V$ is a compact subset.
Then, the induced homomorphism
$$\io_{S_XV*}^{X-U\cup V}\!\circ\De_X^V\!: H_1(V;\Z)_{X-U}\lra 
\cR_{X-U}^V \subset H_{\fc}(X\!-\!U\!\cup\!V;\Z)$$
is well-defined and is an isomorphism.
In particular, if the restriction homomorphism 
\BE{rimtoricrl_e} H^{n-\fc-1}(X;\Z)\!\lra\!H^{n-\fc-1}(V;\Z),\EE
is zero, then  $\cR_{X-U}^V\!\approx\!H_1(V;\Z)$.
\end{crl}

\begin{crl}\label{rimtori_crl2}
Suppose $X$ is a compact oriented manifold and $U,V\!\subset\!X$ are compact oriented disjoint submanifolds
of codimension~$\fc$.
Then, the homology homomorphisms induced by inclusions give rise to the commutative square 
of short exact sequences of Figure~\ref{compRT_fig}.
\end{crl}

\begin{proof}
The rows in this diagram are exact by Corollary~\ref{rimtori_crl}.
The exactness of the middle column is clear, as is the exactness of
the last column at~$\cR_{X-V}^U$.
The last column is exact at~$\cR_X^{U\cup V}$ by~\eref{cRXVdfn_e},
with $(X,V)$ replaced by~$(X\!-\!V,U)$.
By the exactness of~\eref{MSseqX_e},
\begin{equation*}\begin{split}
\cR_X^V&=\big\{\io_{S_XV*}^{X-V}(A_{S_XV})\!:\,A_{S_XV}\!\in\!H_{\fc}(S_XV;\Z),~
q_{X*}^V(S_XV)\!=\!0\big\}, \\
\cR_X^{U\cup V}&\supset\big\{\io_{S_XV*}^{X-U\cup V}(A_{S_XV})\!:\,A_{S_XV}\!\in\!H_{\fc}(S_XV;\Z),~
q_{X*}^V(S_XV)\!=\!0\big\}.
\end{split}\end{equation*}
This implies that the last column is exact at~$\cR_X^V$.\\ 

\begin{figure}
\begin{gather*}
\xymatrix{& 0\ar[d]& 0\ar[d]&& 0\ar[d]& \\
0 \ar[r]& H_{X-V}^U \ar[r] \ar[d]& H_1(U;\Z) 
\ar[rr]^{\io_{S_XU*}^{X-U\cup V}\!\circ\De_X^U}\ar[d]&& 
\cR_{X-V}^U  \ar[r]\ar[d]& 0\\
0 \ar[r]& H_X^{U\cup V} \ar[r]\ar[d]& H_1(U;\Z)\!\oplus\!H_1(V;\Z) \ar[rr]\ar[d]&& 
\cR_X^{U\cup V}  \ar[r]\ar[d]^{\io_{X-U\cup V}^{X-V}}& 0\\
0 \ar[r]& H_X^V \ar[r] \ar[d]& H_1(V;\Z) 
\ar[rr]^{\io_{S_XV*}^{X-V}\!\circ\De_X^V}\ar[d]&& 
\cR_X^V  \ar[r]\ar[d]& 0\\
& 0& 0&& 0&}
\end{gather*}
\caption{Comparison of rim tori.}
\label{compRT_fig}
\end{figure}

\noindent
By~\eref{H1VXdfn_e} and \eref{cap0_e} with $U\!=\!\eset$,
$$H_{X-V}^U\!\oplus\!0 \subset  H_X^{U\cup V}\cap H_1(U;\Z)\!\oplus\!0
=\big\{(A\!\cap\!U,0)\!:\,A\!\in\!H_{\fc+1}(X;\Z)~\hbox{s.t.}~A\!\cap\!V\!=\!0\big\}.$$
In particular,  the left column in the diagram is exact  at~$H_{X-V}^U$.
By the exactness of~\eref{rimtori_e} with $m\!=\!\fc\!+\!1$ and $U\!=\!\eset$ at the second position,
the above inclusion is in fact an equality and the left column is exact at~$H_X^{U\cup V}$.
The exactness of the left column at~$H_X^{U\cup V}$ is immediate from~\eref{H1VXdfn_e}.\\

\noindent
The commutativity of the bottom right square is equivalent to the homomorphism
$$\io_{X-U\cup V*}^{X-V}\!\circ\!\io_{S_XU*}^{X-U\cup V}\!\circ\!\De_X^U\!: H_1(U;\Z)\lra H_{\fc}(X\!-\!V;\Z)$$
being zero.
This is immediate from the exactness of~\eref{MSseqX_e} with $U$ and~$V$ interchanged, since 
$q_{X*}^U\!\circ\!\De_X^U\!=\!0$ by the exactness of~\eref{GysinSeq_e}.
The commutativity of the other three squares is clear.
\end{proof}

\begin{eg}\label{Proj_eg}
Let $Z$ be a compact oriented manifold and $E_1,E_2\!\lra\!Z$ be complex vector bundles
of ranks~$r_1$ and~$r_2$, respectively.
By \cite[Theorem~5.7.9]{Sp}, there is a commutative diagram 
$$\xymatrix{ H_*(\P E_2;\Z) \ar[rr]\ar[d]_{\approx}&& 
H_*(\P(E_1\!\oplus\!E_2);\Z) \ar[d]^{\approx}\\ 
H_*(Z;\Z)\otimes H_*(\P^{r_2-1};\Z) \ar[rr]&& H_*(Z;\Z)\otimes H_*(\P^{r_1+r_2-1};\Z)}$$
of homomorphisms of modules. 
In particular, the homomorphism
$$\io_{\P(E_1\oplus E_2)-\P E_1*}^{\P(E_1 \oplus E_2)}\!:
H_{2r_2}\big(\P(E_1\!\oplus\!E_2)\!-\!\P E_1;\Z\big)
\approx H_{2r_2}\big(\P E_2;\Z\big) \lra H_{2r_2}\big(\P(E_1\!\oplus\!E_2);\Z\big)$$
is injective.
Thus,
$$\cR_{\P(E_1 \oplus E_2)}^{\P E_1}=\{0\}\,.$$
In the case $r_2\!=\!1$, which is the most relevant for our purposes, 
this statement follows from $\P E_2$ being
a section of the fiber bundle $\P(E_1\!\oplus\!E_2)\!\lra\!Z$.
\end{eg}

\begin{eg}\label{EllFib_eg}
Let $\wh\P^2_9$ denote the blowup of $\P^2$ at the 9 intersection points 
of a general pair of smooth cubic curves~$C_1$ and~$C_2$,
i.e.~the base locus of a general pencil of cubics in~$\P^2$.
The proper transforms of these cubics in $\wh\P^2_9$ are pairwise disjoint and 
form a fibration $\pi\!:\wh\P^2_9\!\lra\!\P^1$, obtained by sending each point
in $\wh\P^2_9$ to the cubic in the pencil passing through~it.
A smooth fiber~$F$ of~$\pi$ is a torus~$\T^2$;
there are also 12~singular fibers, each of which is a sphere with a transverse self-intersection.
By Corollary~\ref{rimtori_crl},
$$\cR_{\wh\P^2_9}^F\approx H_1(F;\Z)\approx\Z^2,$$
because $H^1(\wh\P^2_9;\Z)\!=\!0$.
\end{eg}

\begin{eg}\label{P1Vrim_eg}
Let $F$ be a compact oriented manifold, $X\!=\!\P^1\!\times\!F$, 
$F_0\!=\!\{0\}\!\times\!F$, and $F_{\i}\!=\!\{\i\}\!\times\!F$.
Then,
\begin{gather}
\notag
H_X^{F_0}=H_1(F_0;\Z)=H_1(F;\Z), \qquad
H_{X-F_0}^{F_{\i}}\!=\!\{0\}\subset H_1(F_{\i};\Z)= H_1(F;\Z),\\
\label{P1Vrim_e3}
H_X^{F_0\cup F_{\i}}\!=\!H_{\De}\subset H_1(F_0\!\cup\!F_{\i};\Z)
=H_1(F;\Z)\!\oplus\!H_1(F;\Z),
\end{gather}
where $H_{\De}$ is the diagonal subgroup.
By Lemma~\ref{rimtori_lmm} or Corollary~\ref{rimtori_crl}, the homomorphism
$$\io_{S_XF_0*}^{X-F_0\cup F_{\i}}\!\circ\De_X^{F_0}\!: 
H_1(F_0;\Z)\!=\!H_1(F;\Z)\lra \cR_X^{F_0\cup F_{\i}}  \subset H_{\fc}(X\!-\!F_0\!\cup\!F_{\i};\Z)$$
is an isomorphism.
Under this identification, the last labeled homomorphism in~\eref{rimtori_e}
 with $m\!=\!3$ corresponds~to
\BE{P1Vrim_e5}H_1(F_0\!\cup\!F_{\i};\Z)\!=\!H_1(F;\Z)\!\oplus\!H_1(F;\Z)\lra H_1(F;\Z), 
\quad (\ga_0,\ga_{\i})\lra\ga_0\!-\!\ga_{\i}\,.\EE
\end{eg}

\section{Gluing along a common submanifold}
\label{glue_sec}

\noindent
We discuss changes in the homology and cohomology after gluing two manifolds along a common submanifold
in Sections~\ref{gluehom_subs} and~\ref{gluecoh_subs}, respectively.
We use Lemma~\ref{cRXYV_lmm} in Section~\ref{vancyc_subs} to express
the vanishing cycles module~$\cR_{X,Y}^V$ defined in~\eref{cRXYVdfn_e}
in terms of the rim tori modules~$\cR_X^V$ and~$\cR_Y^V$ defined in~\eref{cRXVdfn_e}.
This lemma is used in Section~\ref{gluerimtori_subs} to compare the rim tori modules before and 
after gluing.
Lemma~\ref{sumcoh_lmm} contains the precise statements of~(a) and~(b) 
at the bottom of \cite[p996]{IPsum}; it is useful for determining the cohomology insertions
compatible with the symplectic sum formula for GW-invariants.

\subsection{Changes in homology}
\label{gluehom_subs}

\noindent
Continuing with the notation introduced in Section~\ref{prelimglue_subs},
we relate the analogue of the module~\eref{cRXYVdfn_e} for homology of any dimension
to  such analogues of the modules~\eref{cRXVdfn_e} for the two pieces.

\begin{lmm}\label{cRXYV_lmm}
If $X$ and $Y$ are manifolds, $V\!\subset\!X,Y$ is a closed submanifold, 
$\vph\!:S_XV\!\lra\!S_YV$ is a diffeomorphism commuting with the projections to~$V$, 
and $q_{\vph}\!:X\!\#_{\vph}\!Y\lra X\!\cup_V\!Y$ is a collapsing map, then
\BE{cRXYV_e}\begin{split}
&\ker\big\{q_{\vph*}\!:\,H_m(X\!\#_{\vph}\!Y)\!\lra\!H_m(X\!\cup_{\vph}\!Y)\big\}\\
&\hspace{1in}=\big\{\io^{X\#_{\vph}Y}_{X-V*}(A_{X-V})\!:~
A_{X-V}\!\in\!H_m(X\!-\!V),~\io_{X-V*}^X(A_{X-V})\!=\!0\big\}\\
&\hspace{1in}=\big\{\io^{X\#_{\vph}Y}_{Y-V*}(A_{Y-V})\!:~
A_{Y-V}\!\in\!H_m(Y\!-\!V),~\io_{Y-V*}^Y(A_{Y-V})\!=\!0\big\}
\end{split}\EE
for all $m\!\in\!\Z$ and for any coefficient ring.
\end{lmm}

\begin{proof}
Denote the codimension of $V$ in $X$ and $Y$ by $\fc$,
$S_XV\!\approx\!S_YV$ by $SV$, and the bundle projection map $SV\!\lra\!V$ by~$q_V$.
Mayer-Vietoris for $X\!\#_{\vph}\!Y\!=\!(X\!-\!V)\!\cup_{SV}\!(Y\!-\!V)$ 
and $X\!\cup_V\!Y$ give a commutative pair of long exact sequences
$$\xymatrix{H_m(SV)\ar[r] \ar[d]^{q_{V*}}&
H_m(X\!-\!V)\oplus H_m(Y\!-\!V) 
\ar[d]|{\io^X_{X-V*}\oplus\io^Y_{Y-V*}}
\ar[rr]^>>>>>>>>>{\io^{X\#_{\vph}\!Y}_{X-V*}\!+\io^{X\#_{\vph}\!Y}_{Y-V*}} &&
H_m(X\!\#_{\vph}\!Y) \ar[d]^{q_{\vph*}}\ar[r]^{\de_{\vph}}& H_{m-1}(SV) \ar[d]^{q_{V*}}\\
H_m(V)\ar[r]^>>>>>>>>>{(\io_{V*}^X,-\io_{V*}^Y)} &
H_m(X)\oplus H_m(Y) 
\ar[rr]^>>>>>>>>>>>>{\io^{X\cup_V\!Y}_{X*}\!+\io^{X\cup_V\!Y}_{Y*}}
&& H_m(X\!\cup_V\!Y) \ar[r]^{\de_{\cup}}& H_{m-1}(V) }$$
The commutativity of the middle square above implies that the second and third expressions
in~\eref{cRXYV_e} are contained in the first.\\

\noindent
Suppose $A_{\#}\!\in\!H_m(X\!\#_V\!Y)$ and $q_{\vph*}(A_{\#})\!=\!0$.
By the proof of Mayer-Vietoris for $X\!\#_{\vph}\!Y$,
there exist bordered pseudocycles
$$f_X\!:(Z_X,\prt Z_X)\lra (X\!-\!V,SV) \qquad\hbox{and}\qquad
f_Y\!:(Z_Y,\prt Z_Y)\lra (Y\!-\!V,SV)$$
such that  $\prt Z_X=-\prt Z_Y$ and
$$f_X \underset{\prt Z_X=-\prt Z_Y}{\cup}f_Y\!:
Z_X\underset{\prt Z_X=-\prt Z_Y}{\cup}Z_Y\lra X\!\#_{\vph}\!Y$$
represents the homology class~$A_{\#}$.
Since 
$$q_{V*} \big[f_X|_{\prt Z_X}\big]
=q_{V*}\de_{\vph}(A_{\#})=\de_{\cup}q_{\vph*}(A_{\#})=0,$$
by \eref{GysinSeq_e} we can choose $f_X$ and $f_Y$ so that $f_X(\prt Z_X)\!=\!f_Y(\prt Z_Y)$ 
equals $SV|_{B_V}$ for some class \hbox{$B_V\!\in\!H_{m-\fc}(V)$}.
The smooth maps 
$$\io_{X-V}^X\!\circ\!f_X\!:Z_X\lra X \qquad\hbox{and}\qquad
\io_{Y-V}^Y\!\circ\!f_Y\!:Z_Y\lra Y$$  
then determine homology classes $A_X$ on $X$ and $A_Y$ on~$Y$;
in the exceptional $\fc\!=\!1$ case, the two boundary components of these 
maps come with opposite signs and thus cancel.
By the commutativity of the diagram on the chain level inducing the above diagram
in homology,
$$\io^{X\cup_V\!Y}_{X*}(A_X)+\io^{X\cup_V\!Y}_{Y*}(A_Y)
=q_{\vph*}(A_{\#})=0.$$
Thus, there exists $A_V\!\in\!H_m(V)$ such that 
$$A_X=\io_{V*}^X(A_V), \qquad A_Y=-\io_{V*}^Y(A_V)\,.$$
The Mayer-Vietoris sequence~\eref{MSseqX_e} with $U\!=\!\eset$ then gives 
$$\big[f_X|_{\prt Z_X}\big]=\de_X(A_X)=0 \qquad\Lra\qquad
\de_{\vph}(A_{\#})=\big[f_X|_{\prt Z_X}\big]=  0.$$
The first Mayer-Vietoris sequence now implies that  
\BE{AXYV_e} A_{\#}=\io_{X-V*}^{X\#_{\vph}Y}(A_{X-V})+\io_{Y-V*}^{X\#_{\vph}Y}(A_{Y-V})\EE
for some $A_{X-V}\in H_m(X\!-\!V)$ and $A_{Y-V}\in H_m(Y\!-\!V)$.
Since 
\begin{equation*}\begin{split}
&\io_{X*}^{X\cup_V Y}\big(\io_{X-V*}^X(A_{X-V})\big)+
\io_{Y*}^{X\cup_V Y}\big(\io_{Y-V*}^Y(A_{Y-V})\big)\\
&\hspace{.5in}=q_{\vph*}\big(\io_{X-V*}^{X\#_{\vph}Y}(A_{X-V})\big)+
q_{\vph*}\big(\io_{Y-V*}^{X\#_{\vph}Y}(A_{Y-V})\big)
=q_{\vph*}\big(A_{\#}\big)=0,
\end{split}\end{equation*}
the second Mayer-Vietoris sequence above implies that  
\BE{AXYV_e2} \io_{X-V*}^X(A_{X-V})=\io_{V*}^X(A_V), \qquad 
\io_{Y-V*}^Y(A_{Y-V})=-\io_{V*}^Y(A_V)\EE
for some $A_V\!\in\!H_m(V)$.
By the first equality above and the exactness of~\eref{MSseqX_e}, there exists
\BE{AXYV_e3}
A_{SV}\in H_m(SV)  \qquad\hbox{s.t.}\qquad
\io_{SV*}^{X-V}(A_{SV})=A_{X-V}, \quad q_{V*}(A_{SV})=A_V.\EE
By~\eref{AXYV_e} and the first equality in~\eref{AXYV_e3}, 
\begin{equation*}\begin{split}
A_{\#}&=\io_{X-V*}^{X\#_{\vph}Y}\big(A_{X-V}-\io_{SV*}^{X-V}(A_{SV})\big)
+\io_{Y-V*}^{X\#_{\vph}Y}\big(A_{Y-V}+\io_{SV*}^{Y-V}(A_{SV})\big)\\
&=\io_{Y-V*}^{X\#_{\vph}Y}\big(A_{Y-V}+\io_{SV*}^{Y-V}(A_{SV})\big).
\end{split}\end{equation*}
By the second equalities in~\eref{AXYV_e2} and~\eref{AXYV_e3}, 
$$\io_{Y-V*}^Y\big(A_{Y-V}+\io_{SV*}^{Y-V}(A_{SV})\big)
=-\io_{V*}^Y(A_V)+\io_{V*}^Y\big(q_{V*}(A_{SV})\big)=0.$$
Thus, the first expression in~\eref{cRXYV_e} is contained in the third expression
and by symmetry in the~second.
\end{proof}

\noindent
Taking $m\!=\!\fc$ in the statement of Lemma~\ref{cRXYV_lmm}, we find that 
\BE{VanRim_e}  \cR_{X,Y}^V = \io^{X\#_{\vph}Y}_{X-V*}(\cR_X^V)=\io^{X\#_{\vph}Y}_{Y-V*}(\cR_Y^V);\EE
see \eref{cRXVdfn_e} and~\eref{cRXYVdfn_e} for the notation.

\subsection{The vanishing cycles}
\label{vancyc_subs}

\noindent
We now focus on the $m\!=\!\fc$ case of Lemma~\ref{cRXYV_lmm}; it relates~$\cR_{X,Y}^V$
to~$\cR_X^V$ and~$\cR_Y^V$.
Let
\BE{HfcSVdfn_e}H_{\fc}(SV;\Z)_{X,Y}=
\big\{A_{SV}\!\in\!H_{\fc}(SV;\Z)\!:\,q_{V*}(A_{SV})\!\in\!\ker\io^X_{V*}\cap\ker\io^Y_{V*}\big\}.\EE
By the exactness of~\eref{MSseqX_e}, the homomorphisms
\BE{SVtoRXV_e}\begin{split}
\io^{X-V}_{SV*}\!:\,H_{\fc}(SV;\Z)_{X,Y}&\lra \cR_X^V\subset H_{\fc}(X\!-\!V;\Z)
\qquad\hbox{and}\\
\io^{Y-V}_{SV*}\!:H_{\fc}(SV;\Z)_{X,Y}&\lra \cR_Y^V \subset H_{\fc}(Y\!-\!V;\Z)
\end{split}\EE
are  surjective.

\begin{crl}\label{cRXYV_crl}
Let $X,Y,V$ and $\vph$ be as in Lemma~\ref{cRXYV_lmm} and $\fc$ be the codimension of~$V$ 
in~$X$ and~$Y$.
\begin{enumerate}[label=(\arabic*),leftmargin=*]
\item\label{cRXYVgen_it0} 
The subgroup $\cR_{X,Y}^V\subset H_{\fc}(X\!\#_{\vph}\!Y;\Z)$ is isomorphic to the cokernel 
of the homomorphism 
\BE{H2sum_e5}
H_{\fc}(SV;\Z)_{X,Y}\lra \cR_X^V\!\oplus\!\cR_Y^V,
~~~ A_{SV}\lra \big(\io^{X-V}_{SV*}(A_{SV}),-\io^{Y-V}_{SV*}(A_{SV})\big).\EE

\item\label{cRXYVtriv_it0} 
If either $\cR_X^V\!=\!\{0\}$ or $\cR_Y^V\!=\!\{0\}$, then $\cR_{X,Y}^V\!=\!\{0\}$ and
 the homomorphism
\BE{H2sum_e0}
\#\!: H_{\fc}(X;\Z)\!\times\!_VH_{\fc}(Y;\Z) \lra H_{\fc}(X\!\#_{\vph}\!Y;\Z), \quad 
(A_X,A_Y)\lra A_X\!\#_{\vph}\!A_Y,\EE
induced by gluing representatives of homology classes along $V$, is well-defined.
\end{enumerate}
\end{crl}

\begin{proof}
\ref{cRXYVgen_it0} By the $m\!=\!\fc$ case of Lemma~\ref{cRXYV_lmm}, this claim is equivalent to 
\BE{H2sumcrl_e1}\begin{split}
&\big\{(A_{X-V},A_{Y-V})\in \cR_X^V\!\oplus\!\cR_Y^V\!:~
\io^{X\#_{\vph}Y}_{X-V*}(A_{X-V})\!+\!\io^{X\#_{\vph}Y}_{Y-V*}(A_{Y-V})=0\big\}\\
&\quad
=\big\{\big(\io_{SV*}^{X-V}(A_{SV}),-\io_{SV*}^{Y-V}(A_{SV})\big)\!:\,
A_{SV}\!\in\!H_{\fc}(SV;\Z),~q_{V*}(A_{SV})\in \ker\io^X_{V*}\cap\ker\io^Y_{V*}\big\}.
\end{split}\EE
Let $A_{X-V}\!\in\!H_{\fc}(X\!-\!V;\Z)$ and $A_{Y-V}\!\in\!H_{\fc}(Y\!-\!V;\Z)$.
By the Mayer-Vietoris sequence for 
$X\!\#_{\vph}\!Y$ in the proof of Lemma~\ref{cRXYV_lmm},
\begin{gather*}
\io^{X\#_{\vph}Y}_{X-V*}(A_{X-V})\!+\!\io^{X\#_{\vph}Y}_{Y-V*}(A_{Y-V})=0
\qquad\Llra\\
(A_{X-V},A_{Y-V})=\big(\io_{SV*}^{X-V}(A_{SV}),-\io_{SV*}^{Y-V}(A_{SV})\big)
\quad\hbox{for some}~A_{SV}\!\in\!H_{\fc}(SV;\Z).
\end{gather*}
For any $A_{SV}\!\in\!H_{\fc}(SV;\Z)$, the commutativity of the first square in 
the diagram of short exact sequences in the proof of Lemma~\ref{cRXYV_lmm} implies that 
$$\big(\io_{SV*}^{X-V}(A_{SV}),-\io_{SV*}^{Y-V}(A_{SV})\big)\in \cR_X^V\!\oplus\!\cR_Y^V
\qquad\Llra\qquad
q_{V*}(A_{SV})\in \ker\io^X_{V*}\cap\ker\io^Y_{V*}.$$
The last two statements give~\eref{H2sumcrl_e1}.\\

\noindent
\ref{cRXYVtriv_it0}  
The second claim of this corollary follows from the first and the surjectivity of 
the homomorphisms~\eref{SVtoRXV_e}.
\end{proof}

\begin{rmk}\label{RTvph_rmk}
The patching map~$\vph$ covering the identity on~$V$
does not effect the homomorphism~\eref{H2sum_e5},
as the former corresponds to a trivialization of an $S^{c-1}$-bundle over~$S^1$ for each fixed
element of $\cR_X^V\!\oplus\!\cR_Y^V$.
Thus, $\cR_{X,Y}^V$ does not depend on the choice of~$\vph$.
However, it may depend on the identification of the copies 
of~$V$ in~$X$ and~$Y$, as illustrated in Example~\ref{P1Vrim_eg2}.
\end{rmk}

\noindent
We next restrict to the setting of Corollary~\ref{rimtori_crl} with $U\!=\!\eset$;
the last restriction is not necessary, but the case $U\!=\!\eset$ suffices for our purposes.
Thus, suppose that $X$, $Y$, and~$V$ are compact and oriented.
Define 
$$\De_{X,Y}^V\!: 
H_1(V;\Z)_X \oplus H_1(V;\Z)_Y \lra\frac{H_1(V;\Z)}{H_X^V\!+\!H_Y^V}\,,\quad
\big([\ga_X]_{H_X^V},[\ga_Y]_{H_Y^V}\big)\lra \big[\ga_X\!-\!\ga_Y\big]_{H_X^V+H_Y^V}.$$
Denote by $\ov{H}_{X,Y}^V$ the image of the composition
$$H_{\fc}(SV;\Z)_{X,Y}
\stackrel{(\io_{SV*}^{X-V},-\io_{SV*}^{Y-V})}{\xra{1.8}}
\cR_X^V\!\oplus\!\cR_Y^V \stackrel{\approx}{\lra}
H_1(V;\Z)_X \oplus H_1(V;\Z)_Y \stackrel{\De_{X,Y}^V}{\xra{.8}}\frac{H_1(V;\Z)}{H_X^V\!+\!H_Y^V},$$
with the second arrow above given by the isomorphisms of Corollary~\ref{rimtori_crl}.
Let $H_{X,Y}^V\!\subset\!H_1(V;\Z)$ be the preimage of $\ov{H}_{X,Y}^V$
under the quotient projection
$$H_1(V;\Z)\lra \frac{H_1(V;\Z)}{H_X^V\!+\!H_Y^V}\,.$$
In particular,
\BE{HVXYprp_e}H_X^V\!+\!H_Y^V\subset H_{X,Y}^V\subset H_1(V;\Z)\EE
and the first inclusion is an equality if either $\io^X_{V*}$ or $\io^Y_{V*}$
is injective on~$H_{\fc}$; see the proof of Corollary~\ref{cRXYV_crl2} below.

\begin{crl}\label{cRXYV_crl2}
Let $X$ and $Y$ be compact oriented manifolds, 
$V\!\subset\!X,Y$ be a  compact oriented submanifold, and 
$\vph\!:S_XV\!\lra\!S_YV$ be an orientation-reversing diffeomorphism commuting with the projections to~$V$.
\begin{enumerate}[label=(\arabic*),leftmargin=*]
\item\label{cRXYV2H1_it} 
The isomorphisms of Corollary~\ref{rimtori_crl} for $(X,V)$ and~$(Y,V)$ induce a commutative diagram
$$\xymatrix{ \frac{H_1(V;\Z)}{H_X^V}\oplus  \frac{H_1(V;\Z)}{H_Y^V}
\ar[d]|{\io_{S_XV*}^{X-V}\!\circ\De_X^V\oplus\io_{S_YV*}^{Y-V}\!\circ\De_Y^V}
\ar[rrrr]^{\ov\De_{X,Y}^V}  &&&&    \frac{H_1(V;\Z)}{H_{X,Y}^V}  \ar[d]_{\approx}^{\fR_{X,Y}^V} \\
 \cR_X^V\oplus\cR_Y^V   \ar[rrrr]^{\io_{X-V*}^{X\#_{\vph}Y}+\io_{Y-V*}^{X\#_{\vph}Y}} &&&& 
  \cR_{X,Y}^V\,. }$$
\item\label{cRXYV2inj_it}
If either of $\io_{V*}^X$ or $\io_{V*}^Y$ is injective on~$H_{\fc}$, then
$\cR_{X,Y}^V$ is isomorphic to the cokernel of the homomorphism 
\BE{H1vsRXYV_e}H_1(V;\Z)\lra H_1(V;\Z)_X\!\oplus\!H_1(V;\Z)_Y, \quad
\ga\lra \big([\ga]_X,[\ga]_Y\big),\EE
with $[\ga]_X$ and $[\ga]_Y$ denoting the corresponding cosets of $\ga$.
\end{enumerate}
\end{crl}

\begin{proof}
\ref{cRXYV2H1_it} The image of $\ker\De_{X,Y}^V$ under the isomorphisms of  
Corollary~\ref{rimtori_crl} is given~by
\BE{cRXYV2H1_e}\begin{split}
&\big\{\big(\io_{S_XV*}^{X-V}(\De_X^V(\ga)),\io_{S_YV*}^{Y-V}(\De_Y^V(\ga))\big)\!:\,
\ga\!\in\!H_1(V;\Z)\big\}\\
&\hspace{1.5in}
=\big\{\big(\io_{S_XV*}^{X-V}(\De_X^V(\ga)),-\io_{S_YV*}^{Y-V}(\De_X^V(\ga))\big)\!:\,
\ga\!\in\!H_1(V;\Z)\big\},
\end{split}\EE
since $\vph$ is orientation-reversing.
Thus, this image is contained in the image of the homomorphism~\eref{H2sum_e5}.
The first claim now follows from Corollary~\ref{cRXYV_crl}\ref{cRXYVgen_it0} and 
the definition of~$H_{X,Y}^V$ above.\\

\noindent
\ref{cRXYV2inj_it} If either $\io_{V*}^X$ or $\io_{V*}^Y$ is injective on~$H_{\fc}$,
\BE{Hfcinj_e} H_{\fc}(SV;\Z)_{X,Y}=\ker q_{V*}=\big\{\De_X^V(\ga)\!:\,\ga\!\in\!H_1(V;\Z)\big\};\EE
the last equality holds by the exactness of~\eref{GysinSeq_e}.
Thus, the image of the homomorphism~\eref{H2sum_e5} is given by~\eref{cRXYV2H1_e}
and $H_{X,Y}^V\!=\!H_X^V\!+\!H_Y^V$.
The second claim of this corollary now follows from the first.
\end{proof}

\begin{crl}\label{cRXYV_crl3}
If $X$ is a compact oriented manifold, 
$V\!\subset\!X$ is a compact oriented submanifold, and 
$\vph\!:S_XV\!\lra\!S_XV$ is an orientation-reversing diffeomorphism commuting with the projection to~$V$,
then  $\cR_{X,X}^V\!\approx\!H_1(V;\Z)_X$.
\end{crl}

\begin{proof}
Since $\io_{V*}^X\!\circ\!q_{V*}\!=\!\io_{X-V*}^X\!\circ\!\io_{SV*}^{X-V}$,
$$ H_{\fc}(SV;\Z)_{X,X}=
\big\{A_{SV}\!\in\!H_{\fc}(SV;\Z)\!:\,\io_{SV*}^{X-V}(A_{SV})\!\in\!\cR_X^V\big\}.$$
Thus, $\ov{H}_{X,Y}^V$ is the image of the diagonal subgroup under the homomorphism
$$H_1(V;Z)_X\oplus H_1(V;Z)_X \lra H_1(V;Z)_X, \qquad 
\big([\ga_1],[\ga_2]\big)=\big[\ga_1\!-\!\ga_2\big].$$
Therefore,  $H_{X,X}^V\!=\!H_X^V$ and the claim follows from 
Corollary~\ref{cRXYV_crl2}\ref{cRXYV2H1_it}.
\end{proof}

\begin{eg}\label{ESrt_eg}
Let $\wh\P^2_9$ be a rational elliptic surface and $F\!\subset\!\wh\P^2_9$ be a smooth fiber
as in Example~\ref{EllFib_eg}.
For the standard identification of~$F$ in two copies of~$\wh\P^2_9$,
the homomorphism~\eref{H1vsRXYV_e} can be written~as
$$\Z^2\lra \Z^2\oplus\Z^2, \qquad (a,b)\lra\big((a,b),(a,b)\big).$$
Thus, $\cR_{\wh\P^2_9,\wh\P^2_9}^F\!\approx\!\Z^2$;
this also follows from Corollary~\ref{cRXYV_crl3}.
\end{eg}

\begin{eg}\label{P1Vrim_eg2}
Let $F$, $X$, and $F_0,F_{\i}\!\subset\!X$ be as in Example~\ref{P1Vrim_eg} and $Y\!=\!X$.
For the standard identification of the copies of $F_0\!\cup\!F_{\i}$ in~$X$ and in~$Y$,
the homomorphism~\eref{H1vsRXYV_e}  can be written~as
$$H_1(F;\Z)\!\oplus\!H_1(F;\Z)\lra H_1(F;\Z)\!\oplus\!H_1(F;\Z), \quad
(\ga_0,\ga_{\i})\lra (\ga_0\!-\!\ga_{\i},\ga_0\!-\!\ga_{\i}).$$
Thus, $\cR_{X,X}^{F_0\cup F_{\i}}\!\approx\!H_1(F;\Z)$ in this case,
with the isomorphism induced by the homomorphism
$$H_1(F;\Z)\!\oplus\!H_1(F;\Z) \lra H_1(F;\Z), \qquad
(\al,\be)\lra \al\!-\!\be\,.$$
For an arbitrary identification of the two copies of~$F_0\!\cup\!F_{\i}$, the above homomorphism becomes
$$H_1(F;\Z)\!\oplus\!H_1(F;\Z)\lra H_1(F;\Z)\!\oplus\!H_1(F;\Z), \quad
(\ga_0,\ga_{\i})\lra (\ga_0\!-\!\ga_{\i},\ga_0\!-\!\phi_*\ga_{\i}),$$
for some diffeomorphism $\phi\!:F\!\lra\!F$.
For example, if $F\!=\!\T^2$ is the two-torus, $\cR_{X,X}^V\!\approx\!\Z^2$ for
the standard identification, but  $\cR_{X,X}^{F_0\cup F_{\i}}$ 
can be $\Z$ or~$\{0\}$ for other identifications.
\end{eg}

\begin{eg}\label{SympBlrim_eg}
Suppose $X$ is an oriented manifold and $Z\!\subset\!X$ is a compact submanifold so
that the normal bundle $\cN_XZ$ admits a complex structure.
Fix a complex structure in~$\cN_XZ$ and an identification of 
the unit disk bundle~$D(\cN_XZ)$ of~$\cN_XZ$ with a neighborhood of~$Z$ in~$X$.
Let 
$$\P_XZ=\P(\cN_XZ\!\times\!\C)=\P(\cN_XZ\oplus\!Z\!\times\!\C), \qquad 
V=\P(\cN_XZ)\subset \P(\cN_XZ\!\times\!\C),$$
and $\Bl_ZX$ be the manifold obtained from~$X$ by replacing $D(\cN_XZ)\!\subset\!X$
with the disk bundle of the complex tautological line bundle $\ga\!\lra\!V$
(which has the same boundary consisting of the unit vectors in~$\cN_XV$).
Thus,  
$$\cN_{\Bl_ZX}V=\ga, \qquad \cN_{\P_XZ}V=\ga^*, 
\quad\hbox{and}\quad X=\Bl_ZX\#_{\vph}\P_XZ\,,$$
for an orientation-reversing diffeomorphism $\vph\!:S_{\Bl_ZX}V\!\lra\!S_{\P_XZ}V$
induced by the canonical isomorphism $\ga\!\otimes\!\ga^*\!=\!V\!\times\!\C$
(e.g.~$\{\vph(v)\}v\!=\!1$ for all $v\!\in\!S_{\Bl_ZX}V$).
By Corollary~\ref{cRXYV_crl}\ref{cRXYVtriv_it0} and Example~\ref{Proj_eg},
$$\cR_{\Bl_ZX,\P_XZ}^V=\{0\}\,,$$
i.e.~there are no rim tori in this case.
A geometric reasoning for this conclusion is given in the proof of \cite[Lemma~2.11]{LR}.
If $(X,\om)$ is a symplectic manifold and $Z\!\subset\!X$ is a symplectic submanifold,
the construction of \cite[Section~7.1]{MS1} endows $\Bl_ZX$ with 
a symplectic form~$\om_{Z,\ep}$;
$(\Bl_ZX,\om_{Z,\ep})$ is then called a \textsf{symplectic blowup of~$(X,\om)$ along~$Z$}.
\end{eg}

\subsection{Changes in rim tori}
\label{gluerimtori_subs}

\noindent
We continue with the setup of Lemma~\ref{cRXYV_lmm}.
If $U\!\subset\!X$ and $W\!\subset\!Y$ are closed submanifolds of codimension~$\fc$ 
disjoint from~$V$, then $U,W\!\subset\!X\!\#_{\vph}\!Y$ are also disjoint submanifolds
of codimension~$\fc$.
Below we relate the rim tori for $(X\!\#_{\vph}\!Y,U\!\cup\!W)$
to rim tori in~$X$ and~$Y$ and to the vanishing cycles in~$X\!\#_{\vph}\!Y$.
This relation is described by the squares of exact sequences in 
Figures~\ref{gluerimtori_fig} and~\ref{gluerimtori_fig2}.
Such a relation is needed to make sense of the rim tori refinement to
the symplectic sum formula for relative GW-invariants suggested by \cite[(12.7)]{IPsum}
and of the convolution product on rim tori covers appearing above \cite[(11.5)]{IPsum};
see \cite[Section~5.2]{GWsumIP} for details.\\

\noindent
With notation as in~\eref{HfcSVdfn_e}, let
$$\obu{H}_{X,Y}^{SV}=\big\{\big(\io_{SV*}^{X-V}(A_{SV}),-\io_{SV*}^{Y-V}(A_{SV})\big)\!:
A_{SV}\!\in\!H_{\fc}(SV;\Z)_{X,Y}\big\}
\subset \cR_X^V\!\oplus\!\cR_Y^V.$$
With $U$ and $W$ as in the previous paragraph fixed, define
\begin{alignat}{1}
\notag
\wt{H}_{X,Y}^{SV}&=\big\{\big(\io_{SV*}^{X-U\cup V}(A_{SV}),-\io_{SV*}^{Y-V\cup W}(A_{SV})\big)\!:
A_{SV}\!\in\!H_{\fc}(SV;\Z)_{X,Y}\big\},\\
\label{oriHdfn_e}
\ori{H}_{X,Y}^{SV}&=\wt{H}_{X,Y}^{SV} \cap \cR_{X-V}^U\!\oplus\!\cR_{Y-V}^W
\subset \cR_X^{U\cup V}\!\oplus\!\cR_Y^{V\cup W},\\
\notag
\wt\cR_{X\#_{\vph}Y}^{U\cup W}&=
\big\{\io^{X\#_{\vph}Y-U\cup W}_{X-U\cup V*}(A_X)\!+\!
\io^{X\#_{\vph}Y-U\cup W}_{Y-V\cup W*}(A_Y)\!:\,
(A_X,A_Y)\!\in\!\cR_X^{U\cup V}\!\oplus\!\cR_Y^{V\cup W}\big\}.
\end{alignat}
By Proposition~\ref{gluerimtori_prp} below and~\eref{cRXYVdfn_e}, 
$$\wt\cR_{X\#_{\vph}Y}^{U\cup W}=\ker\big\{q_{\vph*}\!\circ\!\io_{X\#_{\vph}Y-U\cup W*}^{X\#_{\vph}Y}
\!:\,H_{\fc}(X\!\#_{\vph}\!Y\!-\!U\!\cup\!W;\Z)\!\lra\!H_{\fc}(X\!\cup_V\!Y;\Z)\big\}.$$

\begin{prp}\label{gluerimtori_prp}
Suppose $X$ and $Y$ are manifolds, $U,V\!\subset\!X$ and $V,W\!\subset\!Y$ are closed
disjoint submanifolds of the same codimension~$\fc$, and
$\vph\!:S_XV\!\lra\!S_YV$ is a diffeomorphism commuting with the projections to~$V$.
Then, the homology homomorphisms induced by inclusions give rise to the commutative square 
of short exact sequences of Figure~\ref{gluerimtori_fig}.
\end{prp}

\begin{figure}
\begin{gather*}
\xymatrix{& 0\ar[d]& 0\ar[d]&&&& 0\ar[d]& \\
0 \ar[r]& {\ori{H}}_{X,Y}^{SV} \ar[r] \ar[d]& \cR_{X-V}^U\!\oplus\!\cR_{Y-V}^W 
\ar[rrrr]^>>>>>>>>>>>>>>>>>>>{\io^{X\#_{\vph}Y-U\cup W}_{X-U\cup V*}+\io^{X\#_{\vph}Y-U\cup W}_{Y-V\cup W*}}  
\ar[d]&&&& 
\cR_{X\#_{\vph}Y}^{U\cup W}  \ar[r]\ar[d]& 0\\
0 \ar[r]& \wt{H}_{X,Y}^{SV} \ar[r]\ar[d]& \cR_X^{U\cup V}\!\oplus\!\cR_Y^{V\cup W} 
\ar[rrrr]^>>>>>>>>>>>>>>>>>>>{\io^{X\#_{\vph}Y-U\cup W}_{X-U\cup V*}+\io^{X\#_{\vph}Y-U\cup W}_{Y-V\cup W*}}
\ar[d]|{\io_{X-U\cup V*}^{X-V}\oplus\io_{Y-V\cup W*}^{Y-V}}  &&&& 
\wt\cR_{X\#_{\vph}Y}^{U\cup W}  \ar[r]\ar[d]^{\io_{X\#_{\vph}Y-U\cup W*}^{X\#_{\vph}Y}}& 0\\
0 \ar[r]& {\obu{H}_{X,Y}^{SV}} \ar[r] \ar[d]& \cR_X^V\!\oplus\!\cR_Y^V 
\ar[rrrr]^{\io^{X\#_{\vph}Y}_{X-V*}+\io^{X\#_{\vph}Y}_{Y-V*}}  \ar[d]&&&&
\cR_{X,Y}^V  \ar[r]\ar[d]& 0\\
& 0& 0&&&& 0&}
\end{gather*}
\caption{Rim tori under gluing: general case.}
\label{gluerimtori_fig}
\end{figure}

\begin{proof}
The commutativity of all four squares is immediate. 
The left column is exact by~\eref{oriHdfn_e} and~\eref{cRXVdfn_e},
while the middle column is exact because the right column in Figure~\ref{compRT_fig} is
(which does not require the additional assumptions of Corollary~\ref{rimtori_crl2}).
The bottom row is exact by  Corollary~\ref{cRXYV_crl}\ref{cRXYVgen_it0}.
The middle row is exact at $\wt{H}_{X,Y}^{SV}$ and $\wt\cR_{X\#_{\vph}Y}^{U\cup W}$
by the definitions of the two modules.
It is exact at $\cR_X^{U\cup V}\!\oplus\!\cR_Y^{V\cup W}$ 
by the exactness of the Mayer-Vietoris sequence for 
$X\!\#_{\vph}\!Y$ in the proof of Lemma~\ref{cRXYV_lmm} with $(X,Y)$ replaced by $(X\!-\!U,Y\!-\!W)$; 
see the proof of Corollary~\ref{cRXYV_crl}.\\

\noindent 
The top row is exact at~$\ori{H}_{X,Y}^{SV}$ by the definition of this module.
It is exact at $\cR_{X-V}^U\!\oplus\!\cR_{Y-V}^W$ by the same reasoning as
in  the proof of Corollary~\ref{cRXYV_crl}.
In order to establish the exactness of the top row at the last position,
we use the commutative diagram
$$\xymatrix{H_{\fc}(SV)\ar[r] \ar@{=}[d]&
H_{\fc}(X\!-\!U\!\cup\!V)\oplus H_{\fc}(Y\!-\!V\!\cup\!W)
\ar[d]|{\io^{X-V}_{X-U\cup V*}\oplus\io^{Y-V}_{Y-V\cup W*}} \ar[rr] &&
H_{\fc}(X\!\#_{\vph}\!Y\!-\!U\!\cup\!W) \ar[d]^{\io^{X\#_{\vph}\!Y}_{X\#_{\vph}\!Y-U\cup W*}}
\ar[r]^>>>>>>{\de_{\vph}}& H_{\fc-1}(SV) \ar@{=}[d]\\
H_{\fc}(SV)\ar[r] &
H_{\fc}(X\!-\!V)\oplus H_m(Y\!-\!V) 
\ar[rr]^>>>>>>>>>>>>>>>{\io^{X\#_{\vph}\!Y}_{X-V*}\!+\io^{X\#_{\vph}\!Y}_{Y-V*}} &&
H_{\fc}(X\!\#_{\vph}\!Y) \ar[r]^{\de_{\vph}}& H_{\fc-1}(SV) }$$
of the Mayer-Vietoris sequences for $X\!\#_{\vph}\!Y$ and $(X\!-\!U)\!\#_{\vph}\!(Y\!-\!W)$.\\

\noindent
If $A\!\in\!\cR_{X\#_{\vph}Y}^{U\cup W}$, then 
$$\de_{\vph}(A)=\de_{\vph}\big(\io_{X\#_{\vph}Y-U\cup W*}^{X\#_{\vph}Y}(A)\big)=\de_{\vph}(0)=0.$$
By the exactness of the first sequence above, this implies that
\BE{gluerim_e5}A=\io_{X-U\cup V*}^{X\#_{\vph}Y-U\cup W}(A_X)
+\io_{Y-V\cup W*}^{X\#_{\vph}Y-U\cup W}(A_Y)\EE
for some
$$A_X\!\in\!H_{\fc}(X\!-\!U\!\cup\!V;\Z), \quad A_Y\!\in\!H_{\fc}(Y\!-\!V\!\cup\!W;\Z).$$
By the commutativity of the middle square in the Mayer-Vietoris diagram,
\begin{equation*}\begin{split}
&\io^{X\#_{\vph}\!Y}_{X-V*}\big(\io^{X-V}_{X-U\cup V*}(A_X)\big)
+\io^{V\#_{\vph}\!Y}_{Y-V*}\big(\io^{Y-V}_{Y-V\cup W*}(A_Y)\big)\\
&\hspace{1in}=\io^{X\#_{\vph}\!Y}_{X\#_{\vph}\!Y-U\cup W*}
\big(\io_{X-U\cup V*}^{X\#_{\vph}Y-U\cup W}(A_X)\big)
+\io^{X\#_{\vph}\!Y}_{X\#_{\vph}\!Y-U\cup W*}
\big(\io_{Y-V\cup W*}^{X\#_{\vph}Y-U\cup W}(A_Y)\big)\\
&\hspace{1in}
=\io^{X\#_{\vph}\!Y}_{X\#_{\vph}\!Y-U\cup W*}(A)=0.
\end{split}\end{equation*}
Thus,  the exactness of the bottom row implies that 
\BE{gluerim_e7}
\io^{X-V}_{X-U\cup V*}(A_X)=\io^{X-V}_{SV*}(A_{SV}), \quad
\io^{Y-V}_{Y-V\cup W*}(A_Y)=-\io^{Y-V}_{SV*}(A_{SV})\EE
for some $A_{SV}\!\in\!H_{\fc}(SV;\Z)$.
By the commutativity of the left square and~\eref{gluerim_e7}, 
\BE{gluerim_e9}
A_X=A_X'\!+\!\io^{X-U\cup V}_{SV*}(A_{SV}), \qquad 
A_Y=A_Y'\!-\!\io^{Y-V\cup W}_{SV*}(A_{SV})\EE
for some $A_X'\!\in\!\cR_{X-V}^U$ and $A_Y'\!\in\!\cR_{Y-V}^W$.
By~\eref{gluerim_e5} and~\eref{gluerim_e9},
$$A=\io_{X-U\cup V*}^{X\#_{\vph}Y-U\cup W}(A_X')
+\io_{Y-V\cup W*}^{X\#_{\vph}Y-U\cup W}(A_Y'),$$
and so the top row is exact at the last position.\\

\noindent
Finally, the first homomorphism in the right column is the inclusion of a submodule.
This column is exact at the middle and last positions by the commutativity of
the diagram and the exactness of the remaining sequences.
\end{proof}

\begin{figure}
\begin{gather*}
\xymatrix{& 0\ar[d]& 0\ar[d]& 0\ar[d]& \\
0 \ar[r]& {\ori{H}}_{X,Y}^V \ar[r] \ar[d]& H_1(U;\Z)\!\oplus\!H_1(W;\Z) \ar[r]\ar[d]& 
\cR_{X\#_{\vph}Y}^{U\cup W}  \ar[r]\ar[d]& 0\\
0 \ar[r]& {\wt{H}}_{X,Y}^V \ar[r]\ar[d]& H_1(U;\Z)\!\oplus\!H_1(V;\Z)\!\oplus\!H_1(V;\Z)\!\oplus\!H_1(W;\Z)
 \ar[r]\ar[d]& 
\wt\cR_{X\#_{\vph}Y}^{U\cup W}  \ar[r]\ar[d]^{\io_{X\#_{\vph}Y-U\cup W*}^{X\#_{\vph}Y}}& 0\\
0 \ar[r]& {\obu{H}_{X,Y}{\,V}} \ar[r] \ar[d]& H_1(V;\Z)\!\oplus\!H_1(V;\Z) 
\ar[r]\ar[d]& 
\cR_{X,Y}^V  \ar[r]\ar[d]& 0\\
& 0& 0& 0&}
\end{gather*}
\caption{Rim tori under gluing: compact oriented case.}
\label{gluerimtori_fig2}
\end{figure}

\noindent
We now restrict to the settings of Corollaries~\ref{rimtori_crl} and~\ref{cRXYV_crl2}.
Thus, suppose that $X$, $Y$, $U$, $V$, and~$W$ are compact and oriented.
Let 
\begin{gather*}
\ori{H}_{X,Y}^V\subset H_1(U;\Z)\!\oplus\!H_1(W;\Z), \qquad 
\obu{H}_{X,Y}^{\,V} \subset H_1(V;\Z)\!\oplus\!H_1(V;\Z),\\
\hbox{and}\qquad
\wt{H}_{X,Y}^V \subset H_1(U;\Z)\!\oplus\!H_1(V;\Z)\!\oplus\!H_1(V;\Z)\!\oplus\!H_1(W;\Z)
\end{gather*}
be the preimages of 
$$\ori{H}_{X,Y}^{SV}\subset \cR_{X-V}^U\!\oplus\!\cR_{Y-V}^W\,,\qquad
H_{X,Y}^{SV}\subset\cR_X^V\!\oplus\!\cR_Y^V, \quad\hbox{and}\quad
\wt{H}_{X,Y}^V \subset \cR_X^{U\cup V}\!\oplus\!\cR_Y^{V\cup W},$$
respectively, under the homomorphisms as in Corollary~\ref{rimtori_crl}.
The commutative square of short exact sequences of Figure~\ref{gluerimtori_fig}
then induces the commutative square of short exact sequences of Figure~\ref{gluerimtori_fig2}.

\begin{eg}\label{S2T2sum_eg2}
Let $X\!=\!\wh\P^2_9$ be the rational elliptic surface of Examples~\ref{EllFib_eg}
and~\ref{ESrt_eg} with smooth fiber $F\!\subset\!\wh\P^2_9$, 
$Y\!=\!\P^1\!\times\!\T^2$, and $F_0,F_{\i}\!\subset\!\P^1\!\times\!\T^2$
be as in Examples~\ref{P1Vrim_eg} and~\ref{P1Vrim_eg2}.
We take $U\!=\!\eset$, $W\!=\!F_{\i}$, $V\!=\!F\!\subset\!X$, and $V\!=\!F_0\!\subset\!Y$.
In this case, $\fc\!=\!2$, the homomorphisms
$$\io_{V*}^X\!: H_2(V;\Z)\lra H_2(X;\Z) \qquad\hbox{and}\qquad
\io_{V*}^Y\!: H_2(V;\Z)\lra H_2(Y;\Z)$$
are injective, and the homomorphisms
\begin{gather*}
\De_X^V\!: H_1(V;\Z)\lra H_2(SV;\Z)_{X,Y}, \qquad
\io_{SV*}^{Y-V\cup W}\!\!: H_2(SV;\Z)_{X,Y}\lra H_2(Y\!-\!V\!\cup\!W;\Z),\\
\hbox{and}\qquad
\io_{SV*}^{X-U\cup V}\!=\!\io_{SV*}^{X-V}\!\!: 
H_2(SV;\Z)_{X,Y}\lra H_2(X\!-\!U\!\cup\!V;\Z)\!=\!H_2(X\!-\!V;\Z), 
\end{gather*}
are isomorphisms.
The exact square in Figure~\ref{gluerimtori_fig2} then specializes 
to the exact square of Figure~\ref{glueRESP1T2_fig}.
The two homomorphisms in the middle column are given~by
$$(0,\ga)\lra (0,0,0,\ga) \qquad\hbox{and}\qquad (0,\ga_1,\ga_2,\ga_3)\lra(\ga_1,\ga_2).$$
The two homomorphisms in the middle row are given~by
$$(\ga_1,\ga_2)\lra (0,\ga_1,\ga_1\!+\!\ga_2,\ga_2) \qquad\hbox{and}\qquad 
(0,\ga_1,\ga_2,\ga_3)\lra\ga_1\!-\!\ga_2\!+\!\ga_3;$$
the restrictions of the first homomorphism to the last component 
and of the second homomorphism to the last two components
correspond to~\eref{P1Vrim_e3} and~\eref{P1Vrim_e5}, respectively.
For the standard identification~$\vph$, $\wh\P^2_9\!=\!\wh\P^2_9\!\#_{\vph}\!(\P^1\!\times\!\T^2)$.
\end{eg}

\begin{figure}
\begin{gather*}
\xymatrix{& 0\ar[d]& 0\ar[d]& 0\ar[d]& \\ 
0 \ar[r]& 0 \ar[r] \ar[d]& 0\!\oplus\!\Z^2 \ar[r]\ar[d]& 
\Z^2  \ar[r]\ar[d]& 0\\
0 \ar[r]& \Z^2\!\oplus\!\Z^2 \ar[r]\ar[d]& 0\!\oplus\!\Z^2\!\oplus\!\Z^2\!\oplus\!\Z^2
 \ar[r]\ar[d]& 
\Z^2  \ar[r]\ar[d]& 0\\
0 \ar[r]& \Z^2\!\oplus\!\Z^2 \ar[r]^{\id} \ar[d]& \Z^2\!\oplus\!\Z^2 \ar[r]\ar[d]& 
0 \ar[r]\ar[d]& 0\\
& 0& 0& 0&}
\end{gather*}
\caption{Rim tori under gluing for 
$(\wh\P^2_9,F)\!=\!(\wh\P^2_9,\eset)\!\#_{\vph}\!(\P^1\!\times\!\T^2,F_{\i})$.}
\label{glueRESP1T2_fig}
\end{figure}

\subsection{Changes in cohomology}
\label{gluecoh_subs}

\noindent
By Lemma~\ref{cRXYV_lmm} and the exactness of~\eref{MSseqX_e},
\BE{cRXYV_e7}\begin{split}
&\ker\big\{q_{\vph*}\!:\,H_m(X\!\#_{\vph}\!Y;\Z)\!\lra\!H_m(X\!\cup_V\!Y;\Z)\big\}\\
&\qquad\hspace{1in}
=\big\{\io^{X\#_{\vph}Y}_{SV*}\!(A_{SV})\!:\,
A_{SV}\!\in\!\ker\{q_V\!:H_m(SV;\Z)\!\lra\!H_m(V;\Z)\}\big\}.
\end{split}\EE
Lemma~\ref{sumcoh_lmm} below, which describes cohomology classes used
as primary inputs for GW-invariants in the symplectic sum formula,
can be seen as the dual of~\eref{cRXYV_e7}.
The analogue of this lemma with field coefficients, which would be sufficient
for the purposes of the symplectic sum formula, follows immediately by 
dualizing~\eref{cRXYV_e} with coefficients in the same field.
Similarly, the proof of Lemma~\ref{sumcoh_lmm} can be viewed as the dual version
of the proof of Lemma~\ref{cRXYV_lmm}, but we include it for the sake of completeness;
like the proof of Lemma~\ref{cRXYV_lmm}, it contains a delicate step.

\begin{lmm}\label{sumcoh_lmm}
If $X$ and $Y$ are manifolds, 
$V\!\subset\!X,Y$ is a closed submanifold, 
$\vph\!:S_XV\!\lra\!S_YV$ is a diffeomorphism commuting with the projections to~$V$, 
and $q_{\vph}\!:X\!\#_{\vph}\!Y\lra X\!\cup_V\!Y$ is a collapsing map, then 
\BE{sumcoh_e}\big\{q_{\vph}^*\al_{\cup}\!:\,\al_{\cup}\!\in\!H^*(X\!\cup_V\!Y;\Z)\big\}
=\big\{\al_{\#}\!\in\!H^*(X\!\#_{\vph}Y;\Z)\!:\,
\al_{\#}|_{SV}\in q_V^*(H^*(V;\Z))\big\},\EE
where $SV\!\subset\!X\!\#_{\vph}Y$ is the sphere bundle $S_XV\!\approx\!S_YV$.
\end{lmm}

\begin{proof} The commutativity of the diagram
$$\xymatrix{SV \ar[r]^{\io_{SV}^{X\#_{\vph}Y}} \ar[d]^{q_V} & X\!\#_{\vph}\!Y \ar[d]^{q_{\vph}}\\
V \ar[r]^{\io_V^{X\cup_VY}} & X\!\cup_V\!Y }$$
implies that the left-hand side of \eref{sumcoh_e} is contained in the right-hand side. 
Below we confirm the opposite inclusion.\\

\noindent
We will use the commutative diagram of
the Mayer-Vietoris cohomology sequences for $X\!\cup_V\!Y$ and
$X\!\#_{\vph}\!Y\!=\!(X\!-\!V)\!\cup_{SV}\!(Y\!-\!V)$,
$$\xymatrix{ H^{m-1}(V)\ar[r]^{\de_{\cup}^*} \ar[d]^{q_V^*}&
H^m(X\!\cup_V\!Y) \ar[d]^{q_{\vph}^*}
\ar[rr]^>>>>>>>>>>>>>{(\io^{X\cup_VY*}_X,\io^{X\cup_VY*}_Y)}&&
H^m(X)\oplus H^m(Y) 
\ar[d]|{\io^{X*}_{X-V}\oplus\io^{Y*}_{Y-V}}
\ar[rr]^>>>>>>>>>>>>>>{\io^{X*}_V-\io^{Y*}_V} &&
H^m(V) \ar[d]^{q_V^*}\\
H^{m-1}(SV)\ar[r]^{\de_{\vph}^*} &
H^m(X\!\#_{\vph}\!Y) 
\ar[rr]^>>>>>>>>>>{(\io^{X\#_{\vph}\!Y*}_{X-V},\io^{X\#_{\vph}\!Y*}_{Y-V})}&& 
{}~H^m(X\!-\!V)\!\oplus\! H^m(Y\!-\!V) 
\ar[rr]^>>>>>>>>>{\io^{X-V*}_{SV}\!-\io^{Y-V*}_{SV}} &&
H^m(SV)}$$
where $H^*$ denotes integral cohomology groups.
Suppose 
$$\al_{\#}\in H^*(X\!\#_{\vph}Y;\Z),\quad \al_V\in H^*(V;\Z), \quad
\al_{\#}|_{SV}=q_V^*\al_V.$$
By Mayer-Vietoris for $M\!=\!(M\!-\!V)\!\cup_{SV}V$, where $M\!=\!X,Y$,
$$H^m(M;\Z) \stackrel{(\io_{M-V}^{M*},\io_V^{M*})}{\xra{1.5}} 
H^m(M\!-\!V;\Z)\oplus H^m(V;\Z) 
\stackrel{\io^{M-V*}_{SV}-q_V^*}{\xra{2}} H^m(SV;\Z),$$
there exist $\al_X\!\in\!H^m(X;\Z)$ and $\al_Y\!\in\!H^m(Y;\Z)$ such that 
\BE{sumcoh_e5}\al_X|_{X-V}=\al_{\#}|_{X-V},\quad \al_Y|_{Y-V}=\al_{\#}|_{Y-V},  \quad
\al_X|_V,\al_Y|_V=\al_V\,.\EE
By the last equality in~\eref{sumcoh_e5} and the Mayer-Vietoris sequence for $X\!\cup_V\!Y$ above, 
there exists 
\BE{sumcoh_e7}\al_{\cup}\in H^m(X\!\cup_V\!Y;\Z) \qquad\hbox{s.t.}\quad
\al_{\cup}|_X=\al_X, \qquad \al_{\cup}|_Y=\al_Y.\EE
By the commutativity of the middle square in the above diagram,
the two equalities in~\eref{sumcoh_e7}, and
the first two equalities in~\eref{sumcoh_e5}, 
$$\big(\al_{\#}\!-\!q_{\vph}^*\al_{\cup}\big)\big|_{X-V}=0 \quad\hbox{and}\quad
\big(\al_{\#}\!-\!q_{\vph}^*\al_{\cup}\big)\big|_{Y-V}=0\,.$$
Along with the exactness of the bottom row, this implies that 
$$\al_{\#}\!-\!q_{\vph}^*\al_{\cup}\in 
\big\{\de_{\vph}^*(\be_{SV})\!:\,\be_{SV}\!\in\!H^{m-1}(SV;\Z)\big\}.$$
The claim then follows from the observation that 
\BE{sumcoh_e9} \big\{\de_{\vph}^*(\be_{SV})\!:\,\be_{SV}\!\in\!H^{m-1}(SV;\Z)\big\}
\subset \big\{q_{\vph}^*(\al_{\cup})\!:\,\al_{\cup}\!\in\!H^m(X\!\cup_V\!Y;\Z)\big\},\EE
which is established below.\\

\noindent
Choose an open subset $\wt{Y}$ of $X\!\cup_V\!Y$
consisting of~$Y$ and a tubular neighborhood of~$V$ in~$X$.
Let $\cS_{\cup}^*$ denote the cochain complex of $\Z$-valued homomorphisms
on the sub-complex of singular chains generated by simplicies in $X\!\cup_V\!Y$
with images in either $X\!-\!V$ or~$\wt{Y}$.
Similarly, let $\cS_{\#}^*$ denote the cochain complex of $\Z$-valued homomorphisms
on the sub-complex of singular chains generated by simplicies in $X\!\#_{\vph}\!Y$
with images in either $q_{\vph}^{-1}(X\!-\!V)$ or $q_{\vph}^{-1}(\wt{Y})$.
By \cite[Section~5.32]{Wa}, the restriction homomorphisms from the usual singular cochain 
complexes,
$$\cS^*(X\!\cup_V\!Y)\lra \cS_{\cup}^* \qquad\hbox{and}\qquad
\cS^*(X\!\#_{\vph}\!Y)\lra \cS_{\#}^*\,,$$
induce isomorphisms in cohomology.
Thus, we can replace the domains of these homomorphisms by their targets
in order to verify~\eref{sumcoh_e9}.
Let $V_{\cup}\!=\!(X\!-\!V)\!\cap\!\wt{Y}$ and $SV_{\#}\!=\!q_{\vph}^{-1}(V_{\cup})$.\\

\noindent
For any $\eta\!\in\!\cS^*(SV_{\#})$, define
\begin{alignat*}{2}
\eta_{q_{\vph}^{-1}(X-V)}&\in\cS^*\big(q_{\vph}^{-1}(X\!-\!V)\big)\,, &\qquad
\eta_{q_{\vph}^{-1}(X-V)}(\si)&= 
\begin{cases} \eta(\si),&\hbox{if}~\Im\,\si\,\subset SV_{\#};\\ 0,&\hbox{otherwise};
\end{cases}\\
\eta_{\#}&\in\cS_{\#}^*\,, &\qquad
\eta_{\#}(\si)&= 
\begin{cases} \eta_{q_{\vph}^{-1}(X-V)}(\partial\si),
&\hbox{if}~\Im\,\si\,\subset q_{\vph}^{-1}(X\!-\!V);\\
0,&\hbox{if}~\Im\,\si\,\subset q_{\vph}^{-1}(\wt{Y});
\end{cases}\\
\eta_{\cup}&\in\cS_{\cup}^*\,, &\qquad
\eta_{\cup}(\si)&=
\begin{cases} \eta_{q_{\vph}^{-1}(X-V)}(\partial(q_{\vph}^{-1}\!\circ\!\si)),
&\hbox{if}~\Im\,\si\,\subset X\!-\!V;\\
0,&\hbox{if}~\Im\,\si\,\subset\wt{Y};
\end{cases}
\end{alignat*}
where $\si$ denotes an appropriate singular simplex.
The homomorphisms $\eta_{\#}$ and $\eta_{\cup}$ are well-defined on the overlaps
if $\de\eta\!=\!0$, i.e.~$\eta$ determines an element $[\eta]$ in $H^*(SV_{\#})$.
In such a case, 
$$\de_{\vph}^*[\eta]=[\eta_{\#}], \qquad q_{\vph}^*[\eta_{\cup}]=[\eta_{\#}],$$
by the construction of the connecting homomorphism in the Snake Lemma and 
the definition of pull-back homomorphisms.
This establishes~\eref{sumcoh_e9}.
\end{proof}

\noindent
In \cite[Section~13]{IPsum}, the cohomology classes on $X\!\#_{\vph}Y$ not contained
in the left-hand side of~\eref{sumcoh_e} are described as {\it cutting through neck}, 
i.e.~$SV\!\subset\!X\!\#_{\vph}Y$.
The next statement makes this terminology precise.

\begin{crl}\label{sumcoh_crl}
Suppose $X$, $Y$, $V$, $\vph$, and $q_{\vph}$ are as in Lemma~\ref{sumcoh_lmm},
$X$ and $Y$ are compact, $\fc$~is the codimension of~$V$ in~$X$ and~$Y$,
and $\al_{\#}\!\in\!H^*(X\!\#_{\vph}\!Y;\Z)$.
Then, $\al_{\#}\!=\!q_{\vph}^*\al_{\cup}$
for some \hbox{$\al_{\cup}\!\in\!H^*(X\!\cup_V\!Y;\Z)$} if and only if
$\PD_{X\#_{\vph}Y}(\al_{\#})$ can be represented by a pseudocycle
\hbox{$f_{\#}\!:Z_{\#}\!\lra\!X\!\#_{\vph}\!Y$} transverse to~$SV$ such that 
$f_{\#}^{-1}(SV)\!=\!f_V^*SV$ for some pseudocycle $f_V\!:Z_V\!\lra\!V$ 
of dimension~$\fc$ less.
\end{crl}

\begin{proof}
Let $f_{\#}\!:Z_{\#}\!\lra\!X\!\#_{\vph}\!Y$ be a pseudocycle representative
for the Poincare dual of~$\al_{\#}$ transverse to~$SV$.
The restriction of~$f_{\#}$ to $f_{\#}^{-1}(SV)$ then represents 
the Poincare dual of $\al_{\#}|_{SV}$.\\

\noindent
(1) If $f_{\#}^{-1}(SV)\!=\!f_V^*SV$ for some pseudocycle $f_V\!:Z_V\!\lra\!V$ of dimension~$\fc$ less, 
$\al_{\#}|_{SV}\!=\!q_V^*\al_V$, where $\al_V\!\in\!H^*(V;\Z)$ is the Poincare dual
of the class represented by~$f_V$.
Lemma~\ref{sumcoh_lmm} then implies that $\al_{\#}\!=\!q_{\vph}^*\al_{\cup}$
for some $\al_{\cup}\!\in\!H^*(X\!\cup_V\!Y;\Z)$.\\

\noindent
(2) If $\al_{\#}\!=\!q_{\vph}^*\al_{\cup}$ for some $\al_{\cup}\!\in\!H^*(X\!\cup_V\!Y;\Z)$,
then $\al_{\#}|_{SV}\!=\!q_V^*\al_V$ for some $\al_V\!\in\!H^*(V;\Z)$;
see Lemma~\ref{sumcoh_lmm}.
Let $f_V\!:Z_V\!\lra\!V$ be a pseudocycle representing the Poincare dual of~$\al_V$
and $\wt{f}_V\!:f_V^*SV\!\lra\!SV$ be the induced pseudocycle from the total space
of the bundle $SV\!\lra\!V$ pulled back by~$f_V$; see the end of Section~\ref{cuthom_subs}.
Thus, there exists a pseudocycle equivalence $\wt{f}\!:\wt{Z}\!\lra\!SV$ so~that 
$$\prt\wt{f}=f_{\#}\big|_{f_{\#}^{-1}(S_V)}-\wt{f}_V\,.$$
Cutting $Z_{\#}$ along the hypersurface $f_{\#}^{-1}(SV)$,
gluing in~$\wt{f}$ and~$-\wt{f}$ along the resulting cuts, 
identifying $\wt{f}$ and~$-\wt{f}$ along~$\wt{f}_V$, and
moving $\pm\wt{f}$ on the complement of~$\wt{f}_V$ outside~$SV$,
we obtain a pseudocycle representative $\wh{f}_{\#}\!:\wh{Z}_{\#}\!\lra\!X\!\#_{\vph}\!Y$ 
for the Poincare dual of~$\al_{\#}$ transverse to~$SV$ such that 
\hbox{$\wh{f}_{\#}^{-1}(SV)\!=\!f_V^*SV$}. 
\end{proof}

\begin{rmk}\label{Topol_rmk}
There is a slight misstatement in part~(a) at the bottom of page~996 in \cite{IPsum}
related to the $\fc\!=\!2$ case of Lemma~\ref{sumcoh_lmm},
since the first map in \cite[(10.13)]{IPsum}
is never injective for dimensional reasons. 
The statement in~(a) should instead be that 
$\al\!\in\!H^m(Z_{\la};\Z)$ separates if 
$$\cup c_1\!: H^{m-1}(V;\Z)\lra H^{m+1}(V;\Z)$$ 
is injective.
In~(b), $j^*\!:H^*(Z_{\la};\Z)\!\lra\!H^*(SV;\Z)$ is the restriction map.
\end{rmk}

\section{Abelian covers of topological spaces}
\label{AbCov_sec}

\noindent
The notation for the abelian covers of topological spaces relevant
in our context is introduced in Section~\ref{AbCovNotat_subs}.
Section~\ref{AbCovProper_subs} is concerned with their topological properties,
focusing on whether their (co)homology is finitely generated or~not.\\

\noindent
By a \sf{topological space}~$V$, we will mean a \sf{locally path-connected} and 
a \sf{semilocally simply connected} topological space~$V$, as in \cite[\S25,82]{Mu};
all manifolds and more generally CW-complexes fall in this category.
The first assumption implies that the connected and path-connected components
of~$V$ are the same; see \cite[Theorem~25.5]{Mu}.
The two assumptions together imply 
the connected covers of the connected components~$V_r$ of~$V$
 are classified by their fundamental subgroups~$\pi_1(V_r)$;
see \cite[Theorems~79.4,~82.1]{Mu}.

\subsection{Notation and examples}
\label{AbCovNotat_subs}

\noindent
Let $\Z_{\pm}\!\subset\!\Z$ denote the nonzero integers.
For a tuple $\bs\!=\!(s_1,\ldots,s_{\ell})\in\Z_{\pm}^{\,\ell}$ with $\ell\!\in\!\Z^{\ge0}$,
we denote by $\gcd(\bs)$ the greatest common divisor of~$s_1,\ldots,s_{\ell}$;
if $\ell\!=\!0$, we set $\gcd(\bs)\!=\!0$.\\

\noindent
Let $V$ be a topological space. 
For any submodule $H\!\subset\!H_1(V;\Z)$, let
\BE{cRH} q_H\!: H_1(V;\Z)\lra \cR_H\equiv \frac{H_1(V;\Z)}{H}\EE
be the projection to the corresponding quotient module.
If $V_1,\ldots,V_N$ are the topological components of~$V$,
\hbox{$\ell_1,\ldots,\ell_N\!\in\!\Z^{\ge0}$}, and
$\bs_1\!\in\!\Z_{\pm}^{\,\ell_1},\ldots,\bs_N\!\in\!\Z_{\pm}^{\,\ell_N}$, then the topological space
$$V_{\bs_1\ldots\bs_N}\equiv V_1^{\ell_1}\!\times\!\ldots\!\times\!V_N^{\ell_N}$$
is connected.\\

\noindent
With $V$ and $\bs_1,\ldots,\bs_N$ as above, define
\begin{gather}\label{PhiVbs_e}
\Phi_{V;\bs_1\ldots\bs_N}\!: 
H_1\big(V_{\bs_1\ldots\bs_N};\Z\big)=\bigoplus_{r=1}^N\! H_1(V_r;\Z)^{\oplus\ell_r} 
\lra H_1(V;\Z),\\ 
\notag
\Phi_{V;\bs_1\ldots\bs_N}
\big((\ga_{r;i})_{i\le\ell_r,r\le N}\big)= \sum_{r=1}^N\sum_{i=1}^{\ell_r}s_{r;i}\ga_{r;i}\,.
\end{gather}
For any submodule $H\!\subset\!H_1(V;\Z)$, let
\begin{gather}
\label{Hbsdfn_e}  H_{\bs_1\ldots\bs_N}
=\Phi_{V;\bs_1\ldots\bs_N}^{-1}(H)\subset H_1\big(V_{\bs_1\ldots\bs_N};\Z\big),\\
\label{cRbsdfn_e}
\cR_{H;\bs_1\ldots\bs_N}'=\Im\big\{q_H\!\circ\!\Phi_{V;\bs_1\ldots\bs_N}\big\}
\subset\cR_H, \qquad
\cR_{H;\bs_1\ldots\bs_N}=\frac{\cR_H}{\cR_{H;\bs_1\ldots\bs_N}'}
\!\times\!\cR_{H;\bs_1\ldots\bs_N}'\,.
\end{gather}
If $\gcd(\bs_r)\!=\!1$ for every $r\!=\!1,\ldots,N$, then 
$$\cR_{H;\bs_1\ldots\bs_N}'=\cR_{H;\bs_1\ldots\bs_N}=\cR_H\,.$$
If $V$ is connected, then \hbox{$\cR_{H;\bs}'\!=\!\gcd(\bs)\cR_H$}
for any $\bs\!\in\!\Z_{\pm}^{\ell}$ and $H_{(1)}\!=\!H$.\\

\noindent
For each $r\!=\!1,\ldots,N$, let $\wh{V}_r\!\lra\!V_r$ be the maximal abelian cover of~$V_r$,
i.e.~the covering projection corresponding to the commutator subgroup of~$\pi_1(V)$.
The group of deck transformations of this regular covering is $H_1(V_r;\Z)$.
The maximal abelian cover of~$V_{\bs_1\ldots\bs_N}$ is given~by
\BE{whVbs_e}
\wh{V}_{\bs_1\ldots\bs_N}\equiv\prod_{r=1}^N\wh{V}_r^{\ell_r}\lra V_{\bs_1\ldots\bs_N}\,;\EE
there is a natural action of $H_1(V_{\bs};\Z)$ on this space.
For any submodule $H\!\subset\!H_1(V;\Z)$, let
\BE{HbsCov_e}\begin{split}
\pi_{H;\bs_1\ldots\bs_N}'\!:\wh{V}_{H;\bs_1\ldots\bs_N}'&
\equiv\wh{V}_{\bs_1\ldots\bs_N}\big/H_{\bs_1\ldots\bs_N} \lra V_{\bs_1\ldots\bs_N}\,,\\
\pi_{H;\bs_1\ldots\bs_N}\!:\wh{V}_{H;\bs_1\ldots\bs_N}&
\equiv \frac{\cR_H}{\cR_{H;\bs_1\ldots\bs_N}'}\!\times\!\wh{V}_{H;\bs_1\ldots\bs_N}' 
\lra V_{\bs_1\ldots\bs_N}\,.
\end{split}\EE
We will write elements of the second covering as
\BE{whVelem_e}\big([\ga]_{H;\bs_1\ldots\bs_N},[\wh{x}]_H\big)
\in  \frac{\cR_H}{\cR_{H;\bs_1\ldots\bs_N}'}\!\times\!\wh{V}_{H;\bs_1\ldots\bs_N}' ,\EE
with the first component denoting the image of $\ga\!\in\!H_1(V;\Z)$
under the homomorphism
$$H_1(V;\Z)\lra \cR_H\lra \frac{\cR_H}{\cR_{H;\bs_1\ldots\bs_N}'}$$
and the second component denoting the image of $\wh{x}\!\in\!\wh{V}_{\bs_1\ldots\bs_N}$.\\

\noindent
The groups of deck transformations of these regular coverings are
\BE{HbsDeck_e}
\Deck\big(\pi_{H;\bs_1\ldots\bs_N}'\big)=\cR_{H;\bs_1\ldots\bs_N}' 
\qquad\hbox{and}\qquad
\Deck\big(\pi_{H;\bs_1\ldots\bs_N}\big)=\cR_{H;\bs_1\ldots\bs_N},\EE 
respectively.
The action of $\cR_{H;\bs_1\ldots\bs_N}'$ is induced from 
the default action of $H_1(V_{\bs};\Z)$ on~$\wh{V}_{\bs}$
via the surjective homomorphism
$$q_H\!\circ\!\Phi_{V;\bs_1\ldots\bs_N}\!:H_1(V_{\bs};\Z)\lra \cR_{H;\bs_1\ldots\bs_N}';$$
the kernel of this homomorphism, i.e.~$H_{\bs_1\ldots\bs_N}$, acts trivially 
on~$\wh{V}_{H;\bs_1\ldots\bs_N}'$.\\ 

\noindent
If $V$ is connected, then 
$$\pi_H\!\equiv\!\pi_{H;(1)}\!: \wh{V}_H\!\equiv\!\wh{V}_{H;(1)}\!=\!\wh{V}/H\lra V$$
is the abelian covering corresponding to the subgroup $H\!\subset\!H_1(V;\Z)$,
i.e.~the covering corresponding to the normal subgroup $\Hur^{-1}(H)\!\subset\!\pi_1(V)$,
where
$$\Hur\!:\,\pi_1(V_{\bs_1\ldots\bs_N})\lra H_1(V_{\bs_1\ldots\bs_N};\Z)$$
is Hurewicz homomorphism; see  \cite[Section~7.4]{Sp}.\\

\noindent
A collection $\{\ga_j\}\!\subset\!H_1(V;\Z)$ of representatives for the elements of 
$\cR_H/\cR_{H;\bs_1\ldots\bs_N}'$ induces a homomorphism
$$H_1(V;\Z)\lra \Deck(\pi_{H;\bs_1\ldots\bs_N}\big), \qquad 
\eta\lra\Th_{\eta}\,,$$
as follows.
For every $\eta\!\in\!H_1(V;\Z)$ and a coset representative~$\ga_j$,
let $\ga_j(\eta)$ be the unique coset representative from the chosen collection
such~that 
\BE{gajeta_e} \ga_j+\eta-\ga_j(\eta)-\Phi_{V;\bs_1\ldots\bs_N}(\eta_j)\in H\EE
for some $\eta_j\!\in\!H_1(V_{\bs_1\ldots\bs_N};\Z)$.
Define 
$$\Th_{\eta}\!:\wh{V}_{H;\bs_1\ldots\bs_N}\lra \wh{V}_{H;\bs_1\ldots\bs_N}, \quad
\Th_{\eta}\big([\ga_j]_{H;\bs_1\ldots\bs_N},[\wh{x}]_H\big)
=\big([\ga_j(\eta)]_{H;\bs_1\ldots\bs_N},[\eta_j\!\cdot\!\wh{x}]_H\big).$$
Since \eref{gajeta_e} determines $\eta_j$ up to an element of $H_{\bs_1\ldots\bs_N}$,
the last component of~$\Th_{\eta}$ is well-defined.\\

\noindent
Suppose in addition that $V'\!\subset\!V$ is the union of $V_1,\ldots,V_{N'}$ for some $N'\!\le\!N$
and $H'\!\subset\!H_1(V';\Z)$ is a submodule.
Let
\begin{alignat*}{2}
q\!:V_{\bs_1\ldots\bs_N}&\lra V'_{\bs_1\ldots\bs_{N'}}, &\qquad
(x_{r;i})_{i\le\ell_r,r\le N}&\lra (x_{r;i})_{i\le\ell_r,r\le N'},\\
\wh{q}\!: \wh{V}_{\bs_1\ldots\bs_N}&\lra \wh{V'}_{\bs_1\ldots\bs_{N'}}, &\qquad
(\wh{x}_{r;i})_{i\le\ell_r,r\le N}&\lra (\wh{x}_{r;i})_{i\le\ell_r,r\le N'}
\end{alignat*}
denote the projections to the $V'$- and $\wh{V'}$-components.
If $H'$ contains the image of~$H$ under the projection
\BE{H1proj_e}q_*\!:H_1(V;\Z)=\bigoplus_{r=1}^NH_1(V_r;\Z) \lra H_1(V';\Z)=\bigoplus_{r=1}^{N'}H_1(V_r;\Z),\EE
then $q_*$ induces a commutative diagram of homomorphisms (not exact sequences)
$$\xymatrix{H_{\bs_1\ldots\bs_N} \ar@{^(->}[r] \ar[d]^{q_*}& 
H_1(V_{\bs_1\ldots\bs_N};\Z) \ar[rr]^{\Phi_{V;\bs_1\ldots\bs_N}} \ar[d]^{q_*}&& 
H_1(V;\Z) \ar[r]\ar[d]^q& \cR_H \ar[r]\ar[d]& \frac{\cR_H}{\cR_{H;\bs_1\ldots\bs_N}} \ar[d]\\
H'_{\bs_1\ldots\bs_{N'}} \ar@{^(->}[r]& 
H_1(V'_{\bs_1\ldots\bs_{N'}};\Z) \ar[rr]^{\Phi_{V';\bs_1\ldots\bs_{N'}}} &&
H_1(V';\Z) \ar[r]& \cR_{H'} \ar[r]& \frac{\cR_{H'}}{\cR_{H';\bs_1\ldots\bs_{N'}}'} }$$
The continuous map 
\BE{whVforg_e}
\big([\ga]_{H;\bs_1\ldots\bs_N},[\wh{x}]_H\big) 
\lra \big([q_*(\ga)]_{H';\bs_1\ldots\bs_{N'}},[\wh{q}(\wh{x})]_{H'}\big)\EE
then induces a commutative diagram 
$$\xymatrix{ \wh{V}_{H;\bs_1\ldots\bs_N} \ar[d]\ar[rr]^{\wt{q}}&& 
\wh{V'}_{H';\bs_1\ldots\bs_{N'}}\ar[d] \\
V_{\bs_1\ldots\bs_N} \ar[rr]^q&& V'_{\bs_1\ldots\bs_{N'}}}$$
of fiber bundles.

\begin{eg}\label{Tcov_eg}
If $V\!=\!\T^2$, $\ell\!\in\!\Z^+$, and $H\!=\!\{0\}$,
then
$$\wh{V}_{H;\bs}=\C\!\times\!\T_{\bs}^{2(\ell-1)}, \quad\hbox{where}\quad
\T_{\bs}^{2(\ell-1)}=\big\{(z_i)_{i\le\ell}\!\in\!\C^{\ell}\!:
\sum_{i=1}^{\ell}\!s_iz_i\in\Z\!\oplus\!\fI\Z\big\} \big/\Z^{2\ell}
\subset \T^{2\ell}\!=\!V_{\bs}.$$ 
The second covering in~\eref{HbsCov_e} can be written~as
\BE{EScover_e}\C\!\times\!\T_{\bs}^{2(\ell-1)}\lra \T^{2\ell}, \qquad
\big(z,[z_i]_{i\le\ell}\big) \lra  \bigg[z_i\!-\!\frac{z}{s_i}\bigg]_{i\le\ell}.\EE
A path $t\!\lra\![\ga_{i'}(t)]_{i'\le\ell}$ in $\T^{2\ell}$  lifts to the~path
$$t\lra \bigg(\frac{1}{\ell}\sum_{i'=1}^{\ell}s_{i'}\ga_{i'}(t),
\bigg[\ga_i(t)+\frac{1}{\ell s_i}\sum_{i'=1}^{\ell}s_{i'}\ga_{i'}(t)\bigg]_{i\le\ell}\bigg)$$
in $\C\!\times\!\T_{\bs}^{2(\ell-1)}$.
Under the standard identification of $H_1(\T^2;\Z)$ with $\Z\!\oplus\!\fI\Z$,
the action of $H_1(\T^2;\Z)^{\oplus\ell}$ on this cover is thus given~by
$$ (\ga_{i'})_{i'\le\ell}\cdot\big(z,[z_i]_{i\le\ell}\big) 
=\bigg(z\!+\!\frac1\ell\sum_{i'=1}^{\ell}s_{i'}\ga_{i'},
\bigg[z_i\!+\!\frac{1}{\ell s_i}\sum_{i'=1}^{\ell}s_{i'}\ga_{i'}\bigg]_{i\le\ell}\bigg)\,.$$
The group of deck transformations of this cover is $\Z_{\gcd(\bs)}^{\,2}\!\oplus\!\gcd(\bs)\Z^2$.
The action of the second component is induced by the action of $H_1(\T^2;\Z)^{\oplus\ell}$ via
the surjective homomorphism 
$$ H_1(\T^2;\Z)^{\oplus\ell}\lra \gcd(\bs)H_1(\T^2;\Z), \qquad
(\ga_{i'})_{i'\le\ell}\lra \sum_{i'=1}^{\ell}s_{i'}\ga_{i'}\,.$$
\end{eg}

\subsection{Some properties}
\label{AbCovProper_subs}

\noindent
We now describe some cases when the (co)homology of the abelian covers $\wh{V}_{H;\bs}$ 
is finitely generated.
As indicated in \cite[Sections~1.2,1.3]{GWsumIP}, the refinement to the usual GW-invariants
suggested in~\cite{IPrel} is more likely to lead to qualitative applications
in the symplectic sum context in such cases.
We continue with the notation of Section~\ref{AbCovNotat_subs}.

\begin{lmm}\label{rimtoriFG_lmm}
Let $V$ be a finite connected CW-complex, $H\!\subset\!H_1(V;\Z)$ be a submodule,
and $\bs\!\in\!\Z_{\pm}^{\ell}$ with $\ell\!\in\!\Z^+$.
If $H_*(\wh{V}_H;\Q)$ is finitely generated, then so is $H_*(\wh{V}_{H;\bs};\Q)$.
\end{lmm}

\begin{proof}
Since $\cR_H/\cR_{H;\bs}'$ is finite, it is sufficient to show that 
$H_*(\wh{V}_{H;\bs}';\Q)$ is finitely generated.
By the Universal Coefficient Theorem \cite[Theorem~53.5]{Mu2},
$H_*$ is finitely generated if and only if $H^*$~is.
Since $H_{\bs/\gcd(\bs)}\!\subset\!H_{\bs}$ and $H\!\subset\!H_1(V;\Z)$
is finitely generated, 
$$q\!: \wh{V}_{H;\bs/\gcd(\bs)}'\!=\!\wh{V}^{\ell}/H_{\bs/\gcd(\bs)}\lra 
\wh{V}_{H;\bs}'\!=\!\wh{V}^{\ell}/H_{\bs}$$
is a finite covering and so the homomorphism
$$q^*\!:\,H^*\big(\wh{V}_{H;\bs}';\Q) \lra H^*\big(\wh{V}_{H;\bs/\gcd(\bs)}';\Q)$$
is injective.
In particular, the claim of the lemma holds if $\ell\!=\!1$.\\

\noindent
Suppose $\ell\!>\!1$. 
Let $\bs'$ denote the tuple consisting of the first $(\ell\!-\!1)$ components of~$\bs$
and 
\BE{Hpr_e} H'=H+s_{\ell\,}H_1(V;\Z)\subset H_1(V;\Z).\EE
The projection $\wh{V}^{\ell}\!\lra\!V$ onto the last component induces 
a fiber bundle
$$q_{\ell}\!: \wh{V}_{H;\bs}'\!=\!\wh{V}^{\ell}/H_{\bs}
\lra V\!=\!\wh{V}/H_1(V;\Z)$$
with fiber $\wh{V}_{H';\bs'}$.
By Serre's Spectral Sequence (e.g.~Theorem~9.2.1, 9.2.17, or 9.3.1 in \cite{Sp}
applied with $\Z_2$-coefficients in the last two cases), 
$H^*(\wh{V}_{H;\bs}';\Q)$ is thus finitely generated if 
$H^*(V;\Q)$ and $H^*(\wh{V}_{H';\bs'}';\Q)$ are finitely generated.
This is the case for $H^*(V;\Q)$ because $V$ is a finite CW-complex.
By induction on~$\ell$, we can assume that this is also the case for $H^*(\wh{V}_{H';\bs'}';\Q)$.
\end{proof}

\begin{rmk}\label{rimtoriFG_rmk}
The statement and proof of Lemma~\ref{rimtoriFG_lmm} can be adapted to a disconnected~$V$.
For each $r\!=\!1,\ldots,N$, let 
$$\cR_{H;r}=q_H\big(H_1(V_r;\Z)\big)  \subset\cR_H\,;$$
these modules span~$\cR_H$.
The first factor in the definition of~$\wh{V}_{H;\bs_1\ldots\bs_N}$ in~\eref{HbsCov_e}
is finite if and only~if the submodule
\BE{cRHr_e2}\wc\cR_{H;\bs_1\ldots\bs_N}\equiv 
\sum_{\begin{subarray}{c}1\le r\le N\\ \ell_r\neq0\end{subarray}}\!\!\! \cR_{H;r}
\subset\cR_H\EE
has finite index. 
This index is finite if $\ell_r\!\neq\!0$ whenever $H_1(V_r;\Q)\!\neq\!\{0\}$
or if $V\!=\!\{0,\i\}\!\times\!F$ for some connected~$F$ and 
$$H\!=\!H_{\De}\subset H_1(V;\Z)=H_1(F;\Z)\oplus H_1(F;\Z)$$
is the diagonal.
If $\wc\cR_{H;\bs_1\ldots\bs_N}$ does not have finite index in~$\cR_H$,
then $H_*(\wh{V}_{H;\bs_1\ldots\bs_N};\Q)$ is clearly not finitely generated.
If the index is finite, $H_*(\wh{V}_{H;\bs_1\ldots\bs_N};\Q)$ is finitely generated.
For the purposes of establishing this statement, $V$ can be replaced by the union of~$V_r$
with $\ell_r\!\neq\!0$ and $H$ by its intersection with the $H_1$ of this subspace.
Thus, we can assume that $\ell_r\!\neq\!0$ for all $r\!=\!1,\ldots,N$.
The proof of Lemma~\ref{rimtoriFG_lmm} then applies by projecting to~$V_N$
and replacing $s_{\ell}H_1(V;\Z)$ in~\eref{Hpr_e} by $s_{N;\ell_N}H_1(V_N;\Z)$ 
if $\ell_N\!\ge\!2$ and by $H_1(V_N;\Z)$ if $\ell_N\!=\!1$ (in this case, $V_N$ no longer
appears in the~fiber).
\end{rmk}

\noindent
We next relate the action of $\Deck(\pi_{H;\bs_1\ldots\bs_N}')$ on the cohomology
of $\wh{V}'_{H;\bs_1\ldots\bs_N}$
to the flux subgroup $\Flux(V)\!\subset\!H_1(V;\Z)$ defined in Section~\ref{appl_subs}.
Let $V_1,\ldots,V_N$ be the connected components of~$V$,
$\bs_1\!\in\!\Z_{\pm}^{\ell_1},\ldots,\bs_N\!\in\!\Z_{\pm}^{\ell_N}$, and
$H\!\subset\!H_1(V;\Z)$ be a submodule.
Define 
\begin{equation*}\begin{split}
\Flux(V)_{H;\bs_1\ldots\bs_N}&=
\big\{q_H\big(\Phi_{V;\bs_1\ldots\bs_N}\big((\ga_{r;i})_{i\le\ell_r,r\le N}\big)\big)\!:
\,\ga_{r;i}\!\in\!\Flux(V_r)~\forall\,i\!\le\!\ell_r,\,r\!\le\!N\big\}\\
&\subset\cR_{H;\bs_1\ldots\bs_N}'\subset \cR_H \,.
\end{split}\end{equation*}

\begin{lmm}\label{rimtoriFlus_lmm}
Let $V$, $V_1,\ldots,V_N$, $\bs_1,\ldots,\bs_N$, and $H$ be as above.
If 
$$\ga\in H_1(V_{\bs_1\ldots\bs_N};\Z)  \qquad\hbox{and}\qquad 
q_H\big(\Phi_{V;\bs_1\ldots\bs_N}(\ga)\big)\in \Flux(V)_{H;\bs_1\ldots\bs_N},$$ 
then the isomorphism
$$\big\{\ga\cdot\big\}^*\!:\,
H^*\big(\wh{V}_{H;\bs_1\ldots\bs_N}';\Z\big)\lra H^*\big(\wh{V}_{H;\bs_1\ldots\bs_N}';\Z\big)$$
is the identity. 
\end{lmm}

\begin{proof} Let 
$$\ga' \equiv (\ga_{r;i}')_{i\le\ell_r,r\le N}\in 
\bigoplus_{r=1}^N\Flux(V_r)^{\oplus\ell_r}$$
be such that 
$q_H(\Phi_{V;\bs_1\ldots\bs_N}(\ga))=q_H(\Phi_{V;\bs_1\ldots\bs_N}(\ga'))$.
Since $\ga'\!-\!\ga\!\in\!H_{\bs_1\ldots\bs_N}$,
the actions of $\ga$ and $\ga'$ on $\wh{V}_{H;\bs_1\ldots\bs_N}'$ are the same
and we can assume that $\ga'\!=\!\ga$.\\

\noindent
For each $r\!=\!1,\ldots,N$ and $i\!=\!1,\ldots,\ell_r$, 
let $\Psi_{r;i;t}\!:V_r\!\lra\!V_r$ be a loop of homeomorphisms generating~$\ga_{r;i}$ 
such~that $\Psi_{r;i;0}\!=\!\id$.
These loops lift to paths of homeomorphisms
\begin{alignat*}{3}
\wh\Psi_{r;i;t}\!: \wh{V}_r&\lra \wh{V}_r, &\quad t&\in[0,1], 
&\quad &\wh\Psi_{r;i;0}=\id_{\wh{V}_r},~~\wh\Psi_{r;i;1}(\wh{x}_{r;i})=\ga_{r;i}\cdot\wh{x}_{r;i},\\
\wt\Psi_t\!: \wh{V}'_{H;\bs_1\ldots\bs_N}&\lra \wh{V}'_{H;\bs_1\ldots\bs_N}, &\quad t&\in[0,1],  
&\quad &\wt\Psi_t\big(\big[(\wh{x}_{r;i})_{i\le\ell_r,r\le N}\big]_H\big)=
\big[\big(\wh\Psi_{r;i;t}(\wh{x}_{r;i})\big)_{i\le\ell_r,r\le N}\big]_H.
\end{alignat*}
Since $\wt\Psi_1\!=\!\ga\cdot$, the homeomorphism $\ga\cdot$ of $\wh{V}'_{H;\bs_1\ldots\bs_N}$
is homotopic to the identity.
This implies the claim.
\end{proof}

\begin{crl}\label{rimtoriFlus_crl}
Let $V$ be a finite CW-complex with connected components $V_1,\ldots,V_N$,
$H\!\subset\!H_1(V;\Z)$ be a submodule, $\ell_1,\ldots\!\ell_N\!\in\!\Z^+$, and 
 $\bs_1\!\in\!\Z_{\pm}^{\ell_1},\ldots,\bs_N\!\in\!\Z_{\pm}^{\ell_N}$.
\begin{enumerate}[label=(\arabic*),leftmargin=*]
\item\label{rimtoriFlus_it1} If the index of $\Flux(V)_H$ in $\cR_H$ is finite, then
 $H^*(\wh{V}'_{H;\bs_1\ldots\bs_N};\Q)$
is finitely generated.

\item\label{rimtoriFlus_it2a} If $\Flux(V)_H\!=\!\cR_H$, then  
$H^*(\wh{V}_{H;\bs_1\ldots\bs_N}';\Q)^{\cR_{H;\bs_1\ldots\bs_N}'}\!=\!H^*(\wh{V}_{H;\bs_1\ldots\bs_N}';\Q)$.

\item\label{rimtoriFlus_it2b} If $\Flux(V)_H\!=\!\cR_H$ and $\rk_{\Z}\cR_H\!\le\!1$, then  
$\pi_{H;\bs_1\ldots\bs_N}'^*H^*(V_{\bs_1\ldots\bs_N};\Q)\!=\!H^*(\wh{V}_{H;\bs_1\ldots\bs_N}';\Q)$.
\end{enumerate}
\end{crl}

\begin{proof}
\ref{rimtoriFlus_it1}
The vector space $H^*(\wh{V}_{H;\bs_1\ldots\bs_N}';\Q)$ is finitely generated over
the group ring~of 
\BE{rimtoriFlus_e1}\Deck\big(\pi_{H;\bs_1\ldots\bs_N}'\big)=\cR_{H;\bs_1\ldots\bs_N}'.\EE
By Lemma~\ref{rimtoriFlus_lmm}, the elements of 
$\Flux(V)_{H;\bs_1\ldots\bs_N}$ act trivially on the cohomology of $\wh{V}'_{H;\bs_1\ldots\bs_N}$.
Thus,  $H^*(\wh{V}_{H;\bs_1\ldots\bs_N}';\Q)$ is finitely generated over
the group ring~of the quotient $\cR_{H;\bs_1\ldots\bs_N}'/\Flux(V)_{H;\bs_1\ldots\bs_N}$.
If the index of $\Flux(V)_H$ in~$\cR_H$ is finite, then the index of
$\Flux(V)_{H;\bs_1\ldots\bs_N}$ in $\cR_{H;\bs_1\ldots\bs_N}'$ is also finite
and so  $H^*(\wh{V}_{H;\bs_1\ldots\bs_N}';\Q)$ is finitely generated over~$\Q$.\\

\noindent
\ref{rimtoriFlus_it2a} If $\Flux(V)_H\!=\!\cR_H$, then 
$\Flux(V)_{H;\bs_1\ldots\bs_N}\!=\!\cR_{H;\bs_1\ldots\bs_N}'$.
By Lemma~\ref{rimtoriFlus_lmm}, $\cR_{H;\bs_1\ldots\bs_N}'$ thus acts trivially on 
$H^*(\wh{V}_{H;\bs_1\ldots\bs_N}';\Q)$.\\

\noindent
\ref{rimtoriFlus_it2b} The last claim of this corollary follows from the second claim 
and  Corollary~\ref{CohSurj_crl} below.
\end{proof}

\noindent
In the remainder of this section, we establish Lemmas~\ref{CohSurj_lmm1} and~\ref{CohSurj_lmm2} below.
They  imply Corollary~\ref{CohSurj_crl}, which is used in the proof of 
Corollary~\ref{rimtoriFlus_crl} above.
The statement and proof of Lemma~\ref{CohSurj_lmm1} are well-known.
As pointed out by M.~Wendt and D.~Ruberman on {\it MathOverflow},
the statement of Lemma~\ref{CohSurj_lmm2} can be obtained either from a spectral sequence 
applied to the Borel construction associated to the $\Z$-covering~$\wt{V}$ or 
from the short exact sequence in the proof of \cite[Assertion~5]{Mi68}.
As we are not aware of any published reference for these statements,
we include their proofs for the sake of completeness.
In the proof of Lemma~\ref{CohSurj_lmm2}, we represent each $\Z$-invariant cohomology class 
on a regular $\Z$-covering  by an explicit cohomology class on 
the associated Borel construction (the  argument suggested by D.~Ruberman is more efficient,
but does not produce such a~cocycle).\\ 

\noindent
For a cochain complex $(C^*,\de)$ with an action of a group~$G$, let 
$$(C^*,\de)^G\equiv \big((C^*)^G,\de|_{(C^*)^G}\big), \qquad\hbox{where}\quad
(C^*)^G=\big\{\eta\!\in\!G^*\!\!:\,g\!\cdot\!\eta\!=\!\eta~\forall\,g\!\in\!G\big\},$$
be the $G$-invariant subcomplex of  $(C^*,\de)$ and
$$H^*\!(C^*,\de)^G\equiv\big\{[\eta]\!\in\!H^*(C^*,\de)\!\!:\,
g\!\cdot\![\eta]\!=\![\eta]~\forall\,g\!\in\!G\big\}$$
be the $G$-invariant part of the cohomology of $(C^*,\de)$.
The inclusion $(C^*,\de)^G$ into $(C^*,\de)$ induces a homomorphism
\BE{CGhomom_m} H^*\big((C^*,\de)^G\big) \lra H^*\!(C^*,\de)^G\,.\EE

\begin{lmm}\label{CohSurj_lmm1}
Let $(C^*,\de)$ be a cochain complex over~$\Q$ with an action of a group~$G$.
If $G$ is finite, then the homomorphism~\eref{CGhomom_m} is an isomorphism.
\end{lmm}

\begin{proof}
Let $\eta\!\in\!C^*$ be a cocycle such that $[\eta]\!\in\!H^*(C^*,\de)^G$.
Then, the cocycle 
$$\eta_G\equiv \frac{1}{|G|}\sum_{g\in G}g\!\cdot\!\eta\in (C^*)^G$$
also represents  $[\eta]$ and so the homomorphism~\eref{CGhomom_m} is surjective.
If \hbox{$\eta\!\in\!(C^*)^G$} is a cocycle such that 
$\eta\!=\!\de\mu$ for some cochain $\mu\!\in\!C^*$, 
then $\eta\!=\!\de\mu_G$ and so the homomorphism~\eref{CGhomom_m} 
is injective.
\end{proof}

\noindent
Let $\wt{V}$ be a topological space with an action of a group $G$ and 
$$\pi\!:\wt{V}\lra V\!\equiv\!\wt{V}/G$$
be the projection to the quotient.
Since $\pi$ commutes with the group action,
\BE{CohPullBack_e}
\pi^*H^*(V;\Q)\subset H^*\!\big(\wt{V};\Q\big)^G
\equiv  H^*\!\big(C^*(\wt{V};\Q),\de\big)^G\,.\EE
In some important cases, the above inclusion is an equality.
If $\pi\!:\wt{V}\!\lra\!V$ is a regular covering, i.e.~the group~$G$ of
its deck transformations acts transitively on the fibers, then $V\!=\!\wt{V}/G$. 
In such a case, every $G$-invariant cochain (resp.~cycle) on~$\wt{V}$ descends 
to a cochain (resp.~cycle) on~$V$, i.e.
\BE{piCoCh_e}\pi^*\big(C^*(V;\Q),\de\big)= \big(C^*(\wt{V};\Q),\de\big)^G\,,\EE
provided the topological space~$V$ is sufficiently nice
and a suitable cohomology theory is used (so~that all elements of $C_*(V;\Q)$ lie inside
evenly covered open subsets of~$V$).
The inclusion in~\eref{CohPullBack_e} is an equality if the homomorphism~\eref{CGhomom_m}
with $C^*\!=\!C^*(\wt{V};\Q)$ is surjective.
The two statements below describe regular coverings for which this is the case.
For the CW-complexes appearing in these statements,
we use the (co)chain complexes generated by the cells of a CW-structure subordinate 
to the evenly covered open subsets of~$V$ or of its pullback to~$\wt{V}$.
A regular covering $\pi\!:\wt{V}\!\lra\!V$ is called \sf{abelian} 
if its group of deck transformations is abelian.

\begin{lmm}\label{CohSurj_lmm2}
Let $\pi\!:\wt{V}\!\lra\!V$ be an abelian covering of a connected CW-complex
 with the group of deck transformations~$\Z$.
Suppose $G$ is a finite group that acts on~$\wt{V}$ so that its action commutes
with the action of~$\Z$ and thus descends to~$V$.
Then,
$$\pi^*\big(H^*(V;\Q)^G\big)=H^*\big(\wt{V};\Q\big)^{\Z\times G}.$$
\end{lmm}

\begin{proof}
Let $u\!: \wt{V}\!\lra\!\wt{V}$ be a generator of the $\Z$-action and 
$$B_{\Z}\wt{V}=\big(\R\!\times\!\wt{V}\big)\big/\Z, \qquad (t,x)\sim(t\!-\!1,u\!\cdot\!x),$$
be the corresponding Borel construction for $\wt{V}$.
Since the actions of~$\Z$ and~$G$ on~$\wt{V}$ commute, the latter induces a $G$-action on $B_{\Z}\wt{V}$.
Since $\Z$ acts freely on~$\wt{V}$, the projection
$$q\!:B_{\Z}\wt{V}\lra V, \qquad q\big([t,x])= \pi(x)\equiv[x],$$
is a $G$-equivariant homotopy equivalence.
The composition of  the inclusion $\io\!:\wt{V}\!\lra\!B_{\Z}\wt{V}$ 
of a fiber for the fibration $B_{\Z}\wt{V}\!\lra\!S^1$ with~$q$ 
is the covering $\pi\!:\wt{V}\!\lra\!V$.
Thus, it is sufficient to show that the inclusion 
$$\io^*\big(H^*(B_{\Z}\wt{V};\Q)^G\big)\subset  H^*(\wt{V};\Q)^{\Z\times G}$$
is in fact an equality.\\

\noindent
Let $\eta\!\in\!C^k(\wt{V};\Q)$ be a cocycle such that $[\eta]\!\in\!H^k(\wt{V};\Q)^{\Z\times G}$.
Since
$$\eta_G\equiv \frac{1}{|G|}\sum_{g\in G}g\!\cdot\!\eta\in C^k(\wt{V};\Q)^G$$ 
determines the same element of $H^k(B_{\Z}\wt{V};\Q)$, we can assume that 
$\eta\!\in\!C^k(\wt{V};\Q)^G$.
Thus, 
$$u^*\eta-\eta=\de\mu \qquad\hbox{for some}\quad \mu\in C^{k-1}(\wt{V};\Q)^G. $$ 
The chain groups $C_k(\wt{V},\Q)$ and $C_{k-1}(\wt{V},\Q)$ are freely generated by 
the simplices 
$$\big\{u^s\!\circ\!\si_i\!:\,s\!\in\!\Z,~i\big\} \qquad\hbox{and}\qquad 
\big\{u^s\!\circ\!\tau_j\!:\,s\!\in\!\Z,~j\big\}$$
for some $k$-cells $\si_i$ and $(k\!-\!1)$-cells~$\tau_j$ on~$\wt{V}$. 
Define
$$\wt\eta\in C^k(\R\!\times\!\wt{V};\Q)^{\Z\times G} \quad\hbox{by}\quad
\wt\eta\big(\{r\}\!\times\!u^s\!\circ\!\si_i\big)=\eta\big(u^{r+s}\!\circ\!\si_i\big),~~
\wt\eta\big([r,r\!+\!1]\!\times\!u^s\!\circ\!\tau_j\big)=\mu\big(u^{r+s}\!\circ\!\tau_j\big).$$
If $\si_i$ is a $k$-cell and $\vp_{\ell}$ is a $(k\!+\!1)$-cell on~$\wt{V}$, then
\begin{equation*}\begin{split}
\wt\eta\big(\prt([r,r\!+\!1]\!\times\!u^s\!\circ\!\si_i)\big)
&=\wt\eta\big(\{r\!+\!1\}\!\times\!u^s\!\circ\!\si_i\big)
-\wt\eta\big(\{r\}\!\times\!u^s\!\circ\!\si_i\big)
-\wt\eta\big([r,r\!+\!1]\!\times\!u^s\!\circ\!\prt\si_i\big)\\
&=\{u^*\eta\}\big(u^{r+s}\!\circ\!\si_i\big)-\eta\big(u^{r+s}\!\circ\!\si_i\big)
-\prt\mu\big(u^{r+s}\!\circ\!\si_i\big)=0,\\
\wt\eta\big(\prt(\{r\}\!\times\!u^s\!\circ\!\vp_{\ell})\big)
&=\wt\eta\big(\{r\}\!\times\!u^s\!\circ\!\prt\vp_{\ell}\big)
=\big\{\de\eta\}\big(u^{r+s}\!\circ\!\vp_{\ell}\big)=0.
\end{split}\end{equation*}
Thus, $\wt\eta$ is a $\Z\!\times\!G$-invariant cocycle and so descends to 
a $G$-invariant cocycle on $B_{\Z}\wt{V}$. 
The latter restricts to $\eta$ along the fiber of $B_{\Z}\wt{V}\!\lra\!S^1$ over $0\!\in\!S^1$.
\end{proof}

\begin{crl}\label{CohSurj_crl}
Let $\pi\!:\wt{V}\!\lra\!V$ be an abelian covering of a connected CW-complex
and $G$ be its group of deck transformations.
If $G$ is finitely generated and $\rk_{\Z}G\!\le\!1$, then
$$\pi^*H^*(V;\Q)=H^*\big(\wt{V};\Q\big)^G.$$
\end{crl}

\begin{proof} By \cite[Theorem~12.6.4]{Artin}, 
$$G\approx \Z^r\times G_f$$ 
for some $r\!\in\!\Z^{\ge0}$ and some finite (abelian) group~$G_f$.
If $r\!=\!0$, the claim follows from~\eref{piCoCh_e} and Lemma~\ref{CohSurj_lmm1}.\\

\noindent
Suppose $r\!=\!1$. Let $\wt{V}_f\!=\!\wt{V}/\Z$ and
$$\pi_{\Z}\!:\wt{V}\lra\wt{V}_f \qquad\hbox{and}\qquad  \pi_f\!:\wt{V}_f\lra V$$
be the quotient projection maps. Thus, $\pi\!=\!\pi_f\!\circ\!\pi_{\Z}$.
By the $r\!=\!0$ case above and Lemma~\ref{CohSurj_lmm2},
$$\pi_f^*H^*(V;\Q)=H^*\big(\wt{V}_f;\Q\big)^{G_f}
\qquad\hbox{and}\qquad
\pi_{\Z}^*\big(H^*\big(\wt{V}_f;\Q\big)^{G_f}\big)=H^*\big(\wt{V};\Q\big)^G\,,$$
respectively. Combining the two equations, we obtain the $r\!=\!1$ case of the claim.
\end{proof}

\section{The refined relative GW-counts}
\label{RelGW_sec}

\noindent
We now provide the details needed to refine the standard relative GW-invariants of~$(X,V,\om)$,
as suggested in \cite[Section~5]{IPrel} and outlined in Section~\ref{RelGW_subs0}.
The coverings~\eref{IPcov_e} are special cases of the abelian covers described in 
Section~\ref{AbCovNotat_subs} and are specified at the beginning of Section~\ref{RTCov_subs}.
In the remainder of Section~\ref{RTCov_subs}, we show that 
the total relative evaluation morphisms~\eref{evdfn_e2} lift  to the these coverings
and establish Theorem~\ref{EqRelGWs_thm}.
The lifts~\eref{evXVlift_e} are not unique, but can be chosen consistently; 
see Theorem~\ref{RimToriAct_thm}  in Section~\ref{RTCov_subs2}.
Theorem~\ref{RelGW_thm} is proved in Section~\ref{RelGWThm_subs}.

\subsection{The rim tori covers}
\label{RTCov_subs}

\noindent
Let $X$ be a compact oriented  manifold and $V\!\subset\!X$ be a compact oriented submanifold of 
codimension~$\fc$ with topological components $V_1,\ldots,V_N$.
With $H_X^V\!\subset\!H_1(V;\Z)$ as in~\eref{H1VXdfn_e}, 
$$\cR_{H_X^V}\equiv H_1(V;\Z)_X\approx\cR_X^V\,;$$
see Corollary~\ref{rimtori_crl}.
For $\ell_1,\ldots,\ell_N\!\in\!\Z^{\ge0}$ and
$\bs_1\!\in\!\Z_{\pm}^{\,\ell_1},\ldots,\bs_N\!\in\!\Z_{\pm}^{\,\ell_N}$, define
\begin{gather*}
H_{X;\bs_1\ldots\bs_N}^V=\big(H_X^V\big)_{\bs_1\ldots\bs_N}\subset H_1(V_{\bs_1\ldots\bs_N};\Z),\\
\cR_{X;\bs_1\ldots\bs_N}'^{\,V}=\cR_{H_X^V;\bs_1\ldots\bs_N}'\subset H_1(V;\Z)_X\,,
\qquad  \cR_{X;\bs_1\ldots\bs_N}^V=\cR_{H_X^V;\bs_1\ldots\bs_N}\,.
\end{gather*}
The rim tori covers~\eref{IPcov_e} are the abelian covers
\BE{IPcov_e2}\begin{split}
\pi_{X;\bs_1\ldots\bs_N}'^V\!\equiv\!\pi_{H_X^V;\bs_1\ldots\bs_N}'^V\!:
\wh{V}_{X;\bs_1\ldots\bs_N}'&\equiv\wh{V}_{H_X^V;\bs_1\ldots\bs_N}'
\lra V_{\bs_1\ldots\bs_N},\\
\pi_{X;\bs_1\ldots\bs_N}^V\!\equiv\!\pi_{H_X^V;\bs_1\ldots\bs_N}^V\!: 
\wh{V}_{X;\bs_1\ldots\bs_N}&\equiv\wh{V}_{H_X^V;\bs_1\ldots\bs_N}
\lra V_{\bs_1\ldots\bs_N}.
\end{split}\EE
We will write elements of the second covering as
\BE{whVelem_e2}\big([\ga]_{X;\bs_1\ldots\bs_N},[\wh{x}]_X\big)
\in  \frac{\cR_X^V}{\cR_{X;\bs_1\ldots\bs_N}'^V}\!\times\!\wh{V}_{X;\bs_1\ldots\bs_N}',\EE
with notation as in~\eref{whVelem_e} for $H\!=\!H_X^V$.\\

\noindent
By~\eref{HbsDeck_e}, the groups of deck transformations of these regular coverings are
\BE{IPdeck_e}\Deck\big(\pi_{X;\bs_1\ldots\bs_N}'^V\big)=\cR_{X;\bs_1\ldots\bs_N}'^{\,V}
\qquad\hbox{and}\qquad
\Deck\big(\pi_{X;\bs_1\ldots\bs_N}^V\big) =\cR_{X;\bs_1\ldots\bs_N}^V\,,\EE
respectively.
If $V$ is connected,
$$\Deck\big(\pi_{X;\bs_1\ldots\bs_N}^V\big)\approx
\frac{H_1(V;\Z)_X}{\gcd(\bs)H_1(V;\Z)_X}\!\times\!\gcd(\bs)H_1(V;\Z)_X\,.$$
In general, the second group in~\eref{IPdeck_e} is different from $\cR_X^V$ 
(contrary to an explicit statement in 
\cite[Section~5]{IPrel} and the spirit of the description).
In the $\ell_1,\ldots,\ell_N\!=\!0$ case, 
$\wh{V}_{X;\bs_1\ldots\bs_N}$ is a discrete set of points identified with~$\cR_X^V$.
In most other cases, the coverings~\eref{IPcov_e2} are non-trivial
(i.e.~the first one is not  $V_{\bs_1\ldots\bs_N}\!\times\!\cR_{X;\bs_1\ldots\bs_N}'^{\,V}$).\\

\noindent
Suppose in addition that $V'\!\subset\!V$ is the union of $V_1,\ldots,V_{N'}$ for some $N'\!\le\!N$.
By Corollary~\ref{rimtori_crl2}, $H_X^{V'}\!\subset\!H_1(V';\Z)$ is 
the image of~$H_X^V$ under the projection~\eref{H1proj_e}.
Thus, the continuous map~\eref{whVforg_e} with $H\!=\!H_X^V$ and $H'\!=\!H_X^{V'}$
induces a commutative diagram
\BE{forgIPcov_e}\begin{split}\xymatrix{ \wh{V}_{X;\bs_1\ldots\bs_N} \ar[d]\ar[rr]^{\wt{q}}&& 
\wh{V'}_{X;\bs_1\ldots\bs_{N'}}\ar[d] \\
V_{\bs_1\ldots\bs_N} \ar[rr]^q&& V'_{\bs_1\ldots\bs_{N'}}}
\end{split}\EE
of fiber bundles.
It corresponds to the right square in the diagram of Figure~\ref{IPev_fig}.

\begin{eg}\label{EScover_eg}
Suppose  $\wh\P^2_9$ is a rational elliptic surface as in Example~\ref{EllFib_eg}, 
$F\!\subset\!\wh\P^2_9$ is a smooth fiber,
$\ell\!\in\!\Z^+$, and $\bs\!\in\!\Z^{\ell}$.
In this case, $N\!=\!1$, $H_X^V\!=\!\{0\}$, and 
the first covering in~\eref{IPcov_e2} is isomorphic to
the restriction of~\eref{EScover_e} to any of the connected components of $\C\!\times\!\T_{\bs}^{2(n-1)}$. 
Its group of deck transformations is $\cR_{\wh\P^2_9;\bs}'^{\,F}\!\approx\!\Z^2$
and can be identified~with
$$\gcd(\bs)\cR_{\wh\P^2_9}^F\subset \cR_{\wh\P^2_9}^F\approx \Z^2\,.$$
The second covering in~\eref{IPcov_e2} is~\eref{EScover_e} itself;
its group of deck transformations is 
isomorphic to $(\Z_{\gcd(\bs)})^2\!\oplus\!\Z^2$.
\end{eg}

\begin{eg}\label{P1Vrim_eg3}
Let $F$, $X$, and $F_0,F_{\i}\!\subset\!X$ be as in Example~\ref{P1Vrim_eg} with $F$ connected and 
$V\!=\!F_0\!\cup\!F_{\i}$.
In this case, $N\!=\!2$,
$$H_1(V;\Z)=H_1(F;\Z)\oplus H_1(F;\Z),$$
and $H_X^V\!\subset\!H_1(V;\Z)$ is the diagonal.
With the identifications of Example~\ref{P1Vrim_eg},
the composition of the homomorphism~\eref{PhiVbs_e} with the projection to~$\cR_X^V$
can be written~as 
\begin{gather*}
H_1(F;\Z)^{\ell_1}\!\oplus\!H_1(F;\Z)^{\ell_2}\lra H_1(F;\Z), \quad
\big((\ga_{0;i})_{i\le\ell_1},(\ga_{\i;i})_{i\le\ell_2}\big)
\lra \sum_{i=1}^{\ell_1}\!s_{1;i}\ga_{0;i}-\sum_{i=1}^{\ell_2}\!s_{2;i}\ga_{\i;i}\,.
\end{gather*}
The first covering in~\eref{IPcov_e2} is thus the first covering in~\eref{HbsCov_e} 
with~$V$ replaced by~$F$, $H\!=\!\{0\}$, and $\bs$ being the merged tuple of~$\bs_1$ and~$-\bs_2$.
If $F\!=\!\T^2$ and $(\ell_1,\ell_2)\!\neq\!\0$, the second covering in~\eref{IPcov_e2}
is described by~\eref{EScover_e} with $\bs$ replaced by the merged tuple of~$\bs_1$
and~$-\bs_2$.
With $V'\!=\!F_0$ in~\eref{forgIPcov_e}, $\wh{V'}_{X;\bs_1}\!=\!V^{\ell_1}$ and 
$\wt{q}\!=\!q$.
\end{eg}

\noindent
If $\Si$ is a  compact oriented $m$-dimensional manifold, $A\!\in\!H_m(X;\Z)$,
$k\!\in\!\Z^{\ge0}$, and \hbox{$p\!>\!m$}, let $\fX_{\Si,k}(X,A)$ be the space of tuples
$(z_1,\ldots,z_k,f)$
such~that $f\!\in\!L^p_1(\Si;X)$,
$f_*[\Si]\!=\!A$, and  \hbox{$z_1,\ldots,z_k\in\Si$} are distinct points. 
If in addition $m\!=\!\fc$, $V_1,\ldots,V_N$ and
$\bs_1,\ldots,\bs_N$ are as before, and \eref{bsumcond_e} holds for each
$(V,\bs)\!=\!(V_r,\bs_r)$, let 
$$\fX_{\Si,k;\bs_1\ldots\bs_N}^{V_1,\ldots,V_N}(X,A)\subset 
\fX_{\Si,k+\ell_1+\ldots+\ell_N}(X,A)$$
be the subspace of tuples $(z_1,\ldots,z_{k+\ell_1+\ldots+\ell_N},f)$ such~that 
\begin{alignat*}{2}
&f^{-1}(V_r)=\big\{z_{k+\ell_1+\ldots+\ell_{r-1}+1},\ldots,
z_{k+\ell_1+\ldots+\ell_r}\big\}
&\qquad &\forall~r=1,\ldots,N,\\
&\ord_{z_{k+\ell_1+\ldots+\ell_{r-1}+i}}^{V_r}f=s_{r;i}
&\qquad &\forall~i\!=\!1,2,\ldots,\ell_r,~r\!=\!1,\ldots,N.
\end{alignat*}
We denote~by 
\BE{evXVdfn_e}
\ev_X^V\!=\!\ev_{k+1}\!\times\!\ldots\!\times\!\ev_{k+\ell_1+\ldots+\ell_N}\!:
\fX_{\Si,k;\bs_1\ldots\bs_N}^{V_1,\ldots,V_N}(X,A)
\lra V_{\bs_1\ldots\bs_N}\EE
\sf{the total relative evaluation morphism}.
The $\fc\!=\!2$ case of the next lemma is  
the implied claim of \cite[Section~5]{IPrel}.

\begin{lmm}\label{RimToriAct_lmm}
Suppose $X$ is a compact oriented manifold, 
$V\!\subset\!X$ is a 
compact oriented submanifold of codimension~$\fc$
with  connected components $V_1,\ldots,V_N$,
$A\!\in\!H_{\fc}(X;\Z)$, and
$\bs_r\!\in\!\Z_{\pm}^{\,\ell_r}$ for \hbox{$r\!=\!1,\ldots,N$}.
If $\Si$ is a  compact oriented $\fc$-dimensional manifold, $k\!\in\!\Z^{\ge0}$,
and $r\!=\!1,\ldots,N$, then the morphism~\eref{evXVdfn_e} 
lifts over the first covering in~\eref{IPcov_e2} to a continuous map
$$\wt\ev_X'^{\,V}\!:\fX_{\Si,k;\bs_1\ldots\bs_N}^{V_1,\ldots,V_N}(X,A)
\lra\wh{V}_{X;\bs_1\ldots\bs_N}'\,.$$
\end{lmm}

\begin{proof}
Let
$$\ga\!:[0,1]\lra\fX_{\Si,k;\bs_1\ldots\bs_N}^{V_1,\ldots,V_N}(X,A),\qquad
t\lra (z_{t;1},\ldots,z_{t;k+\ell_1+\ldots+\ell_N},f_t),$$ 
be a loop.
For each $r\!=\!1,\ldots,N$ and $i\!=\!1,\ldots,\ell_r$,
$$\ga_{r;i}\equiv
\ev_{k+\ell_1+\ldots+\ell_{r-1}+i}\circ\ga\!:[0,1]\lra V_r$$
is also a loop.
For each $i\!=\!1,\ldots,\ell_1\!+\!\ldots\!+\!\ell_N$, 
let $B_{z_{k+i}}\!\subset\!\Si$ be a small ball around~$z_{k+i}$.
As in the construction of $f\!\#\!(-f')$ in Section~\ref{prelimcut_subs},
it can be assumed that 
$$f_t\!: \partial B_{z_{k+\ell_1+\ldots+\ell_{r-1}+i}}\lra
S_XV|_{f_t(z_{k+\ell_1+\ldots+\ell_{r-1}+i})}\,.$$
Since $\ga$ is a loop,
\begin{equation*}\begin{split}
0=[f_1\!\#\!(-f_0)]
&=\sum_{r=1}^N\sum_{i=1}^{\ell_r}
\io_{S_XV_r*}^{X-V}\big(\De_X^{V_r}(s_{r;i}\ga_{r;i})\big)\\
&=\io_{S_XV*}^{X-V}\big(\De_X^V\big(\Phi_{V;\bs_1\ldots\bs_N}(\ev_X^V\!\circ\!\ga)\big)
\in \cR_X^V\,.
\end{split}\end{equation*}
By Corollary~\ref{rimtori_crl}, this implies that 
$\Phi_{V;\bs_1\ldots\bs_N}(\ev_X^V\!\circ\!\ga)\!\in\!H_X^V$.
Thus, the image of the fundamental group of $\fX_{\Si,k;\bs_1\ldots\bs_N}^{V_1,\ldots,V_N}(X,A)$ 
under~\eref{evXVdfn_e} in $\pi_1(V_{\bs_1\ldots\bs_N})$
lies in the image of the fundamental group of $\wh{V}_{X;\bs_1\ldots\bs_N}'$ 
under $\pi_{X;\bs_1\ldots\bs_N}'^V$.
By \cite[Lemma~79.1]{Mu}, this implies the claim.
\end{proof}

\begin{proof}[{\bf{\emph{Proof of Theorem~\ref{EqRelGWs_thm}}}}]
Since $\gcd(\bs)$ and $|\cR_X^V|$ are relatively prime, 
$\wh{V}_{H;\bs}\!=\!\wh{V}_{H;\bs}'$.
The second inclusion in~\eref{cohcover_e} is then an equality by 
Corollary~\ref{rimtoriFlus_crl}\ref{rimtoriFlus_it2a} with $H\!=\!H_X^V$.
Under the additional assumption~\eref{EqRelGWs_e2}, both 
inclusions in~\eref{cohcover_e} are equalities by 
Corollary~\ref{rimtoriFlus_crl}\ref{rimtoriFlus_it2b}.
\end{proof}

\begin{rmk}\label{EqRelGWs_rmk}
Let $V$ be possibly disconnected with topological components $V_1,\ldots,V_N$.
The conclusions of Theorem~\ref{EqRelGWs_thm} then hold if
\eref{EqRelGWs_e1} and the relatively prime condition are replaced~by
\BE{EqRelGWs_e3}
\big\{\sum_{\ell_r\neq0}[\ga_r]_{H_X^V}\!:\,\ga_r\!\in\!\Flux(V_r)_X~\forall\,r\big\}=\Flux(V)_X
\qquad\hbox{and}\qquad
\cR_{X;\bs_1\ldots\bs_N}'^V=\cR_X^V,\EE
respectively.
The latter is the case if $\gcd(\bs_r)$ and $|H_1(V_r;\Z)|$ are relatively prime 
for every~$r$.
It is also the case if $F$ is connected, 
$X\!=\!\P^1\!\times\!F$ and $V\!=\!\{0,\i\}\!\times\!F$ as in Example~\ref{P1Vrim_eg3},
and $\gcd(\bs_1,\bs_2)$ and $|H_1(F;\Z)|$ are relatively prime.   
The conclusion of Proposition~\ref{EqRelGWs_prp} holds if $V_r\!\approx\!\T^{2n-2}$
for  every $r\!=\!1,\ldots,N$ and \eref{EqRelGWs_e3} holds.
\end{rmk}

\subsection{Consistent choices of lifts}
\label{RTCov_subs2}

\noindent
We next show that the lifts in Lemma~\ref{RimToriAct_lmm} can be chosen in 
a systematic way, consistent with their use in~\cite{IPsum} for refining 
the symplectic sum formula for GW-invariants and with the diagram in Figure~\ref{IPev_fig}.
The significance of Theorem~\ref{RimToriAct_thm} for the former
is demonstrated by \cite[Proposition~4.2]{GWsumIP}.
We continue with the notation of Section~\ref{RTCov_subs}.

\begin{thm}\label{RimToriAct_thm}
Suppose $X$ is a compact oriented manifold, $V\!\subset\!X$ is a 
compact oriented submanifold of codimension~$\fc$ with connected components
$V_1,\ldots,V_N$, $A\!\in\!H_{\fc}(X;\Z)$, and 
$\bs_r\!\in\!\Z_{\pm}^{\,\ell_r}$ for $r\!=\!1,\ldots,N$. 
Let $\{\ga_j\}\!\subset\!H_1(V;\Z)$ be a collection of representatives
for the elements of $\cR_X^V/\cR_{X;\bs_1\ldots\bs_N}'^{\,V}$.
If $\Si$ is a  compact oriented $\fc$-dimensional manifold and $k\!\in\!\Z^{\ge0}$,
there exists a~lift 
\BE{EvLift_e}\wt\ev_X^V\!:
\fX_{\Si,k;\bs_1\ldots\bs_N}^{V_1,\ldots,V_N}(X,A) \lra 
\wh{V}_{X;\bs_1\ldots\bs_N}\EE
of the morphism~$\ev_X^V$ in~\eref{evXVdfn_e} 
over the covering $\pi_{X;\bs_1\ldots\bs_N}^V$ in~\eref{IPcov_e2} with the following property.
For any $\bff,\bff'\!\in\!\fX_{\Si,k;\bs_1\ldots\bs_N}^{V_1,\ldots,V_N}(X,A)$ with
\BE{RTmatch_e}\wt\ev_X^V(\bff)=\big([\ga_j]_{X;\bs_1\ldots\bs_N},[\ga\!\cdot\!\wh{x}]_X\big)
\quad\hbox{and}\quad 
\wt\ev_X^V(\bff')=\big([\ga_{j'}]_{X;\bs_1\ldots\bs_N},[\wh{x}]_X\big)\EE
for some $\wh{x}\!\in\!\wh{V}_{\bs_1\ldots\bs_N}$, 
$\ga\!\in\!H_1(V_{\bs_1\ldots\bs_N};\Z)$, and $j,j'$ indexing the coset representatives,
the map components of~$\bff$ and~$\bff'$ satisfy
\BE{RimToriAct_e}\big[f\!\#\!(-f')\big]=
\io_{S_XV*}^{X-V}\big(\De_X^V\big(\Phi_{V;\bs_1\ldots\bs_N}(\ga)\!+\!\ga_j\!-\!\ga_{j'}\big)\big)
\in H_{\fc}(X\!-\!V;\Z).\EE
Furthermore, $\wt\ev_X^V(\bff')$ is the unique point in 
$\pi_{X;\bs_1\ldots\bs_N}^{V~-1}(\ev_X^V(\bff'))$ so that \eref{RimToriAct_e} holds
for a given value of~$\wt\ev_X^V(\bff)$.
\end{thm}

\begin{proof} 
We can assume that $\fX_{\Si,k;\bs_1\ldots\bs_N}^{V_1,\ldots,V_N}(X,A)\!\neq\!\eset$;
otherwise, there is nothing to prove.
Choose a base point $\wh{x}_{\bs_1\ldots\bs_N}\!\in\!\wh{V}_{\bs_1\ldots\bs_N}$
over some $x_{\bs_1\ldots\bs_N}\!\in\!V_{\bs_1\ldots\bs_N}$,  
an overall base map 
\BE{bff0_e}\bff_0\equiv (z_{0;1},\ldots,z_{0;k+\ell_1+\ldots+\ell_N},f_0)
\in \fX_{\Si,k;\bs_1\ldots\bs_N}^{V_1,\ldots,V_N}(X,A) \quad\hbox{s.t.}~~ 
\ev_X^V(\bff_0)\!=\!x_{\bs_1\ldots\bs_N},\EE
and a base point 
$$\bff_m\equiv (z_{m;1},\ldots,z_{m;k+\ell_1+\ldots+\ell_N},f_m)
\in \fX_{\Si,k;\bs_1\ldots\bs_N}^{V_1,\ldots,V_N}(X,A) $$
for each topological component so~that 
\BE{fmf0_e} 
\ev_X^V(\bff_m)\!=\!x_{\bs_1\ldots\bs_N}, \qquad
\big[f_m\!\#(-f_0)\big]=\io_{S_XV*}^{X-V}\big(\De_X^V(\ga_j)\big)  \in H_{\fc}(X\!-\!V;\Z)\EE
for some $j\!=\!j(m)$.
We define 
\BE{wtevmain_e}\wt\ev_X^V\!:\fX_{\Si,k;\bs_1\ldots\bs_N}^{V_1,\ldots,V_N}(X,A)
\lra \wh{V}_{X;\bs_1\ldots\bs_N}\equiv
\frac{\cR_X^V}{\cR_{X;\bs_1\ldots\bs_N}'^{\,V}}\!\times\!\wh{V}_{X;\bs_1\ldots\bs_N}'\EE
to be the lift of $\ev_X^V$ over~$\pi_{X;\bs_1\ldots\bs_N}^V$ so~that
$$ \wt\ev_X^V(\bff_m)=\big([\ga_{j(m)}]_{X;\bs_1\ldots\bs_N},[\wh{x}_{\bs_1\ldots\bs_N}]_X\big) 
\qquad\forall~m;$$
see \cite[Lemma~79.1]{Mu}.\\

\noindent
It remains to verify that this lift has the claimed properties.
Suppose $\bff,\bff'\!\in\!\fX_{\Si,k;\bs_1\ldots\bs_N}^{V_1,\ldots,V_N}(X,A)$
satisfy~\eref{RTmatch_e}  and thus  $\ev_X^V(\bff)\!=\!\ev_X^V(\bff')$.
Let $\bff_m$ and $\bff_{m'}$ be the base points of the topological components 
of $\fX_{\Si,k;\bs_1\ldots\bs_N}^{V_1,\ldots,V_N}(X,A)$
containing $\bff$ and $\bff'$, respectively,
and $\al$ and $\al'$ be paths from~$\bff_m$ to~$\bff$ and 
from~$\bff_{m'}$ to~$\bff'$, respectively.
By~\eref{RTmatch_e}, $j(m)\!=\!j$ and $j(m')\!=\!j'$. 
Along with~\eref{fmf0_e}, this~gives 
\BE{fmdiff_e}
\big[f_m\!\#\!(-f_{m'})\big]=\big[f_m\!\#\!(-f_0)\big]-\big[f_{m'}\#\!(-f_0)\big]
=\io_{S_XV*}^{X-V}\big(\De_X^V(\ga_j\!-\!\ga_{j'})\big)  \in H_{\fc}(X\!-\!V;\Z).\EE
Since $\ev_X^V(\bff_m)\!=\!\ev_X^V(\bff_{m'})$ and $\ev_X^V(\bff)\!=\!\ev_X^V(\bff')$,
$$\ga'\equiv\big(-\ev_X^V\!\circ\!\al'\big)*\big(\ev_X^V\!\circ\!\al\big)\!:[0,1]\lra V_r$$
is a well-defined loop.
Furthermore, 
$$[f\!\#\!(-f')]-[f_m\!\#\!(-f_{m'})]
=\io_{S_XV*}^{X-V}\big(\De_X^V(\Phi_{V;\bs_1\ldots\bs_N}(\ga'))\big)
\in  H_{\fc}(X\!-\!V;\Z)\,.$$
Combining this statement with~\eref{fmdiff_e}, we find that 
\BE{liftcomp_e}
[f\!\#\!(-f')]=
\io_{S_XV*}^{X-V}\big(\De_X^V\big(\Phi_{V;\bs_1\ldots\bs_N}(\ga')\!+\!\ga_j\!-\!\ga_{j'}\big)\big)
\in  H_{\fc}(X\!-\!V;\Z).\EE\\

\noindent
Let $\wh\al,\wh\al',\wh\ga'\!:[0,1]\!\lra\!\wh{V}_{\bs_1\ldots\bs_N}$ be 
the lifts of $\ev_X^V\!\circ\!\al,\ev_X^V\!\circ\!\al',\ga'$ such~that
$$\wh\al(0),\wh\al'(0)=\wh{x}_{\bs_1,\ldots,\bs_N}, \qquad
\wh\ga'(0)=\wh\al'(1).$$
Thus,
$$\big[\wh\al(1)\big]_X=\wt\ev_X'^{\,V}(\bff)=[\ga\!\cdot\!\wh{x}]_X,\qquad
\big[\wh\al'(1)\big]_X=\wt\ev_X'^{\,V}(\bff')=[\wh{x}]_X,$$
where $\wt\ev_X'^{\,V}$ is the composition of~\eref{wtevmain_e} with the projection 
to the second component.
Since $\wh\al\!=\!\wh\al'\!*\!\wh\ga'\!:[0,1]\!\lra\!\wh{V}_{\bs_1\ldots\bs_N}$,
$$[\ga\!\cdot\!\wh{x}]_X=\big[\wh\al(1)\big]_X
=\big[\ga'\!\cdot\!\wh\al'(1)\big]_X=[\ga'\!\cdot\!\wh{x}]_X\,.$$
Since $\wh{V}'_{X;\bs_1\ldots\bs_N}\!=\!\wh{V}_{\bs_1\ldots\bs_N}/H_{X;\bs_1\ldots\bs_N}^V$,
it follows that 
$$\Phi_{V;\bs_1\ldots\bs_N}(\ga)-\Phi_{V;\bs_1\ldots\bs_N}(\ga')\in H_X^V\subset H_1(V;\Z)\,.$$
The claim~\eref{RimToriAct_e} now follows from~\eref{liftcomp_e} and Corollary~\ref{rimtori_crl}
with $U\!=\!\eset$. 
The latter also implies the uniqueness claim.
\end{proof}

\noindent
We will call a lift~\eref{EvLift_e} satisfying the properties  of Theorem~\ref{RimToriAct_thm}
\sf{$\{\ga_j\}$-compatible}.
The next proposition describes all such lifts.

\begin{prp}\label{RimToriAct_prp}
Suppose $X$, $V$, $\fc$, $A$, $\bs_r$, $\{\ga_j\}$, $\Si$, and $k$  
are as in the statement of Theorem~\ref{RimToriAct_thm}.
\begin{enumerate}[label=(\arabic*),leftmargin=*]

\item\label{comptest_it} Let $x_{\bs_1\ldots\bs_N}\!\in\!V_{\bs_1\ldots\bs_N}$.
If a lift~\eref{EvLift_e} of~\eref{evXVdfn_e} satisfies~\eref{RimToriAct_e}
for all $\bff,\bff'\!\in\!\fX_{\Si,k;\bs_1\ldots\bs_N}^{V_1,\ldots,V_N}(X,A)$ 
with $\ev_X^V(\bff)\!=\!x_{\bs_1\ldots\bs_N}$, then it is  $\{\ga_j\}$-compatible.

\item\label{Thcomp_it} Let $\eta\!\in\!H_1(X;\Z)$.
If $\wt\ev_X^V$ is a lift of~\eref{evXVdfn_e} compatible with $\{\ga_j\}$,
then so is the lift $\Th_{\eta}\!\circ\!\wt\ev_X^V$.

\item\label{Thcomp_it2} If  the lifts $\wt\ev_X^V$ and $\wt\ev_X'^V$ as in~\eref{evXVdfn_e} 
are compatible with $\{\ga_j\}$,
then  $\wt\ev_X'^V\!=\!\Th_{\eta}\!\circ\!\wt\ev_X^V$ for some 
 \hbox{$\eta\!\in\!H_1(X;\Z)$}.
\end{enumerate}
\end{prp}

\begin{proof}
\ref{comptest_it} This is immediate because the two sides of~\eref{RimToriAct_e}
take discrete values, but depend continuously on $(\bff,\bff')$.\\

\noindent
\ref{Thcomp_it} If  $\bff,\bff'\!\in\!\fX_{\Si,k;\bs_1\ldots\bs_N}^{V_1,\ldots,V_N}(X,A)$
satisfy~\eref{RTmatch_e}, then 
$$\Th_{\eta}\big(\wt\ev_X^V(\bff)\big)=
\big([\ga_j(\eta)]_{X;\bs_1\ldots\bs_N},[\eta_j\ga\!\cdot\!\wh{x}]_X\big),
\quad
\Th_{\eta}\big(\wt\ev_X^V(\bff')\big)=\big([\ga_{j'}(\eta)]_{X;\bs_1\ldots\bs_N},
[\eta_{j'}\!\cdot\!\wh{x}]_X\big),$$
with notation as in~\eref{gajeta_e}.
By~\eref{gajeta_e} and~\eref{RimToriAct_e}, 
\begin{equation*}\begin{split}
\io_{S_XV*}^{X-V}\big(\De_X^V\big(
\Phi_{V;\bs_1\ldots\bs_N}(\eta_j\!+\!\ga\!-\!\eta_{j'})\!+\!\ga_j(\eta)
\!-\!\ga_{j'}(\eta)\big)\big)
&=\io_{S_XV*}^{X-V}\big(\De_X^V
\big(\Phi_{V;\bs_1\ldots\bs_N}(\ga)\!+\!\ga_j\!-\!\ga_{j'}\big)\big)\\
&=\big[f\!\#\!(-f')\big] \in H_{\fc}(X\!-\!V;\Z).
\end{split}\end{equation*}
This establishes the second claim.\\

\noindent
\ref{Thcomp_it2} Let $x_{\bs_1\ldots\bs_N}$ and $\bff_0$ be as in~\eref{bff0_e}.
Suppose
$$\wt\ev_X^V(\bff_0)=\big([\ga_j]_{X;\bs_1\ldots\bs_N},[\wh{x}]_X\big), \qquad
\wt\ev_X'^V(\bff_0)=\big([\ga_j']_{X;\bs_1\ldots\bs_N},[\wh{x}']_X\big),$$
where $\ga_j,\ga_j'$ are among the chosen coset representatives.
Since $\wt\ev_X^V$ and $\wt\ev_X'^V$ are lifts of~\eref{evXVdfn_e},
$$[\wh{x}']_X=[\eta_j\!\cdot\!\wh{x}]_X\,$$
for some $\eta_j\!\in\!H_1(V_{\bs_1\ldots\bs_N};\Z)$. Let
$$\eta=\ga_{j}'-\ga_{j}+\Phi_{V;\bs_1\ldots\bs_N}(\eta_j).$$
In particular, $\ga_j(\eta)\!=\!\ga_j'$.
Suppose $\bff'\!\in\!\fX_{\Si,k;\bs_1\ldots\bs_N}^{V_1,\ldots,V_N}(X,A)$,
$$\wt\ev_X^V(\bff')=\big([\ga_{j'}]_{X;\bs_1\ldots\bs_N},[\ga^{-1}\!\cdot\!\wh{x}]_X\big), \quad
\wt\ev_X'^V(\bff')=\big([\ga_{j'}']_{X;\bs_1\ldots\bs_N},[\eta_{j'}\ga^{-1}\!\cdot\!\wh{x}]_X\big)$$
for some $\ga,\eta_{j'}\!\in\!H_1(V_{\bs_1\ldots\bs_N};\Z)$.
Since $\wt\ev_X^V$ and $\wt\ev_X'^V$ are lifts of~\eref{evXVdfn_e} compatible with $\{\ga_j\}$,
\begin{equation*}\begin{split}
\io_{S_XV*}^{X-V}\big(\De_X^V
\big(\Phi_{V;\bs_1\ldots\bs_N}(\ga)\!+\!\ga_j\!-\!\ga_{j'}\big)\big)
&=\big[f_0\!\#\!(-f')\big] \\
&=\io_{S_XV*}^{X-V}\big(\De_X^V
\big(\Phi_{V;\bs_1\ldots\bs_N}(\eta_j\!+\!\ga\!-\!\eta_{j'})\!+\!\ga_j'\!-\!\ga_{j'}'\big)\big).
\end{split}\end{equation*}
It follows that 
\begin{gather*}
\ga_{j'}+\eta-\ga_{j'}'-\Phi_{V;\bs_1\ldots\bs_N}(\eta_{j'})\in H_X^V, \qquad
\ga_{j'}(\eta)\!=\!\ga_{j'}', \\
\Th_{\eta}\big([\ga_{j'}]_{X;\bs_1\ldots\bs_N},[\ga^{-1}\!\cdot\!\wh{x}]_X\big)=
\big([\ga_{j'}']_{X;\bs_1\ldots\bs_N},[\eta_{j'}\ga^{-1}\!\cdot\!\wh{x}]_X\big).
\end{gather*}
Thus, $\wt\ev_X'^V\!=\!\Th_{\eta}\!\circ\!\wt\ev_X^V$. 
\end{proof}

\begin{rmk}\label{cHlift_rmk}
The above choices of $\{\ga_j\}$, ~$\bff_0$, and $\wh{x}_{\bs_1\ldots\bs_N}$ roughly
correspond to the two set-theoretic descriptions  of~$\wh{V}_{X;\bs_1\ldots\bs_N}$ in~\cite{IPrel}
as equivalence classes of cycles in~$X\!-\!V$ that are standard near~$V$.
These descriptions in \cite[Section~5]{IPrel} do not specify a topology, especially 
when the contact points come together.
A hands-on description of the topology of $\wh{V}_{X;(1)}$ is given at the end of \cite[Section~5]{IPrel};
our definition of $\wh{V}_{X;\bs_1\ldots\bs_N}$ 
via the homomorphisms $\Phi_{V;\bs_1\ldots\bs_N}$ is based on this description.
Lemma~\ref{RimToriAct_lmm} makes it apparent that the lifts
of the evaluation maps can be extended over the compactified moduli spaces of relative
stable maps once 
these spaces are shown to be compatible with gluing;
the latter is necessary for any virtual construction of the moduli cycle
and ensures that the boundary strata of the configuration spaces are of real codimension
at least~2 (and so the loop~$\ga$ as in the proof of Lemma~\ref{RimToriAct_lmm} can be
assumed to lie completely in the main stratum). 
\end{rmk}

\begin{eg}\label{EScover_eg4}
If  $X\!=\!\wh\P^2_9$ is a rational elliptic surface,
$V\!=\!F$ is a smooth fiber as in Examples~\ref{EllFib_eg} and~\ref{EScover_eg},
and $\ell\!\in\!\Z^+$, then $H_X^V\!=\!0$.
We can identify $H_1(V;\Z)$ with $\Z\!\oplus\!\fI\Z$ and
take
$$x_{\bs}=[\0_{2\ell}]\in V^{\ell}\!=\!\T^{2\ell}, \qquad
\wh{x}_{\bs}=\0_{2\ell}\in\C^{\ell}.$$
Let $\{\ga_j\}\!\subset\!\Z\!\oplus\!\fI\Z$ be a collection of representatives for 
the elements of $\Z_{\gcd(\bs)}\!\oplus\!\fI\Z_{\gcd(\bs)}$.
As indicated in \cite[Example~3.1]{GWsumIP}, it is convenient to identify 
the corresponding base points for the components of~$\wh{V}_{X;\bs}$~as
$$\big([\ga_j]_{X;\bs},[\wh{x}_{\bs}]_X\big)=
\big(\ell^{-1}\ga_j,[\ell^{-1}s_i^{-1}\ga_j]_{i\le\ell}\big)\in
\wh{V}_{X;\bs}= \C\!\times\!\T_{\bs}^{2(\ell-1)}\,.$$
\end{eg}

\begin{eg}\label{P1Vrim_eg4}
In the setting of Examples~\ref{P1Vrim_eg} and~\ref{P1Vrim_eg3}, only 
the tuples $\bs_1\!\in\!\Z_{\pm}^{\ell_1}$ and $\bs_2\!\in\!\Z_{\pm}^{\ell_2}$ with
$$\sum_{i=1}^{\ell_1}s_{1;i}=\sum_{i=1}^{\ell_2}s_{2;i}$$
are relevant in the context of Theorem~\ref{RimToriAct_thm}.
We first choose a base point $x_1\!\in\!F$.
For all $\bs_1$ and~$\bs_2$ as above, we then choose a base point $\wh{x}_{\bs_1\bs_2}$
in $\wh{F}_{\bs}$ over $x_1^{\ell_1+\ell_2}$ in~$F^{\ell_1+\ell_2}$,
 where $\bs$ is the merged tuple of~$\bs_1$ and~$\bs_2$ as before,
and a collection  $\{\ga_j\}\!\subset\!H_1(F;\Z)$ of 
representatives for the elements~of
$$\frac{H_1(V;\Z)_X}{\cR_{X;\bs_1\bs_2}'^{\,V}}=\frac{H_1(F;\Z)}{\gcd(\bs_1,\bs_2)H_1(F;\Z)}\,.$$
If $F\!=\!\T^2$ and $(\ell_1,\ell_2)\!\neq\!\0$, then
we can identify $H_1(F;\Z)$ with $\Z\!\oplus\!\fI\Z$.
As indicated in \cite[Example~3.3]{GWsumIP}, it is then convenient to identify 
the corresponding base points for the components of~$\wh{V}_{X;\bs}$~as
$$\big([\ga_j]_{X;\bs_1\bs_2},[\wh{x}_{\bs_1\bs_2}]_X\big)
=\big(\ell^{-1}\ga_j,\big([\ell^{-1}s_{1;i}^{-1}\ga_j]_{i\le\ell_1},
[-\ell^{-1}s_{2;i}^{-1}\ga_j]_{i\le\ell_2}\big)\big)\in 
\C\!\times\!\T_{(\bs_1,-\bs_2)}^{2(\ell_1+\ell_2-1)}.$$
\end{eg}

\subsection{Proof of Theorem~\ref{RelGW_thm}}
\label{RelGWThm_subs}

Let $n_V\!=\!\dim_{\R}V$ and $r_0\!=\!\rk_{\Z}H_1(V;\Z)_X$.
In light of~\eref{evfactor_e} and~\eref{evXVlift_e}, it is sufficient to show that 
\BE{cHsplit_e} \wh{V}_{X;\bs}'\approx \R^{r_0}\!\times\!Y\EE
for some manifold $Y$ of dimension $n_V\ell\!-\!r_0$.\\

\noindent
With $B\!=\!(S^1)^m$ and $H_X^V$ as in~\eref{H1VXdfn_e}, let
$$H_X^B=q_*(H_X^V)\subset H_1(B;\Z), \qquad  
\wt{H}_X^B=q_*^{\,-1}(H_X^B)\subset H_1(V;\Z).$$
Denote by $\cK_{X;\bs}^B\!\subset\!\pi_1(B^{\ell})$ and 
$\cK_{X;\bs}^V,\wt\cK_{X;\bs}^B\!\subset\!\pi_1(V^{\ell})$
the preimages of $H_X^B$ and $H_X^V,\wt{H}_X^B$
under the homomorphisms
$$\pi_1(B^{\ell})\stackrel{\Hur}{\lra} H_1(B^{\ell};\Z)
\stackrel{\Phi_{B;\bs}}\lra H_1(B;\Z) 
\quad\hbox{and}\quad
\pi_1(V^{\ell})\stackrel{\Hur}{\lra} H_1(V^{\ell};\Z)
\stackrel{\Phi_{V;\bs}}\lra H_1(V;\Z),$$
with $\Phi_{B;\bs}$ and $\Phi_{V;\bs}$ as in~\eref{PhiVbs_e}.
From the commutativity of the diagram
$$\xymatrix{ \pi_1(V^{\ell}) \ar[rr]^{\Phi_{V;\bs}\circ\Hur} \ar[d]_{q_*}&& H_1(V;\Z)  \ar[d]^{q_*}\\
\pi_1(B^{\ell})  \ar[rr]^{\Phi_{B;\bs}\circ\Hur} && H_1(B;\Z)\,,}$$
we find that $\wt\cK_{X;\bs}^B\!=\!q_*^{-1}(\cK_{X;\bs}^B)$.\\

\noindent 
Let 
$$\pi_{X;\bs}'^B\!:\wh{B}_{X;\bs}'\!\equiv\!\wh{B}_{H_X^B;\bs}' \lra B^{\ell}
\qquad\hbox{and}\qquad
\ti{q}\!: \wt{V}_{X;\bs}^B\!\equiv\!\pi_{X;\bs}'^{B\,*}V^{\ell}\lra \wh{B}_{X;\bs}' $$
be the covering corresponding to the normal subgroup $\cK_{X;\bs}^B$ of $\pi_1(B^{\ell})$ 
and  the pullback of the fibration $q^{\ell}\!:V^{\ell}\!\lra\!B^{\ell}$, respectively.
Since $\ti{q}_*$ is surjective on~$\pi_1$, the natural projection
$$\wt\pi_{X;\bs}'^B\!:\wt{V}_{X;\bs}^B\lra V^{\ell} $$
is the covering corresponding to the normal subgroup $q_*^{-1}(\cK_{X;\bs}^B)\!=\!\wt\cK_{X;\bs}^B$
of $\pi_1(V^{\ell})$.
Since \hbox{$\cK_{X;\bs}^V\!\subset\!\wt\cK_{X;\bs}^B$}, 
the total space of the covering
$$\pi_{X;\bs}'^V\!:\, \wh{V}_{X;\bs}'\lra V_{\bs}\!=\!V^{\ell}$$
corresponding to $\cK_{X;\bs}^V$ is also a covering of~$\wt{V}_{X;\bs}^B$.\\

\noindent
By the homotopy exact sequence 
\cite[Theorem~7.2.10]{Sp} for the fibration $q\!:V\!\lra\!B$
and the Hurewicz isomorphism \cite[Theorem~7.5.5]{Sp}, the sequence
$$H_1(F;\Z)\lra H_1(V;\Z) \stackrel{q_*}{\lra} H_1(B;\Z)\lra0$$
is exact. 
Since $H_1(F;\Q)\!=\!\{0\}$,
$$\rk_{\Z}\big(H_1(B;\Z)/H_X^B\big)= \rk_{\Z}H_1(V;\Z)_X = r_0.$$
Similarly to Example~\ref{Tcov_eg}, this implies~that 
$$ \wh{B}_{X;\bs}'\approx \R^{r_0}\!\times\!(S^1)^{m\ell-r_0}\,.$$
Thus, $\wt{V}_{X;\bs}^B\!\approx\!\R^{r_0}\!\times\!Y'$
for some manifold $Y'$ of dimension $n_V\ell\!-\!r_0$
and so \eref{cHsplit_e} holds for some covering~$Y$ of~$Y'$.

\begin{rmk}\label{RelGW_rmk}
Theorem~\ref{RelGW_thm} extends to disconnected divisors $V$ 
with topological components $V_1,\ldots,V_N$ by replacing 
the rank of~$H_1(V;\Z)_X$ with the rank~of the submodule in~\eref{cRHr_e2} with $H\!=\!H_X^V$.
\end{rmk}

\vspace{.3in}

\noindent
{\it Simons Center for Geometry and Physics, SUNY Stony Brook, NY 11794\\
mtehrani@scgp.stonybrook.edu}\\

\noindent
{\it Department of Mathematics, SUNY Stony Brook, Stony Brook, NY 11794\\
azinger@math.sunysb.edu}\\

\end{document}